\documentclass[10pt,a4paper,reqno]{amsart}

\usepackage{lmodern}
\usepackage[T1]{fontenc}
\usepackage[utf8]{inputenc}
\usepackage[english]{babel}
\usepackage{microtype} 
\usepackage[yyyymmdd,hhmmss]{datetime}
\usepackage{float}

\usepackage{slashed,url,bm,amssymb,mathrsfs,mathtools,todonotes,xparse,booktabs,graphicx}
\usepackage[centertableaux]{ytableau}
\makeatletter
\edef\boxframe@normal@YT{\boxframe@YT}
\usepackage[pdftex]{hyperref}
\usepackage{float,placeins}

\usepackage{blindtext} 
\usepackage{epigraph} 

\usepackage{array}
\usepackage[most]{tcolorbox}
\newtcolorbox{editbox}{colback=yellow!10, colframe=red!60!black, title={\small\bfseries Edit}, fonttitle=\sffamily, breakable}

\usepackage{fancyhdr}
\usepackage[font={small,sl}]{caption}
\newtheorem{theorem}{Theorem}[section]

\newtheorem{lemma}[theorem]{Lemma}
\newtheorem{definition}[theorem]{Definition}
\newtheorem{proposition}[theorem]{Proposition}
\newtheorem{corollary}[theorem]{Corollary}
\newtheorem{conjecture}[theorem]{Conjecture}
\newtheorem{observation}[theorem]{Observation}
\newtheorem{example}[theorem]{Example}

\definecolor{pBlue}{RGB}{86,139,190}
\definecolor{pCyan}{RGB}{149,186,201}
\definecolor{pSand}{RGB}{184,166,121}
\definecolor{pAlgae}{RGB}{87,115,135}
\definecolor{pSkin}{RGB}{236,216,167}
\definecolor{pGray}{RGB}{156,175,156}
\definecolor{pPink}{RGB}{215,114,127}
\definecolor{pOrange}{RGB}{211,153,80}
\definecolor{pDarkGreen}{RGB}{43,140,51}

\definecolor{boxL}{RGB}{153,211,80}
\definecolor{boxU}{RGB}{80,153,211}
\definecolor{boxS}{RGB}{236,216,167}
\definecolor{boxG}{RGB}{240,240,240}
\definecolor{boxB}{RGB}{0,0,0}
\definecolor{boxGr}{RGB}{205,205,205}

\newcommand{\setZ}{\mathbb{Z}}

\newcommand{\jackjLR}{g}
\newcommand{\JackjLR}{{\bm{g}}}

\newcommand{\addset}{\mathcal{O}}
\newcommand{\remset}{\mathcal{I}}

\newcommand{\bmD}{{\bm {D}}}

\newcommand{\bmA}{{\bm {A}}}
\newcommand{\bmB}{{\bm {B}}}
\newcommand{\bmX}{{\bm {X}}}

\newcommand{\bmss}{{\bm {s}}}
\newcommand{\bmf}{{\bm {f}}}
\newcommand{\bbeta}{{\bf{\beta}}}

\newcommand{\al}{\alpha}

\newcommand{\simpterm}{{\bm {E}}}

\renewcommand{\*}{\mathbin{\star}}

\newcommand{\bigstanleysum}{{ \bm{g}_{\star}  }}
\newcommand{\stanleypoly}{{ {\bm G} }}
\newcommand{\bigstanleypoly}{{ {\stanleypoly}_{\star} }}

\newcommand{\hkAut}{{ \mathrm{hkAut}}}

\newcommand{\bh}[0]{{\bm{h}}}

\newcommand{\bmk}{{\bm{k}}}

\newcommand{\sU}{{\scriptscriptstyle{\mathrm U}}}
\newcommand{\sL}{{\scriptscriptstyle{\mathrm L}}}
\newcommand{\sA}{{\scriptscriptstyle{\mathrm A}}}
\newcommand{\sB}{{\scriptscriptstyle{\mathrm B}}}

\DeclareMathSymbol{\shortminus}{\mathbin}{AMSa}{"39}

\newcommand{\BZ}{\mathbb{Z}}
\newcommand{\BQ}{{\mathbb{Q}}}
\newcommand{\CO}{\mathscr{O}}

\newcommand{\StD}{\mathsf{SD}}

\newcommand{\StS}{\mathsf{SS}}
\newcommand{\StDR}{\mathsf{R}}

\newcommand{\petvertrep} {{ V_{\mathrm{P}} }}
\newcommand{\johnvertrep} {{ V_{\mathrm{J}} }}
\newcommand{\petersen} {{ \mathrm{P} }}
\newcommand{\johnson} {{ \mathrm{J} }}

\newcommand{\windfact}{\bm{F}}

\newcommand{\claw}{{\mathrm{cl}}}

\newcommand{\sgn}{{\mathrm {sgn}}}
\newcommand{\id}{{\mathrm {id}}}
\newcommand{\Sym}{{\mathrm {Sym}}}
\newcommand{\Aut}{{\mathrm {Aut}}}

\newcommand{\windowfamily}{{\mathcal{W}}}

\newcommand{\BC}{{\mathbb{C}}}

\newcommand{\allvertices}{{ \mathrm{H} }}

\newcommand{\orbitsum}{{ \mathcal{T} }}

\tikzset{
  graphlet/.style={baseline=-0.6ex, x=1.2em, y=1.2em, line width=0.35pt},
  vtx/.style={circle, fill=black, inner sep=1.2pt},
  ghostedge/.style={draw=black, opacity=0.25, line width=0.35pt},
  ghost/.style={circle, draw=black, draw opacity=0.25, fill=black, fill opacity=0.10, inner sep=1.2pt},
  edge/.style={line width=0.9pt},
  fivegraph/.style={baseline=-0.6ex, x=1.2em, y=1.2em},
}
\usetikzlibrary{calc}
\usetikzlibrary{cd}

\newcommand{\gPathTwo}{%
  \tikz[graphlet]{%
    \node[vtx] (a) at (0,0) {};
    \node[vtx] (b) at (0.5,0) {};
    \draw[edge] (a)--(b);
  }%
}

\newcommand{\gIsoTwo}{%
  \tikz[graphlet]{%
    \node[vtx] (a) at (0,0) {};
    \node[ghost] (b) at (0.5,0) {};
    \node[vtx] (c) at (1.0,0) {};
    \draw[ghostedge] (a)--(b)--(c);
  }%
}

\newcommand{\gEdgeIsoThree}{%
  \tikz[fivegraph]{%
    \node[vtx] (a) at (0,0) {};
    \node[vtx] (b) at (0.5,0) {};
    \node[vtx] (c) at (1.0,0) {};
    \draw[edge] (a)--(b);
  }%
}

\newcommand{\gEdgeIsoThreeGhost}{%
   \tikz[fivegraph,scale=0.6]{%
    \node[vtx] (a) at (90:0.9) {};
    \node[ghost] (b) at (18:0.9) {};
    \node[vtx] (c) at (-54:0.9) {};
    \node[vtx] (d) at (-126:0.9) {};
    \node[ghost] (e) at (162:0.9) {};
    \draw[thick] (c)--(d);
    \draw[ghostedge] (d)--(e)--(a)--(b)--(c);
  }%
}

\newcommand{\gPathThree}{%
  \tikz[graphlet]{%
    \node[vtx] (a) at (0,0) {};
    \node[vtx] (b) at (0.5,0) {};
    \node[vtx] (c) at (1.0,0) {};
    \draw[edge] (a)--(b)--(c);
  }%
}

\newcommand{\gStarGhostThree}{%
  \tikz[graphlet]{%
    \node[vtx]   (x1) at (-0.6,0.3) {};
    \node[vtx]   (x2) at ( 0.0,0.3) {};
    \node[vtx]   (x3) at ( 0.6,0.3) {};
    \node[ghost] (c)  at ( 0.0,-0.3) {};
    \draw[ghostedge] (c)--(x1) (c)--(x2) (c)--(x3);
  }%
}

\newcommand{\gthreeIsoGhost}{%
  \tikz[graphlet]{%
    \node[vtx]   (x1) at (-0.6,0.3) {};
    \node[vtx]   (x2) at ( 0.0,0.3) {};
    \node[vtx]   (x3) at ( 0.6,0.3) {};
    \node[ghost] (c)  at ( -0.3,-0.3) {};
    \node[ghost] (d)  at ( 0.3,-0.3) {};
    \draw[ghostedge] (c)--(x1) (c)--(x2) (c)--(d) (d)--(x3);
  }%
}

\newcommand{\gFourStar}{%
  \tikz[graphlet]{%
    \node[vtx]   (x1) at (-0.6,0.3) {};
    \node[vtx]   (x2) at ( 0.0,0.3) {};
    \node[vtx]   (x3) at ( 0.6,0.3) {};
    \node[vtx]   (c)  at ( 0.0,-0.3) {};
    \draw[edge] (c)--(x1) (c)--(x2) (c)--(x3);
  }%
}

\newcommand{\gPFour}{%
  \tikz[fivegraph,scale=0.7]{%
    \node[ghost] (a) at (90:0.9) {};
    \node[vtx] (b) at (18:0.9) {};
    \node[vtx] (c) at (-54:0.9) {};
    \node[vtx] (d) at (-126:0.9) {};
    \node[vtx] (e) at (162:0.9) {};
    \draw[edge] (b)--(c)--(d)--(e);
    \draw[ghostedge] (a)--(b) (a)--(e);
  }%
}

\newcommand{\gPthreeIso}{%
  \tikz[fivegraph]{%
    \node[vtx] (c)  at (0,0) {};
    \node[vtx] (l1) at (-0.7,0.4) {};
    \node[vtx] (l2) at (-0.7,-0.4) {};
    \node[ghost] (m)  at (0.7,0.0) {};
    \node[vtx] (l3) at (1.4,0.4) {};
    \draw[edge] (c)--(l1) (c)--(l2);
    \draw[ghostedge] (m)--(l3)  (c)--(m);
  }%
}

\newcommand{\gPtwoIsotwo}{%
  \tikz[fivegraph]{%
    \node[ghost] (c)  at (0,0) {};
    \node[vtx] (l1) at (-0.7,0.4) {};
    \node[vtx] (l2) at (-0.7,-0.4) {};
    \node[vtx] (m)  at (0.7,0.0) {};
    \node[vtx] (l3) at (1.4,0.4) {};
    \draw[edge] (m)--(l3)  ;
    \draw[ghostedge] (c)--(l1) (c)--(l2) (c)--(m);
  }%
}

\newcommand{\gtwotwo}{%
  \tikz[fivegraph,scale=0.7]{%
    \node[vtx] (a) at (0:0.9) {};
    \node[vtx] (b) at (60:0.9) {};
    \node[ghost] (c) at (120:0.9) {};
    \node[vtx] (d) at (180:0.9) {};
    \node[vtx] (e) at (240:0.9) {};
    \node[ghost] (f) at (300:0.9) {};
    \draw[edge] (a)--(b) (d)--(e);
     \draw[ghostedge] (a)--(f)--(e) (b)--(c)--(d);
  }%
}

\newcommand{\gFourIso}{%
  \tikz[fivegraph]{%
    \node[ghost] (c)  at (0,0) {};
    \node[vtx] (l1) at (-0.7,0.4) {};
    \node[vtx] (l2) at (-0.7,-0.4) {};
    \node[ghost] (m)  at (0.7,0.0) {};
    \node[vtx] (l3) at (1.4,0.4) {};
    \node[vtx] (g) at (1.4,-0.4) {};
    \draw[edge] ;
    \draw[ghostedge] (m)--(g)  (m)--(l3)  (c)--(m) (c)--(l1) (c)--(l2);
  }%
}

\newcommand{\gFiveCycle}{%
  \tikz[fivegraph,scale=0.7]{%
    \node[vtx] (a) at (90:0.9) {};
    \node[vtx] (b) at (18:0.9) {};
    \node[vtx] (c) at (-54:0.9) {};
    \node[vtx] (d) at (-126:0.9) {};
    \node[vtx] (e) at (162:0.9) {};
    \draw[edge] (a)--(b)--(c)--(d)--(e)--(a);
  }%
}

\newcommand{\gFivePath}{%
  \tikz[fivegraph,scale=0.7]{%
    \node[vtx] (a) at (0:0.9) {};
    \node[vtx] (b) at (60:0.9) {};
    \node[vtx] (c) at (120:0.9) {};
    \node[vtx] (d) at (180:0.9) {};
    \node[vtx] (e) at (240:0.9) {};
    \node[ghost] (f) at (300:0.9) {};
    \draw[edge] (a)--(b)--(c)--(d)--(e);
     \draw[ghostedge] (f)--(a) (f)--(e);
  }%
}

\newcommand{\gFiveTreeA}{%
  \tikz[fivegraph]{%
    \node[vtx] (c)  at (0,0) {};
    \node[vtx] (l1) at (-0.7,0.4) {};
    \node[vtx] (l2) at (-0.7,-0.4) {};
    \node[vtx] (m)  at (0.7,0.0) {};
    \node[vtx] (l3) at (1.4,0.4) {};
    \node[ghost] (g) at (1.4,-0.4) {};
    \draw[edge] (c)--(l1) (c)--(l2) (c)--(m) (m)--(l3);
    \draw[ghostedge] (m)--(g);
  }%
}

\newcommand{\gFivePFourPlusIso}{%
  \tikz[fivegraph,scale=0.7]{%
    \node[ghost] (a) at (90:0.9) {};
    \node[vtx] (b) at (18:0.9) {};
    \node[vtx] (c) at (-54:0.9) {};
    \node[vtx] (d) at (-126:0.9) {};
    \node[vtx] (e) at (162:0.9) {};
    \node[vtx] (f) at (0,0) {};
    \draw[edge] (b)--(c)--(d)--(e);
    \draw[ghostedge] (f)--(a) (a)--(b) (a)--(e);
  }%
}

\newcommand{\gPthreeTwoIso}{%
  \tikz[fivegraph]{%
    \node[vtx] (c)  at (0,0) {};
    \node[vtx] (l1) at (-0.7,0.4) {};
    \node[vtx] (l2) at (-0.7,-0.4) {};
    \node[ghost] (m)  at (0.7,0.0) {};
    \node[vtx] (l3) at (1.4,0.4) {};
    \node[vtx] (g) at (1.4,-0.4) {};
    \draw[edge] (c)--(l1) (c)--(l2);
    \draw[ghostedge] (m)--(g)  (m)--(l3)  (c)--(m);
  }%
}

\newcommand{\gTwoEdgesIso}{%
  \tikz[fivegraph,scale=0.7]{%
    \node[ghost] (a) at (0:0.9) {};
    \node[vtx] (b) at (60:0.9) {};
    \node[vtx] (c) at (120:0.9) {};
    \node[ghost] (d) at (180:0.9) {};
    \node[vtx] (e) at (240:0.9) {};
    \node[vtx] (f) at (300:0.9) {};
    \node[vtx] (t) at (0,0) {};
    \draw[edge] (b)--(c) (f)--(e);
    \draw[ghostedge] (a)--(b) (c)--(d) (e)--(d)  (f)--(a) (t)--(a) (t)--(d);
  }%
}

\newcommand{\gsixO}{%
  \tikz[fivegraph,scale=0.7]{%
    \node[vtx] (a) at (0:0.9) {};
    \node[vtx] (b) at (60:0.9) {};
    \node[vtx] (c) at (120:0.9) {};
    \node[vtx] (d) at (180:0.9) {};
    \node[vtx] (e) at (240:0.9) {};
    \node[vtx] (f) at (300:0.9) {};
    \draw[edge] (b)--(c) (f)--(e) (c)--(d) (e)--(d) (a)--(b)    (f)--(a)  ;
    \draw[ghostedge]  ;
  }%
}
\newcommand{\gsixXi}{%
  \tikz[fivegraph,scale=0.7]{%
    \node[vtx] (a) at (0:0.45) {};
    \node[vtx] (b) at (60:0.9) {};
    \node[vtx] (c) at (120:0.9) {};
    \node[vtx] (d) at (180:0.45) {};
    \node[vtx] (e) at (240:0.9) {};
    \node[vtx] (f) at (300:0.9) {};
    \draw[edge] (b)--(c) (f)--(e)    (a)--(d)  ;
    \draw[ghostedge]  ;
  }%
}
\newcommand{\gsixH}{%
  \tikz[fivegraph,scale=0.7]{%
    \node[vtx] (a) at (0:0.45) {};
    \node[vtx] (b) at (60:0.9) {};
    \node[vtx] (c) at (120:0.9) {};
    \node[vtx] (d) at (180:0.45) {};
    \node[vtx] (e) at (240:0.9) {};
    \node[vtx] (f) at (300:0.9) {};
    \draw[edge] (c)--(d)--(e) (f)--(a)--(b) (d)--(a)  ;
    \draw[ghostedge]  ;
  }%
}
\newcommand{\gsixE}{%
  \tikz[fivegraph,scale=0.7]{%
    \node[ghost] (a) at (0:0.9) {};
    \node[vtx] (b) at (60:0.9) {};
    \node[vtx] (c) at (120:0.9) {};
    \node[vtx] (d) at (180:0.9) {};
    \node[vtx] (e) at (240:0.9) {};
    \node[vtx] (f) at (300:0.9) {};
    \node[vtx] (t) at (0,0) {};
    \draw[edge] (b)--(c) (f)--(e) (c)--(d) (e)--(d) (t)--(d);
    \draw[ghostedge] (a)--(b)    (f)--(a) (t)--(a) ;
  }%
}
\newcommand{\gsixC}{%
  \tikz[fivegraph,scale=0.7]{%
    \node[ghost] (a) at (0:0.9) {};
    \node[vtx] (b) at (60:0.9) {};
    \node[vtx] (c) at (120:0.9) {};
    \node[vtx] (d) at (180:0.9) {};
    \node[vtx] (e) at (240:0.9) {};
    \node[vtx] (f) at (300:0.9) {};
    \node[vtx] (t) at (0,0) {};
    \draw[edge] (b)--(c) (f)--(e) (c)--(d) (e)--(d) ;
    \draw[ghostedge] (a)--(b)  (f)--(a) ;
  }%
}
\newcommand{\gsixQ}{%
  \tikz[fivegraph,scale=0.7]{%
    \node[vtx] (a) at (90:0.9) {};
    \node[vtx] (b) at (18:0.9) {};
    \node[vtx] (c) at (-54:0.9) {};
    \node[vtx] (d) at (-126:0.9) {};
    \node[vtx] (e) at (162:0.9) {};
    \node[vtx] (f) at (0,0) {};
    \draw[edge] (b)--(c)--(d)--(e)  (f)--(a) (a)--(b) (a)--(e);
    \draw[ghostedge] ;
  }%
}

\newcommand{\gsixOa}{
\begin{tikzpicture}[scale=0.4, every node/.style={circle, draw, inner sep=1.2pt}, baseline={([yshift=-1.0ex]current bounding box.center)}
]
  \node (o1) at ( 90:1.6) {$$};
  \node[very thick, fill] (o2) at ( 18:1.6) {$$};
  \node[red, very thick, fill] (o3) at (-54:1.6) {$$};
  \node (o4) at (-126:1.6) {$$};
  \node (o5) at (162:1.6) {$$};
  \node[red, very thick, fill] (i1) at ( 90:0.9) {$$};
  \node[red, very thick, fill] (i2) at ( 18:0.9) {$$};
  \node[very thick, fill] (i3) at (-54:0.9) {$$};
  \node[very thick, fill] (i4) at (-126:0.9) {$$};
  \node (i5) at (162:0.9) {$$};
  \draw (o1)--(o2);
  \draw[very thick, fill] (o2)--(o3);
  \draw (o5)--(o1);
  \draw (o3)--(o4);
  \draw (o4)--(o5);
  \draw (o1)--(i1);
  \draw[very thick, fill] (o2)--(i2);
  \draw[very thick, fill] (o3)--(i3);
  \draw (o4)--(i4);
  \draw (o5)--(i5);
  \draw[very thick, fill] (i1)--(i3);
  \draw (i3)--(i5);
  \draw (i5)--(i2);
  \draw[very thick, fill] (i2)--(i4);
  \draw[very thick, fill] (i4)--(i1);
\end{tikzpicture}
}

\newcommand{\gsixOi}{
\begin{tikzpicture}[scale=0.4, every node/.style={circle, draw, inner sep=1.2pt}, baseline={([yshift=-1.0ex]current bounding box.center)}
]
  \node (o1) at (18:1.6) {$$};
  \node[red, very thick, fill] (o2) at (-54:1.6) {$$};
  \node[red, very thick, fill] (o3) at (-126:1.6) {$$};
  \node (o4) at (162:1.6) {$$};
  \node (o5) at (90:1.6) {$$};
  \node[red, very thick, fill] (i1) at ( 18:0.9) {$$};
  \node[very thick, fill] (i2) at (-54:0.9) {$$};
  \node[very thick, fill] (i3) at (-126:0.9) {$$};
  \node[red, very thick, fill] (i4) at (162:0.9) {$$};
  \node (i5) at ( 90:0.9) {$$};
  \draw (o1)--(o2);
  \draw[red, very thick, fill] (o2)--(o3);
  \draw (o5)--(o1);
  \draw (o3)--(o4);
  \draw (o4)--(o5);
  \draw (o1)--(i1);
  \draw[very thick, fill] (o2)--(i2);
  \draw[very thick, fill] (o3)--(i3);
  \draw (o4)--(i4);
  \draw (o5)--(i5);
  \draw[very thick, fill] (i1)--(i3);
  \draw (i3)--(i5);
  \draw (i5)--(i2);
  \draw[ very thick, fill] (i2)--(i4);
  \draw[red, very thick, fill] (i4)--(i1);
\end{tikzpicture}
}

\newcommand{\gsixXIa}{
\begin{tikzpicture}[scale=0.4, every node/.style={circle, draw, inner sep=1.2pt}, baseline={([yshift=-1.0ex]current bounding box.center)}]
  \node[red, very thick, fill] (o1) at ( 90:1.6) {$$};
  \node[very thick, fill] (o2) at ( 18:1.6) {$$};
  \node (o3) at (-54:1.6) {$$};
  \node[red, very thick, fill] (o4) at (-126:1.6) {$$};
  \node (o5) at (162:1.6) {$$};
  \node (i1) at ( 90:0.9) {$$};
  \node (i2) at ( 18:0.9) {$$};
  \node[very thick, fill] (i3) at (-54:0.9) {$$};
  \node[very thick, fill] (i4) at (-126:0.9) {$$};
  \node[red, very thick, fill] (i5) at (162:0.9) {$$};
  \draw[ very thick, fill] (o1)--(o2);
  \draw (o2)--(o3);
  \draw (o5)--(o1);
  \draw (o3)--(o4);
  \draw (o4)--(o5);
  \draw (o1)--(i1);
  \draw (o2)--(i2);
  \draw (o3)--(i3);
  \draw[very thick, fill] (o4)--(i4);
  \draw (o5)--(i5);
  \draw (i1)--(i3);
  \draw[very thick, fill] (i3)--(i5);
  \draw (i5)--(i2);
  \draw (i2)--(i4);
  \draw (i4)--(i1);
\end{tikzpicture}
}

\newcommand{\gsixXIi}{
\begin{tikzpicture}[scale=0.4, every node/.style={circle, draw, inner sep=1.2pt}, baseline={([yshift=-1.0ex]current bounding box.center)}]
  \node[red, very thick, fill] (o1) at (-54 :1.6) {$$};
  \node[red, very thick, fill] (o2) at (-126 :1.6) {$$};
  \node (o3) at (162:1.6) {$$};
  \node[red, very thick, fill] (o4) at (90:1.6) {$$};
  \node (o5) at (18:1.6) {$$};
  \node (i1) at ( -54:0.9) {$$};
  \node (i2) at ( -126:0.9) {$$};
  \node[red, very thick, fill] (i3) at (162:0.9) {$$};
  \node[red, very thick, fill] (i4) at ( 90 :0.9) {$$};
  \node[red, very thick, fill] (i5) at (18:0.9) {$$};
  \draw[red, very thick, fill] (o1)--(o2);
  \draw (o2)--(o3);
  \draw (o5)--(o1);
  \draw (o3)--(o4);
  \draw (o4)--(o5);
  \draw (o1)--(i1);
  \draw (o2)--(i2);
  \draw (o3)--(i3);
  \draw[red, very thick, fill] (o4)--(i4);
  \draw (o5)--(i5);
  \draw (i1)--(i3);
  \draw[red, very thick, fill] (i3)--(i5);
  \draw (i5)--(i2);
  \draw (i2)--(i4);
  \draw (i4)--(i1);
\end{tikzpicture}
}

\newcommand{\gsixEi }{
 \begin{tikzpicture}[scale=0.4, every node/.style={circle, draw, inner sep=1.2pt}, baseline={([yshift=-1.0ex]current bounding box.center)}
]
  \node (o1) at ( 90:1.6) {$$};
  \node[very thick, fill] (o2) at ( 18:1.6) {$$};
  \node[red, very thick, fill] (o3) at (-54:1.6) {$$};
  \node (o4) at (-126:1.6) {$$};
  \node (o5) at (162:1.6) {$$};
  \node[red, very thick, fill] (i1) at ( 90:0.9) {$$};
  \node (i2) at ( 18:0.9) {$$};
  \node[very thick, fill] (i3) at (-54:0.9) {$$};
  \node[very thick, fill] (i4) at (-126:0.9) {$$};
  \node[red, very thick, fill] (i5) at (162:0.9) {$$};
  \draw (o1)--(o2);
  \draw[very thick, fill] (o2)--(o3);
  \draw (o3)--(o4);
  \draw (o4)--(o5);
  \draw (o5)--(o1);
  \draw (o1)--(i1);
  \draw (o2)--(i2);
  \draw[very thick, fill] (o3)--(i3);
  \draw (o4)--(i4);
  \draw (o5)--(i5);
  \draw[very thick, fill] (i1)--(i3);
  \draw[very thick, fill] (i3)--(i5);
  \draw (i5)--(i2);
  \draw (i2)--(i4);
  \draw[very thick, fill] (i4)--(i1);
\end{tikzpicture} 
}

\newcommand{\gsixCione}{
\begin{tikzpicture}[scale=0.4, every node/.style={circle, draw, inner sep=1.2pt}, baseline={([yshift=-1.0ex]current bounding box.center)}]
  \node[red, very thick, fill] (o1) at ( 90:1.6) {$$};
  \node[very thick, fill] (o2) at ( 18:1.6) {$$};
  \node (o3) at (-54:1.6) {$$};
  \node[red, very thick, fill] (o4) at (-126:1.6) {$$};
  \node (o5) at (162:1.6) {$$};
  \node (i1) at ( 90:0.9) {$$};
  \node[red, very thick, fill] (i2) at ( 18:0.9) {$$};
  \node[very thick, fill] (i3) at (-54:0.9) {$$};
  \node[very thick, fill] (i4) at (-126:0.9) {$$};
  \node (i5) at (162:0.9) {$$};
  \draw[very thick, fill] (o1)--(o2);
  \draw (o2)--(o3);
  \draw (o3)--(o4);
  \draw (o4)--(o5);
  \draw (o5)--(o1);
  \draw (o1)--(i1);
  \draw[very thick, fill] (o2)--(i2);
  \draw (o3)--(i3);
  \draw[very thick, fill] (o4)--(i4);
  \draw (o5)--(i5);
  \draw (i1)--(i3);
  \draw (i3)--(i5);
  \draw (i5)--(i2);
  \draw[very thick, fill] (i2)--(i4);
  \draw (i4)--(i1);
\end{tikzpicture}
}

\newcommand{\gsixEii}{
\begin{tikzpicture}[scale=0.4, every node/.style={circle, draw, inner sep=1.2pt}, baseline={([yshift=-1.0ex]current bounding box.center)}
]
  \node[red, very thick, fill] (o1) at ( 90:1.6) {$$};
  \node (o2) at ( 18:1.6) {$$};
  \node[red, very thick, fill] (o3) at (-54:1.6) {$$};
  \node[red, very thick, fill] (o4) at (-126:1.6) {$$};
  \node (o5) at (162:1.6) {$$};
  \node[red, very thick, fill] (i1) at ( 90:0.9) {$$};
  \node[very thick, fill] (i2) at ( 18:0.9) {$$};
  \node (i3) at (-54:0.9) {$$};
  \node[very thick, fill] (i4) at (-126:0.9) {$$};
  \node (i5) at (162:0.9) {$$};
  \draw (o1)--(o2);
  \draw (o2)--(o3);
  \draw[red, very thick, fill] (o3)--(o4);
  \draw (o4)--(o5);
  \draw (o5)--(o1);
  \draw[red, very thick, fill] (o1)--(i1);
  \draw (o2)--(i2);
  \draw (o3)--(i3);
  \draw[very thick, fill] (o4)--(i4);
  \draw (o5)--(i5);
  \draw (i1)--(i3);
  \draw (i3)--(i5);
  \draw (i5)--(i2);
  \draw[very thick, fill] (i2)--(i4);
  \draw[very thick, fill] (i4)--(i1);
\end{tikzpicture} 
}

\newcommand{\gsixCii}{
\begin{tikzpicture}[scale=0.4, every node/.style={circle, draw, inner sep=1.2pt}, baseline={([yshift=-1.0ex]current bounding box.center)}]
  \node[red, very thick, fill] (o1) at ( 90:1.6) {$$};
  \node (o2) at ( 18:1.6) {$$};
  \node[red, very thick, fill] (o3) at (-54:1.6) {$$};
  \node[red, very thick, fill] (o4) at (-126:1.6) {$$};
  \node (o5) at (162:1.6) {$$};
  \node[red, very thick, fill] (i1) at ( 90:0.9) {$$};
  \node (i2) at ( 18:0.9) {$$};
  \node (i3) at (-54:0.9) {$$};
  \node[very thick, fill] (i4) at (-126:0.9) {$$};
  \node[red, very thick, fill] (i5) at (162:0.9) {$$};
  \draw (o1)--(o2);
  \draw (o2)--(o3);
  \draw[red, very thick, fill] (o3)--(o4);
  \draw (o4)--(o5);
  \draw (o5)--(o1);
  \draw[red, very thick, fill] (o1)--(i1);
  \draw (o2)--(i2);
  \draw (o3)--(i3);
  \draw[very thick, fill] (o4)--(i4);
  \draw (o5)--(i5);
  \draw (i1)--(i3);
  \draw (i3)--(i5);
  \draw (i5)--(i2);
  \draw (i2)--(i4);
  \draw[very thick, fill] (i4)--(i1);
\end{tikzpicture}
}

\newcommand{\gsixCaii}{
\begin{tikzpicture}[scale=0.4, every node/.style={circle, draw, inner sep=1.2pt}, baseline={([yshift=-1.0ex]current bounding box.center)}]
  \node[red, very thick, fill] (o1) at ( 90:1.6) {$$};
  \node (o2) at ( 18:1.6) {$$};
  \node (o3) at (-54:1.6) {$$};
  \node[red, very thick, fill] (o4) at (-126:1.6) {$$};
  \node[very thick, fill] (o5) at (162:1.6) {$$};
  \node[red, very thick, fill] (i1) at ( 90:0.9) {$$};
  \node[red, very thick, fill] (i2) at ( 18:0.9) {$$};
  \node[very thick, fill] (i3) at (-54:0.9) {$$};
  \node (i4) at (-126:0.9) {$$};
  \node (i5) at (162:0.9) {$$};
  \draw (o1)--(o2);
  \draw (o2)--(o3);
  \draw (o3)--(o4);
  \draw[very thick, fill] (o4)--(o5);
  \draw[very thick, fill] (o5)--(o1);
  \draw[red, very thick, fill] (o1)--(i1);
  \draw (o2)--(i2);
  \draw (o3)--(i3);
  \draw (o4)--(i4);
  \draw (o5)--(i5);
  \draw[very thick, fill] (i1)--(i3);
  \draw (i3)--(i5);
  \draw (i5)--(i2);
  \draw (i2)--(i4);
  \draw (i4)--(i1);
\end{tikzpicture}
}

\newcommand{\gsixCiii}{
\begin{tikzpicture}[scale=0.4, every node/.style={circle, draw, inner sep=1.2pt}, baseline={([yshift=-1.0ex]current bounding box.center)}]
  \node[red, very thick, fill] (o1) at ( 90:1.6) {$$};
  \node[very thick, fill] (o2) at ( 18:1.6) {$$};
  \node (o3) at (-54:1.6) {$$};
  \node (o4) at (-126:1.6) {$$};
  \node[very thick, fill] (o5) at (162:1.6) {$$};
  \node (i1) at ( 90:0.9) {$$};
  \node[red, very thick, fill] (i2) at ( 18:0.9) {$$};
  \node[very thick, fill] (i3) at (-54:0.9) {$$};
  \node[very thick, fill] (i4) at (-126:0.9) {$$};
  \node (i5) at (162:0.9) {$$};
  \draw[very thick, fill] (o1)--(o2);
  \draw (o2)--(o3);
  \draw (o3)--(o4);
  \draw (o4)--(o5);
  \draw[very thick, fill] (o5)--(o1);
  \draw (o1)--(i1);
  \draw[very thick, fill] (o2)--(i2);
  \draw (o3)--(i3);
  \draw (o4)--(i4);
  \draw (o5)--(i5);
  \draw (i1)--(i3);
  \draw (i3)--(i5);
  \draw (i5)--(i2);
  \draw[very thick, fill] (i2)--(i4);
  \draw (i4)--(i1);
\end{tikzpicture}
}

\newcommand{\gsixQi}{
\begin{tikzpicture}[scale=0.4, every node/.style={circle, draw, inner sep=1.2pt}, baseline={([yshift=-1.0ex]current bounding box.center)}
]
  \node (o1) at ( 90:1.6) {$$};
  \node (o2) at ( 18:1.6) {$$};
  \node (o3) at (-54:1.6) {$$};
  \node (o4) at (-126:1.6) {$$};
  \node[very thick, fill] (o5) at (162:1.6) {$$};
  \node[red, very thick, fill] (i1) at ( 90:0.9) {$$};
  \node[red, very thick, fill] (i2) at ( 18:0.9) {$$};
  \node[very thick, fill] (i3) at (-54:0.9) {$$};
  \node[very thick, fill] (i4) at (-126:0.9) {$$};
  \node[red, very thick, fill] (i5) at (162:0.9) {$$};
  \draw (o1)--(o2);
  \draw (o2)--(o3);
  \draw (o3)--(o4);
  \draw (o4)--(o5);
  \draw (o5)--(o1);
  \draw (o1)--(i1);
  \draw (o2)--(i2);
  \draw (o3)--(i3);
  \draw (o4)--(i4);
  \draw[very thick, fill] (o5)--(i5);
  \draw[very thick, fill] (i1)--(i3);
  \draw[very thick, fill] (i3)--(i5);
  \draw[red, very thick, fill] (i5)--(i2);
  \draw[very thick, fill] (i2)--(i4);
  \draw[very thick, fill] (i4)--(i1);
\end{tikzpicture} 
}

\newcommand{\gsixQaa}{
\begin{tikzpicture}[scale=0.4, every node/.style={circle, draw, inner sep=1.2pt}, baseline={([yshift=-1.0ex]current bounding box.center)}]
  \node[red, very thick, fill] (o1) at ( 90:1.6) {$$};
  \node (o2) at ( 18:1.6) {$$};
  \node[red, very thick, fill] (o3) at (-54:1.6) {$$};
  \node (o4) at (-126:1.6) {$$};
  \node[very thick, fill] (o5) at (162:1.6) {$$};
  \node[red, very thick, fill] (i1) at ( 90:0.9) {$$};
  \node (i2) at ( 18:0.9) {$$};
  \node[very thick, fill] (i3) at (-54:0.9) {$$};
  \node (i4) at (-126:0.9) {$$};
  \node[red, very thick, fill] (i5) at (162:0.9) {$$};
  \draw (o1)--(o2);
  \draw (o2)--(o3);
  \draw (o3)--(o4);
  \draw (o4)--(o5);
  \draw[very thick, fill] (o5)--(o1);
  \draw[red, very thick, fill] (o1)--(i1);
  \draw (o2)--(i2);
  \draw[very thick, fill] (o3)--(i3);
  \draw (o4)--(i4);
  \draw[red, very thick, fill] (o5)--(i5);
  \draw[very thick, fill] (i1)--(i3);
  \draw[very thick, fill] (i3)--(i5);
  \draw (i5)--(i2);
  \draw (i2)--(i4);
  \draw (i4)--(i1);
\end{tikzpicture}
}

\newcommand{\gsixQiii}{
\begin{tikzpicture}[scale=0.4, every node/.style={circle, draw, inner sep=1.2pt}, baseline={([yshift=-1.0ex]current bounding box.center)}]
  \node[red, very thick, fill] (o1) at ( 90:1.6) {$$};
  \node[very thick, fill] (o2) at ( 18:1.6) {$$};
  \node (o3) at (-54:1.6) {$$};
  \node (o4) at (-126:1.6) {$$};
  \node[very thick, fill] (o5) at (162:1.6) {$$};
  \node[red, very thick, fill] (i1) at ( 90:0.9) {$$};
  \node[red, very thick, fill] (i2) at ( 18:0.9) {$$};
  \node (i3) at (-54:0.9) {$$};
  \node (i4) at (-126:0.9) {$$};
  \node[red, very thick, fill] (i5) at (162:0.9) {$$};
  \draw[very thick, fill] (o1)--(o2);
  \draw (o2)--(o3);
  \draw (o3)--(o4);
  \draw (o4)--(o5);
  \draw[very thick, fill] (o5)--(o1);
  \draw[red, very thick, fill] (o1)--(i1);
  \draw[very thick, fill] (o2)--(i2);
  \draw (o3)--(i3);
  \draw (o4)--(i4);
  \draw[very thick, fill] (o5)--(i5);
  \draw (i1)--(i3);
  \draw (i3)--(i5);
  \draw[red, very thick, fill] (i5)--(i2);
  \draw (i2)--(i4);
  \draw (i4)--(i1);
\end{tikzpicture}
}

\newcommand{\gsixEiii}{
\begin{tikzpicture}[scale=0.4, every node/.style={circle, draw, inner sep=1.2pt}, baseline={([yshift=-1.0ex]current bounding box.center)}]
  \node[red, very thick, fill] (o1) at ( 90:1.6) {$$};
  \node (o2) at ( 18:1.6) {$$};
  \node (o3) at (-54:1.6) {$$};
  \node (o4) at (-126:1.6) {$$};
  \node[very thick, fill] (o5) at (162:1.6) {$$};
  \node (i1) at ( 90:0.9) {$$};
  \node[red, very thick, fill] (i2) at ( 18:0.9) {$$};
  \node[very thick, fill] (i3) at (-54:0.9) {$$};
  \node[very thick, fill] (i4) at (-126:0.9) {$$};
  \node[red, very thick, fill] (i5) at (162:0.9) {$$};
  \draw (o1)--(o2);
  \draw (o2)--(o3);
  \draw (o3)--(o4);
  \draw (o4)--(o5);
  \draw[very thick, fill] (o5)--(o1);
  \draw (o1)--(i1);
  \draw (o2)--(i2);
  \draw (o3)--(i3);
  \draw (o4)--(i4);
  \draw[very thick, fill] (o5)--(i5);
  \draw (i1)--(i3);
  \draw[very thick, fill] (i3)--(i5);
  \draw[red, very thick, fill] (i5)--(i2);
  \draw[very thick, fill] (i2)--(i4);
  \draw (i4)--(i1);
\end{tikzpicture}
}

\newcommand{\gsixEaaa}{
\begin{tikzpicture}[scale=0.4, every node/.style={circle, draw, inner sep=1.2pt}, baseline={([yshift=-1.0ex]current bounding box.center)}]
  \node[red, very thick, fill] (o1) at ( 90:1.6) {$$};
  \node (o2) at ( 18:1.6) {$$};
  \node[red, very thick, fill] (o3) at (-54:1.6) {$$};
  \node (o4) at (-126:1.6) {$$};
  \node[very thick, fill] (o5) at (162:1.6) {$$};
  \node (i1) at ( 90:0.9) {$$};
  \node[red, very thick, fill] (i2) at ( 18:0.9) {$$};
  \node[very thick, fill] (i3) at (-54:0.9) {$$};
  \node (i4) at (-126:0.9) {$$};
  \node[red, very thick, fill] (i5) at (162:0.9) {$$};
  \draw (o1)--(o2);
  \draw (o2)--(o3);
  \draw (o3)--(o4);
  \draw (o4)--(o5);
  \draw[very thick, fill] (o5)--(o1);
  \draw (o1)--(i1);
  \draw (o2)--(i2);
  \draw[very thick, fill] (o3)--(i3);
  \draw (o4)--(i4);
  \draw[very thick, fill] (o5)--(i5);
  \draw (i1)--(i3);
  \draw[very thick, fill] (i3)--(i5);
  \draw[red, very thick, fill] (i5)--(i2);
  \draw (i2)--(i4);
  \draw (i4)--(i1);
\end{tikzpicture}
}

\newcommand{\gsixHIi}{
\begin{tikzpicture}[scale=0.4, every node/.style={circle, draw, inner sep=1.2pt}, baseline={([yshift=-1.0ex]current bounding box.center)}]
  \node (o1) at ( 90:1.6) {$$};
  \node[very thick, fill] (o2) at ( 18:1.6) {$$};
  \node[red, very thick, fill] (o3) at (-54:1.6) {$$};
  \node[red, very thick, fill] (o4) at (-126:1.6) {$$};
  \node[very thick, fill] (o5) at (162:1.6) {$$};
  \node (i1) at ( 90:0.9) {$$};
  \node (i2) at ( 18:0.9) {$$};
  \node[very thick, fill] (i3) at (-54:0.9) {$$};
  \node[very thick, fill] (i4) at (-126:0.9) {$$};
  \node (i5) at (162:0.9) {$$};
  \draw (o1)--(o2);
  \draw[very thick, fill] (o2)--(o3);
  \draw[red, very thick, fill] (o3)--(o4);
  \draw[very thick, fill] (o4)--(o5);
  \draw (o5)--(o1);
  \draw (o1)--(i1);
  \draw (o2)--(i2);
  \draw[very thick, fill] (o3)--(i3);
  \draw[very thick, fill] (o4)--(i4);
  \draw (o5)--(i5);
  \draw (i1)--(i3);
  \draw (i3)--(i5);
  \draw (i5)--(i2);
  \draw (i2)--(i4);
  \draw (i4)--(i1);
\end{tikzpicture}
}

\newcommand{\gsixHIii}{
\begin{tikzpicture}[scale=0.4, every node/.style={circle, draw, inner sep=1.2pt}, baseline={([yshift=-1.0ex]current bounding box.center)}]
  \node (o1) at ( 90:1.6) {$$};
  \node (o2) at ( 18:1.6) {$$};
  \node[red, very thick, fill] (o3) at (-54:1.6) {$$};
  \node[red, very thick, fill] (o4) at (-126:1.6) {$$};
  \node[very thick, fill] (o5) at (162:1.6) {$$};
  \node[red, very thick, fill] (i1) at ( 90:0.9) {$$};
  \node[red, very thick, fill] (i2) at ( 18:0.9) {$$};
  \node (i3) at (-54:0.9) {$$};
  \node[very thick, fill] (i4) at (-126:0.9) {$$};
  \node (i5) at (162:0.9) {$$};
  \draw (o1)--(o2);
  \draw (o2)--(o3);
  \draw[red, very thick, fill] (o3)--(o4);
  \draw[very thick, fill] (o4)--(o5);
  \draw (o5)--(o1);
  \draw (o1)--(i1);
  \draw (o2)--(i2);
  \draw (o3)--(i3);
  \draw[very thick, fill] (o4)--(i4);
  \draw (o5)--(i5);
  \draw (i1)--(i3);
  \draw (i3)--(i5);
  \draw (i5)--(i2);
  \draw[very thick, fill] (i2)--(i4);
  \draw[very thick, fill] (i4)--(i1);
\end{tikzpicture}
}

\newcommand{\geightii}{
\begin{tikzpicture}[scale=0.4, every node/.style={circle, draw, inner sep=1.2pt}, baseline={([yshift=-1.0ex]current bounding box.center)}
]
  \node (o1) at ( 90:1.6) {$$};
  \node[very thick, fill] (o2) at ( 18:1.6) {$$};
  \node[red, very thick, fill] (o3) at (-54:1.6) {$$};
  \node[red, very thick, fill] (o4) at (-126:1.6) {$$};
  \node (o5) at (162:1.6) {$$};
  \node[red,very thick, fill] (i1) at ( 90:0.9) {$$};
  \node[red, very thick, fill] (i2) at ( 18:0.9) {$$};
  \node[ very thick, fill] (i3) at (-54:0.9) {$$};
  \node[ very thick, fill] (i4) at (-126:0.9) {$$};
  \node[red, very thick, fill] (i5) at (162:0.9) {$$};
  \draw (o1)--(o2);
  \draw[very thick, fill] (o2)--(o3);
  \draw[red, very thick, fill] (o3)--(o4);
  \draw (o4)--(o5);
  \draw (o5)--(o1);
  \draw (o1)--(i1);
  \draw[very thick, fill] (o2)--(i2);
  \draw[very thick, fill] (o3)--(i3);
  \draw[very thick, fill] (o4)--(i4);
  \draw (o5)--(i5);
  \draw[very thick, fill] (i1)--(i3);
  \draw[very thick, fill] (i3)--(i5);
  \draw[red, very thick, fill] (i5)--(i2);
  \draw[very thick, fill] (i2)--(i4);
  \draw[very thick, fill] (i4)--(i1);
\end{tikzpicture} 
}

\newcommand{\geightiii}{
\begin{tikzpicture}[scale=0.4, every node/.style={circle, draw, inner sep=1.2pt}, baseline={([yshift=-1.0ex]current bounding box.center)}]
  \node[red, very thick, fill] (o1) at ( 90:1.6) {$$};
  \node[very thick, fill] (o2) at ( 18:1.6) {$$};
  \node[red, very thick, fill] (o3) at (-54:1.6) {$$};
  \node[red, very thick, fill] (o4) at (-126:1.6) {$$};
  \node (o5) at (162:1.6) {$$};
  \node (i1) at ( 90:0.9) {$$};
  \node[red, very thick, fill] (i2) at ( 18:0.9) {$$};
  \node[very thick, fill] (i3) at (-54:0.9) {$$};
  \node[very thick, fill] (i4) at (-126:0.9) {$$};
  \node[red, very thick, fill] (i5) at (162:0.9) {$$};
  \draw[very thick, fill] (o1)--(o2);
  \draw[very thick, fill] (o2)--(o3);
  \draw[red, very thick, fill] (o3)--(o4);
  \draw (o4)--(o5);
  \draw (o5)--(o1);
  \draw (o1)--(i1);
  \draw[very thick, fill] (o2)--(i2);
  \draw[very thick, fill] (o3)--(i3);
  \draw[very thick, fill] (o4)--(i4);
  \draw (o5)--(i5);
  \draw (i1)--(i3);
  \draw[very thick, fill] (i3)--(i5);
  \draw[red, very thick, fill] (i5)--(i2);
  \draw[very thick, fill] (i2)--(i4);
  \draw (i4)--(i1);
\end{tikzpicture}
}

\newcommand{\geightaa}{
\begin{tikzpicture}[scale=0.4, every node/.style={circle, draw, inner sep=1.2pt}, baseline={([yshift=-1.0ex]current bounding box.center)}]
  \node[red, very thick, fill] (o1) at (-54 :1.6) {$$};
  \node[red, very thick, fill] (o2) at (-126 :1.6) {$$};
  \node (o3) at (162:1.6) {$$};
  \node[red, very thick, fill] (o4) at (90:1.6) {$$};
  \node (o5) at (18:1.6) {$$};
  \node[ very thick, fill] (i1) at ( -54:0.9) {$$};
  \node[ very thick, fill] (i2) at ( -126:0.9) {$$};
  \node[red, very thick, fill] (i3) at (162:0.9) {$$};
  \node[red, very thick, fill] (i4) at ( 90 :0.9) {$$};
  \node[red, very thick, fill] (i5) at (18:0.9) {$$};
  \draw[red, very thick, fill] (o1)--(o2);
  \draw (o2)--(o3);
  \draw (o5)--(o1);
  \draw (o3)--(o4);
  \draw (o4)--(o5);
  \draw[ very thick, fill] (o1)--(i1);
  \draw[ very thick, fill] (o2)--(i2);
  \draw (o3)--(i3);
  \draw[red, very thick, fill] (o4)--(i4);
  \draw (o5)--(i5);
  \draw[ very thick, fill] (i1)--(i3);
  \draw[red, very thick, fill] (i3)--(i5);
  \draw[ very thick, fill] (i5)--(i2);
  \draw[ very thick, fill] (i2)--(i4);
  \draw[ very thick, fill] (i4)--(i1);
\end{tikzpicture}
}

\newcommand{\geighta}{
\begin{tikzpicture}[scale=0.4, every node/.style={circle, draw, inner sep=1.2pt}, baseline={([yshift=-1.0ex]current bounding box.center)}
]
  \node[very thick, fill] (o1) at (18:1.6) {$$};
  \node (o2) at (-54:1.6) {$$};
  \node (o3) at (-126:1.6) {$$};
  \node[very thick, fill] (o4) at (162:1.6) {$$};
  \node[red, very thick, fill] (o5) at (90:1.6) {$$};
  \node[red, very thick, fill] (i1) at ( 18:0.9) {$$};
  \node[very thick, fill] (i2) at (-54:0.9) {$$};
  \node[very thick, fill] (i3) at (-126:0.9) {$$};
  \node[red, very thick, fill] (i4) at (162:0.9) {$$};
  \node[red, very thick, fill] (i5) at ( 90:0.9) {$$};
  \draw (o1)--(o2);
  \draw (o2)--(o3);
  \draw[very thick, fill] (o5)--(o1);
  \draw (o3)--(o4);
  \draw[very thick, fill](o4)--(o5);
  \draw[very thick, fill] (o1)--(i1);
  \draw (o2)--(i2);
  \draw (o3)--(i3);
  \draw[very thick, fill] (o4)--(i4);
  \draw[red, very thick, fill] (o5)--(i5);
  \draw[very thick, fill] (i1)--(i3);
  \draw[very thick, fill] (i3)--(i5);
  \draw[very thick, fill] (i5)--(i2);
  \draw[ very thick, fill] (i2)--(i4);
  \draw[red, very thick, fill] (i4)--(i1);
\end{tikzpicture}
}

\title[Hidden structure of Jack L-R Coefficients]{Hidden structure of Jack Littlewood-Richardson Coefficients}

\author[R. Mickler]{Ryan Mickler}
\address{Singulariti Research, 55 University Street, Carlton, 3053, Australia}
\email{ry.mickler@gmail.com}

\begin{document}

\begin{abstract}

We argue that Jack Littlewood-Richardson coefficients $g_{\mu\nu}^{\lambda}(\al)$ are specialisations of certain novel polynomials. For the triple of partitions $(\mu,\nu,\lambda)=(21,21,321)$, we prove the corresponding polynomial is invariant under $S_6 \times \BZ_2$, which is identified as the automorphism group of the Johnson graph $J(6,3)$. We conjecture that these polynomials exhibit a factorization property on certain hyperplanes, which is a consequence of compatibility relations between polynomials associated to adjacent triples in the Young graph. As a consequence of this, we conjecture that the difference of adjacent Jack Littlewood-Richardson coefficients is divisible by the shared hook length.  \\
\end{abstract}

\maketitle


\ytableausetup{boxsize=1.1em}

\section{Background}

In this work, we continue the project begun in \cite{Mickler:2024stanley} of investigating the structure of Jack Littlewood-Richardson coefficients $g_{\mu\nu}^{\lambda}$, particularly for the case of $\mu=\{2,1\}$, $\nu=\{2,1\}$, $\lambda=\{3,2,1\}$ which is given by
\[ g_{21,21}^{321} = \frac{6\al(2+11\al+2\al^2)}{(1+2\al)(2+\al)(2+3\al)(3+2\al)}. \]
That the numerator of this coefficient is a \emph{non-negative} polynomial in $\al$ is the content of the so-called Stanley conjecture of \cite{Stanley:1989}. It this this conjecture that has motivated the present investigation.

This work culminates in a conjecture that reveals hidden structure in these coefficients.

\begin{conjecture}[see \eqref{conj:congruence}]
For two triples of partitions $(\mu_1,\nu_1,\lambda_1)$, $(\mu_2,\nu_2,\lambda_2)$ that differ by a single box move (a `pivot') in one of the pairs in the triples, say at $b_1 \in \sigma_1$, $b_2 \in \sigma_2$, so that
\[ h_{\sigma_1}^\sU({b_1}) = h_{\sigma_2}^{\sL}({b_2}),\]
then we have 
\begin{equation}\label{eq:adjmod}
 g_{\mu_1\nu_1;\lambda_1}(\al) \equiv \, g_{\mu_2\nu_2;\lambda_2}(\al) \mod h_{\sigma_1}^\sU({b_1}). 
\end{equation}
\end{conjecture}

For example, we look at the coefficients for two triples that differ by a single box move (for the box $b=(0,0)$ in $\{321\}$)
\begin{eqnarray}
g_{21,21;321}(\al)  &=&   6\al^4(1+2\al)(2+\al)(2+11\al+2\al^2), \\
g_{21,21;2211}(\al) &=&  4 \al^4(1 + 2\al)^2(3 + \al)(4 + \al).
\end{eqnarray}

We see that the \emph{difference} of these polynomials
\[g_{21,21;321}(\al)-g_{21,21;2211}(\al) = 2 \al^4 (1 + 2 \al) {\color{red} (3 + 2 \al)} (-4 + 6 \al + \al^2) \]
is divisible by $h_{2211}^{\sU}(b)=h_{321}^{\sL}(b)={\color{red} (3+2\al)}$.

\subsection{Notation}

We use the notation from \cite{Mickler:2024stanley} which we briefly summarize here. Let $\sA$ represent \emph{hook choice} of either upper ($\sU$) or lower ($\sL$), let $\bar \sA$ or $(-)\sA$ denote the opposite choice. 
For a box $b \in \sigma$ and a choice $\sA$, we consider a \emph{hook symbol} to be the formal symbol $\bh_\sigma^\sA(b)$ (or $\bh_b^\sA$ for short). The evaluation map ($\mathrm{ev}$, or $[\cdot]$) sends symbols to their corresponding $\al$-hook lengths\footnote{Other common notations are $h^*_\lambda(b) = h_\lambda'(b) = h_\lambda^\sU(b)$, and $h^\lambda_*(b) = h_\lambda(b) = h_\lambda^\sL(b)$.}.
\begin{equation}\label{evalmap}
\mathrm{ev} : \bh_\sigma^\sA(b) \mapsto [\bh_\sigma^\sA(b)]:=h^\sA_\sigma(b) \in \BZ[\al] 
 \end{equation}

\begin{definition}[\cite{Mickler:2024stanley}]
For any triple of partitions $\mu,\nu, \lambda$, a \emph{\bf{Stanley Diagram}} $\bm{D} \in \StD_{\mu,\nu}^{\lambda}$ is an assignment $b \to \bmD_b \in \{\sU,\sL\}$ for every box $b$ in each of $\mu,\nu$ and $\lambda$. 
We draw such assignments as a fraction e.g.
\[ 
\ytableausetup{boxsize=0.8em}
\bm{D} =  \frac{ \begin{ytableau}
*(boxU) \sU \\
*(boxL)  \sL  &  *(boxL) \sL
\end{ytableau}\,\,\begin{ytableau}
*(boxU) \sU  \\
*(boxL)  \sL &  *(boxU) \sU
\end{ytableau} }{ \begin{ytableau}
*(boxU) \sU \\
*(boxL) \sL  & *(boxU) \sU \\
*(boxL) \sL & *(boxU)\sU  & *(boxL) \sL
\end{ytableau} } \in \StD_{21,21}^{321}
\ytableausetup{boxsize=0.8em}
\] 
Inside a tableau, we denote upper hook symbol by $\begin{ytableau}*(boxU) \sU \end{ytableau}$ and lower hook symbol by $\begin{ytableau}*(boxL) \sL \end{ytableau}$.
Note that the hook choices for the boxes in $\lambda$ are in the denominator. 
Equivalently, a Stanley diagram is a rational function of the hook symbols,
\[ \bmD \leftrightarrow \frac{\left( \prod_{b\in \mu} \bh_b^{\bmD_b}  \right) \left( \prod_{b\in \nu} \bh_b^{\bmD_b}  \right)}{ \left( \prod_{b\in \lambda} \bh_b^{\bmD_b}  \right) }. \]

For a diagram $\bmD$, we let $\bmD'$ be the polynomial in the hook symbols given by $\bmD' := {\bf{j}}_\lambda \, \bmD$, where ${\bf{j}}_\lambda = \prod_{b\in \lambda} (\bh_b^{\sU} \bh_b^{\sL} )$. That  is,
\[ \bmD' = \left( \prod_{b\in \mu} \bh_b^{\bmD_b}  \right) \left( \prod_{b\in \nu} \bh_b^{\bmD_b}  \right) \left( \prod_{b\in \lambda} \bh_b^{\overline \bmD_b}  \right).  \]

\end{definition}

The motivating conjecture for this investigation is due to Richard Stanley, for whom we've named the constructions above.

 \begin{conjecture}[{``Strong'' Stanley Conjecture} \protect{\cite[C8.5]{Stanley:1989}}]\label{conj:strongstanley2}
If $c_{\mu\nu}^\lambda =1$, then there exists a Stanley diagram $\bm{D}\in \StD_{\mu\nu}^{\lambda}$ 
 such that the corresponding Jack LR coefficient is computed by evaluation
\begin{equation*}\label{formula:strongstanley2}
\jackjLR_{\mu \nu}^{\lambda} = [\bm{D}].
\end{equation*}
Furthermore, $\bmD$ may be taken to have an equal number of upper (and hence, lower) hooks in the numerator and denominator of the diagram.
\end{conjecture}


\ytableausetup{boxsize=0.8em}

The following generalization of Stanley's conjecture \eqref{conj:strongstanley2} to all cases $c_{\mu\nu}^\lambda>1$ was given in \cite{Mickler:2024stanley}.

\begin{definition}
A \emph{{\bf{Stanley Sum}}} $\bmss \in \StS_{\mu\nu}^{\lambda}$ is an element in the $\setZ$-span of Stanley Diagrams 
\[ \bmss  = \sum_{\bm{D}\, \in \StD_{\mu\nu}^{\lambda}} c_\bmD\,  \bm{D} \in \StS_{\mu\nu}^{\lambda},\]
with $c_\bmD \in \BZ$. The evaluation map extends linearly to such sums. 
\end{definition}

\begin{conjecture}[General Structure, \cite{Mickler:2024stanley}]\label{conj:generalstructure}
For each triple of partitions $\mu,\nu,\lambda$, there exists a Stanley sum, 
\begin{equation}\label{ex:stanleyruleform}
\JackjLR_{\mu\nu}^{\lambda} = \sum_{\bm{D} \in \StD_{\mu\nu}^{\lambda}} c_\bmD \, \bm{D}, \qquad c_\bmD \in \BZ,
\end{equation}
such that the corresponding Jack LR coefficient is computed by evaluation,
\begin{equation}
\jackjLR_{\mu\nu}^{\lambda} = [\JackjLR_{\mu\nu}^{\lambda}].
\end{equation}
We call such a Stanley sum $\JackjLR_{\mu\nu}^{\lambda}$ a \emph{\bf{Rule}} for $\jackjLR_{\mu\nu}^{\lambda}$.
\end{conjecture}
As was discussed in \cite{Mickler:2024stanley}, 
the form \eqref{ex:stanleyruleform} is not overly restrictive as there are $2^{|\lambda|+|\mu|+|\nu|}$ free integer coefficients $c_\bmD$, leaving many degrees of freedom.

We define $\JackjLR_{\mu\nu;\lambda} := \left(\JackjLR_{\mu\nu}^{\lambda}\right)'$ as a polynomial form of the rule ($[\JackjLR_{\mu\nu;\lambda}]$ is sometimes referred to as a Stanley Coefficient).

\begin{example}
The Jack LR rule is for the triple $\{1,1,1^2\}$ is given by the single diagram,
\begin{equation}
\JackjLR_{1,1}^{1^2} = \frac{\begin{ytableau}
*(boxU) a 
\end{ytableau} \times
\begin{ytableau}
*(boxU) b  
\end{ytableau}
}{
\begin{ytableau}
 *(boxU) c \\
 *(boxU) d   
\end{ytableau}}= \frac{\bh_\mu^\sU(a)\bh_\nu^\sU(b)}{\bh_{\lambda}^\sU(c)\bh_{\lambda}^\sU(d)} \in \StD_{\mu\nu}^{\lambda}
\end{equation}
which evaluates to
\[ [\JackjLR_{1,1}^{1^2}]=\jackjLR_{1,1}^{1^2} = \frac{\al^2}{\al(1+\al)} = \frac{\al}{1+\al}. \]
This gives the polynomial
\begin{equation}
\JackjLR_{1,1;1^2} = \begin{ytableau}
*(boxU)  \\
\end{ytableau}\,\,\begin{ytableau}
*(boxU)   \\
\end{ytableau} \,\, \begin{ytableau}
*(boxL) \\
*(boxL) 
\end{ytableau}  = \bh_\mu^\sU(a)\bh_\nu^\sU(b)\bh_\lambda^\sL(c)\bh_\lambda^\sL(d) \in \StD_{\mu\nu;\lambda},
\end{equation}
with corresponding Stanley coefficient 
\[ \jackjLR_{1,1;1^2} = \al^2(1)(2) = 2\al^2. \]
\end{example}

\subsubsection{The Main Example}
The previous work \cite{Mickler:2024stanley} was mainly focussed on the 
smallest triple of partitions $\{\mu,\nu,\lambda\}$ such that the Schur LR coefficient $c_{\mu\nu}^{\lambda}=2$ is greater than $1$, namely $\{21,21;321\}$. For this triple, we label the boxes as
\ytableausetup{boxsize=1.0em}
\begin{equation}
 \frac{\begin{ytableau}
*(boxS) a_1 \\
 *(boxS) a_2 &*(boxS)  a_3 
\end{ytableau} \,\,\, \begin{ytableau}
*(boxS) b_1  \\
*(boxS) b_2  & *(boxS) b_3
\end{ytableau}
}{\begin{ytableau}
*(boxS) c_1\\
*(boxS) c_2  & *(boxS) c_3  \\
*(boxS) c_4  &*(boxS)  c_5  & *(boxS) c_6 
\end{ytableau}}
\end{equation}

The main result of \cite{Mickler:2024stanley} was the discovery of the following 26-term Stanley sum $\bigstanleysum = \sum_{\bmD} c_{\bmD}\cdot \bmD \in \StS_{21,21}^{321}$,

\begin{eqnarray}\label{eq:generalsolution}
\ytableausetup{boxsize=0.4em}
\bigstanleysum &:= & 7 \left( \frac{ \begin{ytableau}
*(boxU)  \\
*(boxU)    &  *(boxU) 
\end{ytableau}\,\,\begin{ytableau}
*(boxU)   \\
*(boxU)   &  *(boxU) 
\end{ytableau} }{ \begin{ytableau}
*(boxU) \\
*(boxU)   & *(boxL)  \\
*(boxU)  & *(boxU)   & *(boxL) 
\end{ytableau} } \right) \quad - 2 \left( \frac{ \begin{ytableau}
*(boxU)  \\
*(boxU)    &  *(boxU) 
\end{ytableau}\,\,\begin{ytableau}
*(boxU)   \\
*(boxU)   &  *(boxU) 
\end{ytableau} }{ \begin{ytableau}
*(boxU) \\
*(boxU)   & *(boxU)  \\
*(boxU)  & *(boxU)   & *(boxL) 
\end{ytableau} } \right) \\
&& - 2 \left( \frac{ \begin{ytableau}
*(boxL)  \\
*(boxU)    &  *(boxU) 
\end{ytableau}\,\,\begin{ytableau}
*(boxU)   \\
*(boxU)   &  *(boxU) 
\end{ytableau} }{ \begin{ytableau}
*(boxU) \\
*(boxU)   & *(boxL)  \\
*(boxU)  & *(boxU)   & *(boxL) 
\end{ytableau} } +\frac{ \begin{ytableau}
*(boxU)  \\
*(boxU)    &  *(boxL) 
\end{ytableau}\,\,\begin{ytableau}
*(boxU)   \\
*(boxU)   &  *(boxU) 
\end{ytableau} }{ \begin{ytableau}
*(boxU) \\
*(boxU)   & *(boxL)  \\
*(boxU)  & *(boxU)   & *(boxL) 
\end{ytableau} } +\frac{ \begin{ytableau}
*(boxU)  \\
*(boxU)    &  *(boxU) 
\end{ytableau}\,\,\begin{ytableau}
*(boxL)   \\
*(boxU)   &  *(boxU) 
\end{ytableau} }{ \begin{ytableau}
*(boxU) \\
*(boxU)   & *(boxL)  \\
*(boxU)  & *(boxU)   & *(boxL) 
\end{ytableau} } +\frac{ \begin{ytableau}
*(boxU)  \\
*(boxU)    &  *(boxU) 
\end{ytableau}\,\,\begin{ytableau}
*(boxU)   \\
*(boxU)   &  *(boxL) 
\end{ytableau} }{ \begin{ytableau}
*(boxL) \\
*(boxU)   & *(boxL)  \\
*(boxU)  & *(boxU)   & *(boxU) 
\end{ytableau} }+\frac{ \begin{ytableau}
*(boxU)  \\
*(boxU)    &  *(boxU) 
\end{ytableau}\,\,\begin{ytableau}
*(boxU)   \\
*(boxU)   &  *(boxU) 
\end{ytableau} }{ \begin{ytableau}
*(boxU) \\
*(boxL)   & *(boxL)  \\
*(boxU)  & *(boxU)   & *(boxL) 
\end{ytableau} } +\frac{ \begin{ytableau}
*(boxU)  \\
*(boxU)    &  *(boxU) 
\end{ytableau}\,\,\begin{ytableau}
*(boxU)   \\
*(boxU)   &  *(boxU) 
\end{ytableau} }{ \begin{ytableau}
*(boxU) \\
*(boxU)   & *(boxL)  \\
*(boxU)  & *(boxL)   & *(boxL) 
\end{ytableau} }  \right) \nonumber \\
&&  - 2 \left( \frac{ \begin{ytableau}
*(boxU)  \\
*(boxL)    &  *(boxU) 
\end{ytableau}\,\,\begin{ytableau}
*(boxU)   \\
*(boxU)   &  *(boxU) 
\end{ytableau} }{ \begin{ytableau}
*(boxU) \\
*(boxU)   & *(boxL)  \\
*(boxU)  & *(boxU)   & *(boxL) 
\end{ytableau} }+
 \frac{ \begin{ytableau}
*(boxU)  \\
*(boxU)    &  *(boxU) 
\end{ytableau}\,\,\begin{ytableau}
*(boxU)   \\
*(boxL)   &  *(boxU) 
\end{ytableau} }{ \begin{ytableau}
*(boxU) \\
*(boxU)   & *(boxL)  \\
*(boxU)  & *(boxU)   & *(boxL) 
\end{ytableau} } + 
 \frac{ \begin{ytableau}
*(boxU)  \\
*(boxU)    &  *(boxU) 
\end{ytableau}\,\,\begin{ytableau}
*(boxU)   \\
*(boxU)   &  *(boxU) 
\end{ytableau} }{ \begin{ytableau}
*(boxU) \\
*(boxU)   & *(boxL)  \\
*(boxL)  & *(boxU)   & *(boxL) 
\end{ytableau} } \right) \nonumber \\
&& +1 \left(  \frac{ \begin{ytableau}
*(boxU)  \\
*(boxL)    &  *(boxU) 
\end{ytableau}\,\,\begin{ytableau}
*(boxU)   \\
*(boxU)   &  *(boxU) 
\end{ytableau} }{ \begin{ytableau}
*(boxU) \\
*(boxU)   & *(boxU)  \\
*(boxU)  & *(boxU)   & *(boxL) 
\end{ytableau} } +  \frac{ \begin{ytableau}
*(boxU)  \\
*(boxU)    &  *(boxU) 
\end{ytableau}\,\,\begin{ytableau}
*(boxU)   \\
*(boxL)   &  *(boxU) 
\end{ytableau} }{ \begin{ytableau}
*(boxU) \\
*(boxU)   & *(boxU)  \\
*(boxU)  & *(boxU)   & *(boxL) 
\end{ytableau} }+  \frac{ \begin{ytableau}
*(boxU)  \\
*(boxU)    &  *(boxU) 
\end{ytableau}\,\,\begin{ytableau}
*(boxU)   \\
*(boxU)   &  *(boxU) 
\end{ytableau} }{ \begin{ytableau}
*(boxU) \\
*(boxU)   & *(boxU)  \\
*(boxL)  & *(boxU)   & *(boxL) 
\end{ytableau} }\right) \nonumber \\
&& +1 \left( \frac{ \begin{ytableau}
*(boxL)  \\
*(boxL)    &  *(boxU) 
\end{ytableau}\,\,\begin{ytableau}
*(boxU)   \\
*(boxU)   &  *(boxU) 
\end{ytableau} }{ \begin{ytableau}
*(boxU) \\
*(boxU)   & *(boxL)  \\
*(boxU)  & *(boxU)   & *(boxL) 
\end{ytableau} } +\frac{ \begin{ytableau}
*(boxU)  \\
*(boxL)    &  *(boxL) 
\end{ytableau}\,\,\begin{ytableau}
*(boxU)   \\
*(boxU)   &  *(boxU) 
\end{ytableau} }{ \begin{ytableau}
*(boxU) \\
*(boxU)   & *(boxL)  \\
*(boxU)  & *(boxU)   & *(boxL) 
\end{ytableau} } +\frac{ \begin{ytableau}
*(boxU)  \\
*(boxU)    &  *(boxU) 
\end{ytableau}\,\,\begin{ytableau}
*(boxL)   \\
*(boxL)   &  *(boxU) 
\end{ytableau} }{ \begin{ytableau}
*(boxU) \\
*(boxU)   & *(boxL)  \\
*(boxU)  & *(boxU)   & *(boxL) 
\end{ytableau} } +\frac{ \begin{ytableau}
*(boxU)  \\
*(boxU)    &  *(boxU) 
\end{ytableau}\,\,\begin{ytableau}
*(boxU)   \\
*(boxL)   &  *(boxL) 
\end{ytableau} }{ \begin{ytableau}
*(boxL) \\
*(boxU)   & *(boxL)  \\
*(boxU)  & *(boxU)   & *(boxU) 
\end{ytableau} }+\frac{ \begin{ytableau}
*(boxU)  \\
*(boxU)    &  *(boxU) 
\end{ytableau}\,\,\begin{ytableau}
*(boxU)   \\
*(boxU)   &  *(boxU) 
\end{ytableau} }{ \begin{ytableau}
*(boxU) \\
*(boxL)   & *(boxL)  \\
*(boxL)  & *(boxU)   & *(boxL) 
\end{ytableau} } +\frac{ \begin{ytableau}
*(boxU)  \\
*(boxU)    &  *(boxU) 
\end{ytableau}\,\,\begin{ytableau}
*(boxU)   \\
*(boxU)   &  *(boxU) 
\end{ytableau} }{ \begin{ytableau}
*(boxU) \\
*(boxU)   & *(boxL)  \\
*(boxL)  & *(boxL)   & *(boxL) 
\end{ytableau} }  \right) \nonumber \\
&& +1 \left( \frac{ \begin{ytableau}
*(boxL)  \\
*(boxU)    &  *(boxU) 
\end{ytableau}\,\,\begin{ytableau}
*(boxU)   \\
*(boxU)   &  *(boxL) 
\end{ytableau} }{ \begin{ytableau}
*(boxL) \\
*(boxU)   & *(boxL)  \\
*(boxU)  & *(boxU)   & *(boxU) 
\end{ytableau} } +
 \frac{ \begin{ytableau}
*(boxU)  \\
*(boxU)    &  *(boxL) 
\end{ytableau}\,\,\begin{ytableau}
*(boxL)   \\
*(boxU)   &  *(boxU) 
\end{ytableau} }{ \begin{ytableau}
*(boxU) \\
*(boxU)   & *(boxL)  \\
*(boxU)  & *(boxU)   & *(boxL) 
\end{ytableau} } +
 \frac{ \begin{ytableau}
*(boxL)  \\
*(boxU)    &  *(boxU) 
\end{ytableau}\,\,\begin{ytableau}
*(boxU)   \\
*(boxU)   &  *(boxU) 
\end{ytableau} }{ \begin{ytableau}
*(boxU) \\
*(boxU)   & *(boxL)  \\
*(boxU)  & *(boxL)   & *(boxL) 
\end{ytableau} }+
 \frac{ \begin{ytableau}
*(boxU)  \\
*(boxU)    &  *(boxL) 
\end{ytableau}\,\,\begin{ytableau}
*(boxU)   \\
*(boxU)   &  *(boxU) 
\end{ytableau} }{ \begin{ytableau}
*(boxU) \\
*(boxL)   & *(boxL)  \\
*(boxU)  & *(boxU)   & *(boxL) 
\end{ytableau} }+
 \frac{ \begin{ytableau}
*(boxU)  \\
*(boxU)    &  *(boxU) 
\end{ytableau}\,\,\begin{ytableau}
*(boxL)   \\
*(boxU)   &  *(boxU) 
\end{ytableau} }{ \begin{ytableau}
*(boxU) \\
*(boxU)   & *(boxL)  \\
*(boxU)  & *(boxL)   & *(boxL) 
\end{ytableau} }+
 \frac{ \begin{ytableau}
*(boxU)  \\
*(boxU)    &  *(boxU) 
\end{ytableau}\,\,\begin{ytableau}
*(boxU)   \\
*(boxU)   &  *(boxL) 
\end{ytableau} }{ \begin{ytableau}
*(boxL) \\
*(boxL)   & *(boxL)  \\
*(boxU)  & *(boxU)   & *(boxU) \nonumber
\end{ytableau} }
\right).
\end{eqnarray}
\ytableausetup{boxsize=0.9em}

Beyond being a rule for $\{21,21,321\}$, this Stanley sum was conjectured to be a rule simultaneously for a 7-parameter family of triples. The following rests on definitions in the next section.
\begin{conjecture}[\cite{Mickler:2024stanley}]
The Stanley sum \eqref{eq:generalsolution} computes the Jack Little-Richardson coefficients for the $7$-parameter window family $\windowfamily=\{W\}$ given by \eqref{eq:7paramwindowronly}, that is
\begin{equation}\label{eq:bigequation}
F_W \,[\bigstanleysum]_W = \jackjLR_{W}.
\end{equation}
where $F_W$ is the window factor given by \eqref{eq:7paramwindowronly}. That is, $\bigstanleysum$ is a Rule,
\[ \bigstanleysum = \JackjLR_{21,21}^{321}. \]
\end{conjecture}

We don't assume this conjecture holds in this present work, rather we just take the Stanley sum \eqref{eq:generalsolution} as given, and investigate its properties. In section \ref{section:discussion}, we develop novel reasoning as to why this sum should compute the Jack LR coefficients.

\subsection{Observations}
We begin this paper by observing the following properties of the Stanley sum \eqref{eq:generalsolution}. 
\begin{observation}
For each diagram $\bmD$ in the support (i.e. $c_\bmD\neq 0$) of $\bigstanleysum$, we have
\begin{equation}\label{eq:tildeb3def}
\bmD_{b_3} = \bmD_{c_1}= \overline\bmD_{c_6}.
\end{equation}
\end{observation}
Thus, we consider a new virtual box, $\tilde b_3$, that refers to the value of all three. 
We black out the boxes at $c_1$ and $c_6$, as they are determined by the hook assignment at $b_3$. From now on we only refer to the remaining ten boxes.

\ytableausetup{boxsize=1.1em}
\begin{equation}
 \frac{\begin{ytableau}
\none   &*(boxS) a_1 \\
\none    &*(boxS) a_2 &*(boxS)  a_3   \\
\none  &\none  & \none 
\end{ytableau} \,\, \begin{ytableau}
\none & *(boxS) b_1  \\
\none &  *(boxS) b_2  & *(boxS) b_3  \\
\none & \none   & \none 
\end{ytableau}
}{\begin{ytableau}
\none \\
\none & *(boxB) c_1\\
\none & *(boxS) c_2  & *(boxS) c_3  \\
\none & *(boxS) c_4  &*(boxS)  c_5  & *(boxB) c_6 \\
\none & \none   & \none  & \none 
\end{ytableau}}
\end{equation}

The second thing we observe is that each diagram in the support is at most two box flips away from the following reference diagram,

\begin{equation}\label{eq:Xdef}
\bmX:=\frac{ \begin{ytableau}
*(boxU)  \\
*(boxU)    &  *(boxU) 
\end{ytableau}\,\,\begin{ytableau}
*(boxU)   \\
*(boxU)   &  *(boxU) 
\end{ytableau} }{ \begin{ytableau}
*(boxB) \\
*(boxU)   & *(boxL)  \\
*(boxU)  & *(boxU)   & *(boxB) 
\end{ytableau} }
\end{equation}

For a subset of boxes $J$, let $\bmD^J$ be the diagram $\bmD$ however with the hooks flipped at all boxes in $J$.

\begin{observation}\label{lemma:petersengraph}
The Stanley sum \eqref{eq:generalsolution} has the following decomposition,
\begin{equation}\label{eq:721sum}
\bigstanleysum = 7 \bmX - 2\sum_{b} \bmX^{\{b\}} + \sum_{\{a,b\} \in \Theta} \bmX^{\{a,b\}},
\end{equation}
where $\Theta$ is a set of fifteen pairs of boxes.
The pairs $\Theta$ constitute the edges of the Petersen graph $\petersen$, with vertices labelled as 
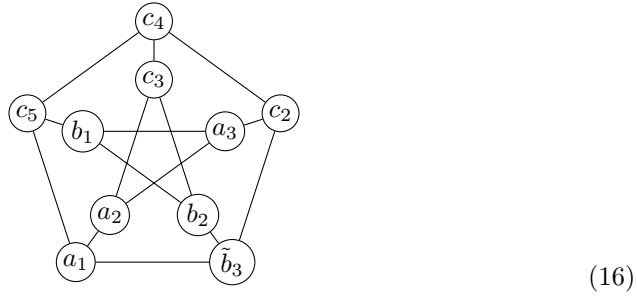
\begin{equation}\label{diag:petersongraph}
\begin{tikzpicture}[scale=1.1, every node/.style={circle, draw, inner sep=1.2pt}]
  \node (o1) at ( 90:1.6) {$c_4$};
  \node (o2) at ( 18:1.6) {$c_2$};
  \node (o3) at (-54:1.6) {$\tilde b_3$};
  \node (o4) at (-126:1.6) {$a_1$};
  \node (o5) at (162:1.6) {$c_5$};
  \node (i1) at ( 90:0.9) {$c_3$};
  \node (i2) at ( 18:0.9) {$ a_3$};
  \node (i3) at (-54:0.9) {$b_2$};
  \node (i4) at (-126:0.9) {$a_2$};
  \node (i5) at (162:0.9) {$ b_1$};
  \draw (o1)--(o2)--(o3)--(o4)--(o5)--(o1);
  \draw (o1)--(i1);
  \draw (o2)--(i2);
  \draw (o3)--(i3);
  \draw (o4)--(i4);
  \draw (o5)--(i5);
  \draw (i1)--(i3)--(i5)--(i2)--(i4)--(i1);
\end{tikzpicture}
\end{equation}
This identification defines an action of $\Aut(\petersen) = S_5$ on boxes.
\end{observation}
This surprising structure underlines many of the constructions in this work. The Petersen graph $\petersen$ is most usefully described at the Kneser graph $K(5,2)$, in which vertices are 2-subsets $I=(ij)$ of a set with five elements, $[5]$, and $I\sim J$ if $I$ and $J$ are disjoint. The $\Aut(\petersen) = S_5$ action is then the permutation action on $[5]$ acting on 2-subsets.

\section{The Hook Space}

\subsection{Windows}

We begin by recalling the definitions of windows from the previous work \cite{Mickler:2024stanley}.

\begin{definition}
For a fixed triple of partitions $T=(\mu,\nu;\lambda)$, a \emph{\bf{window}} $W$ is a triple of partitions $(\mu_W,\nu_W; \lambda_W)$ such that $W$ reduces to the \emph{\bf{root}} triple $T$ under a sequence of three moves, 

\begin{itemize}
\item King-Tollu-Toumazet \cite{King:2009} factorization on the triple,
\item Windowing on the pair $\mu\subset \lambda$,
\item Windowing on the pair $\nu \subset \lambda$.
\end{itemize}
Thus a window $W$ comes with an inclusion of boxes from $b\in T \mapsto b_W \in W$. A \emph{\bf{window family}} $\windowfamily = \{W_i\}_{i\in V}$ is a parametrised collection of windows.
\end{definition}

\begin{example}\label{ex:window1112}
For the root $T=\{1,1,1^2\}$, consider the window family  $\windowfamily = \{W_{m,n}\}$ with
\[ W_{m,n} = \{1^{m+1},(n+1),(n+1)1^{m+1}\}, \]
where the four boxes $\{a,b,c,d\}$ of $T$ are included as
\begin{equation}
\frac{\begin{ytableau}
\none[m] &  \\
\none    & *(boxS) a \\
\none[] & \none  & \none 
\end{ytableau} \times
\begin{ytableau}
\none   \\
\none & *(boxS) b &  \\
\none[] & \none  & \none[n] 
\end{ytableau}
}{
\begin{ytableau}
\none & *(boxS) c \\
\none[m] &  \\
\none    & *(boxS) d &   \\
\none[] & \none  & \none[n]
\end{ytableau}}.
\end{equation}
Note there is no KTT factorization in this example.
\end{example}

\begin{conjecture}[Jack Windowing, \cite{Mickler:2024stanley}] Let $W$ be a window for a triple $T$. There exists a rule $\JackjLR_T$, such that a rule $\JackjLR_W$ can be constructed with the following properties.
\begin{itemize}
\item The boxes outside of $T \subset W$ are assigned a specific choice of hooks $\windfact_W$
\item 
The boxes inside $T \subset W$ are assigned the same rule $\JackjLR_T$ (extended by linearity). 
\end{itemize} Symbolically,
\begin{equation}
\JackjLR_W = \windfact_W \circ (\JackjLR_T)_W.
\end{equation}
\end{conjecture}

\begin{example}
Returning to the window in example \eqref{ex:window1112}, we find

\begin{equation*}
\JackjLR_{1^{m+1},n+1}^{(n+1)1^{m+1}} = \windfact_{W_{m,n}} \circ (\JackjLR_{1,1}^{1^2})_{W_{m,n}} = \frac{\begin{ytableau}
\none[m] & *(boxU) \\
\none    & *(boxS)  \\
\none[] & \none  & \none 
\end{ytableau} \,\,
\begin{ytableau}
\none   \\
\none & *(boxS)  & *(boxL) \\
\none[] & \none  & \none[n] 
\end{ytableau}
}{
\begin{ytableau}
\none & *(boxS)  \\
\none[m] & *(boxU) \\
\none    & *(boxS)  & *(boxL)  \\
\none[] & \none  & \none[n]
\end{ytableau}} 
\circ 
\frac{\begin{ytableau}
*(boxU)  
\end{ytableau} \,\,
\begin{ytableau}
*(boxU)  
\end{ytableau}
}{
\begin{ytableau}
 *(boxU)  \\
 *(boxU)    
\end{ytableau}} = 
\frac{\begin{ytableau}
\none[m] & *(boxU) \\
\none    & *(boxU)  \\
\none[] & \none  & \none 
\end{ytableau} \,\,
\begin{ytableau}
\none   \\
\none & *(boxU)  & *(boxL) \\
\none[] & \none  & \none[n] 
\end{ytableau}
}{
\begin{ytableau}
\none & *(boxU)  \\
\none[m] & *(boxU) \\
\none    & *(boxU)  & *(boxL)  \\
\none[] & \none  & \none[n]
\end{ytableau}}.
\end{equation*}

\end{example}

\ytableausetup{boxsize=1.0em}

\begin{example}\label{example:g21212211}
We consider the following sequence of three moves which reduces a $7$-parameter family down to the $c_{\mu\nu}^{\lambda}=1$ root $\{21,21,2211\}$.

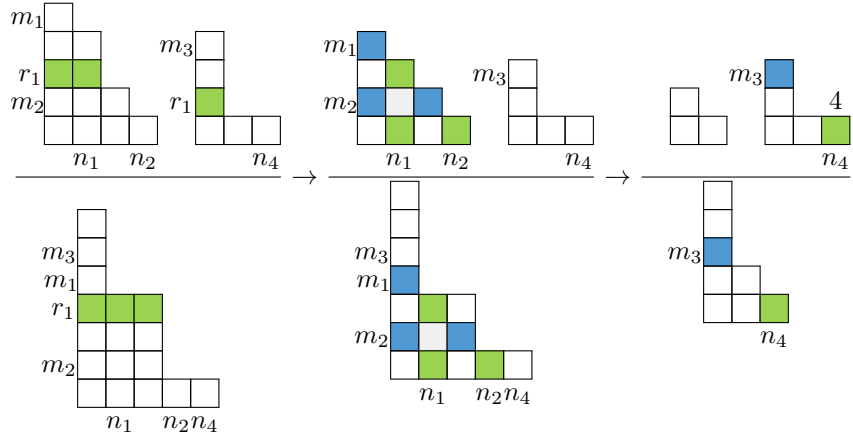
\[
 \frac{\begin{ytableau}
\none[m_1\,]  &     \\
\none   &  & \\
\none[r_1]  &*(boxL)  & *(boxL) \\
\none[m_2\,]   &   &   &   \\
\none    &  &  &   &    \\
\none  &\none & \none[n_1]  & \none & \none[n_2] 
\end{ytableau} \,\, \begin{ytableau}
\none \\
\none[m_3\,\,] &   \\
\none &     \\
\none[r_1] & *(boxL)  & \none & \none[]  \\
\none &     &   &    \\
\none & \none   & \none & \none[n_4]
\end{ytableau}
}{\begin{ytableau}
\none   \\
\none &  \\
\none[m_3\,\,] &   \\
\none[m_1\,] &      \\
\none[r_1] & *(boxL) & *(boxL) & *(boxL) \\
\none &   &   &     \\
\none[m_2\,\,] &   &    &   \\
\none &    &    &   &    &   \\
\none & \none  & \none[n_1] & \none  & \none[n_2] & \none[n_4] & \none 
\end{ytableau}}
\to
 \frac{\begin{ytableau}
\none[m_1\,]  &*(boxU)   \\
\none   &   &*(boxL) \\
\none[m_2\,]   &*(boxU)  & *(boxG)  & *(boxU)  \\
\none    &  &*(boxL)  &  & *(boxL)   \\
\none  &\none & \none[n_1]  & \none & \none[n_2] 
\end{ytableau} \,\, \begin{ytableau}
\none \\
\none[m_3\,\,] &  \\
\none &   \\
\none &   &   &   \\
\none & \none   & \none & \none[n_4]
\end{ytableau}
}{\begin{ytableau}
\none &  \\
\none &  \\
\none[m_3\,\,] &   \\
\none[m_1\,] & *(boxU)   \\
\none &  &*(boxL)  &   \\
\none[m_2\,\,] & *(boxU) & *(boxG) &*(boxU)   \\
\none &  &*(boxL)  &  & *(boxL)  &     \\
\none & \none  & \none[n_1] & \none  & \none[n_2] & \none[n_4] & \none 
\end{ytableau}}
\to
 \frac{\begin{ytableau}
\none  \\
\none   &   \\
\none   &    &  \\
\none  &\none   & \none 
\end{ytableau} \,\, \begin{ytableau}
\none[m_3\,\,] & *(boxU)    \\
\none &   & \none & \none[4]  \\
\none &   &   & *(boxL) \\
\none & \none   & \none & \none[n_4]
\end{ytableau}
}{\begin{ytableau}
\none &   \\
\none &  \\
\none[m_3\,\,] &*(boxU)  \\
\none &   &    \\ 
\none &   &   &  *(boxL)  \\
\none & \none  & \none  & \none[n_4] & \none 
\end{ytableau}}
\]
For this example we find the rule (which in this case is fixed by the above properties)
\ytableausetup{boxsize=1.1em}
\begin{equation}\label{eq:window21212211}
 \JackjLR_{W} = \frac{\begin{ytableau}
\none[m_1\,]  &*(boxU)    \\
\none   &*(boxS)  &*(boxL) \\
\none[r_1]  &*(boxL)  & *(boxL) \\
\none[m_2\,]   &*(boxU)  & *(boxG)  & *(boxU)  \\
\none    &*(boxS) &*(boxL)  &*(boxS) & *(boxL)   \\
\none  &\none & \none[n_1]  & \none & \none[n_2] 
\end{ytableau} \,\, \begin{ytableau}
\none[m_3\,\,] & *(boxU) \\
\none & *(boxS)  \\
\none[r_1] & *(boxL)  & \none & \none[]  \\
\none &  *(boxS)  & *(boxS) & *(boxL) \\
\none & \none   & \none & \none[n_4]
\end{ytableau}
}{\begin{ytableau}
\none & *(boxS) \\
\none & *(boxS)\\
\none[m_3\,\,] & *(boxU) \\
\none[m_1\,] & *(boxU)    \\
\none[r_1] & *(boxL) & *(boxL) & *(boxL) \\
\none & *(boxS) &*(boxL)  & *(boxS)  \\
\none[m_2\,\,] & *(boxU) & *(boxG) &*(boxU)  \\
\none & *(boxS) &*(boxL)  &*(boxS)  & *(boxL)  &  *(boxL)  \\
\none & \none  & \none[n_1] & \none  & \none[n_2] & \none[n_4] & \none 
\end{ytableau}} \circ
\frac{\begin{ytableau}
*(boxU)   \\
*(boxU)  &   *(boxU) 
\end{ytableau} \,\, \begin{ytableau}
*(boxU)   \\
*(boxU)    & *(boxU) 
\end{ytableau}
}{\begin{ytableau}
*(boxU)   \\
*(boxU)  \\
*(boxU)   & *(boxU)      \\
*(boxU)  & *(boxU)  
\end{ytableau}}=: \windfact_{W} \circ (\JackjLR_{21,21}^{2211} )_W
\end{equation} 


\end{example}

\subsection{Hook spaces}

To proceed, we introduce \emph{hook spaces}. These are essentially the space of Stanley sums modulo the kernel of the evaluation map for a given window family. 


\begin{example}
We return to the window family $W_{m,n}$ of \eqref{ex:window1112} for the triple $\{1,1,1^2\}$,
\begin{equation}
\frac{\begin{ytableau}
\none[m] &  \\
\none    & *(boxU) a \\
\none[\mu] & \none  & \none 
\end{ytableau} \times
\begin{ytableau}
\none   \\
\none & *(boxU) b &  \\
\none[\nu] & \none  & \none[n] 
\end{ytableau}
}{
\begin{ytableau}
\none & *(boxU) c \\
\none[m] &  \\
\none    & *(boxU) d &   \\
\none[\lambda] & \none  & \none[n]
\end{ytableau}}
\end{equation}
Here, we have the hook lengths
\begin{eqnarray*}
h_\lambda^\sU(d) = m+1 + (n+1)\al ,&  h_\mu^\sL(a) = m+1 + (0)\al, \\
h_\nu^\sU(b) = 0 + (n+1)\al,\,\qquad & h_\lambda^\sU(c) = \al.\qquad\qquad\quad
\end{eqnarray*}
We note the single relation holds for all $m,n$
\[ h_\lambda^\sU(d) - h_\mu^\sL(a)-h_\nu^\sU(b) = 0 \]
Thus the linear sum 
\begin{equation}\label{eq:simplerel}
 {\bf{r}}_1 := \bh_\lambda^\sU(d) - \bh_\mu^\sL(a)-\bh_\nu^\sU(b),
 \end{equation}
is in the kernel of the evaluation map.
\end{example}

The relation \eqref{eq:simplerel} captures the configuration of the family $W_{m,n}$ as imposed on the hook symbols of the root of the family. We introduce a formal variable $\bbeta$, which evaluates to $[\bbeta] = \al-1$.

\begin{definition} 
For a root triple $\{\mu,\nu,\lambda\}$ with a window family $\windowfamily$, define the \emph{\bf{hook space}} $\StDR_{\mu\nu,\lambda}(\windowfamily)$ as the intersection of the kernels of the evaluation map $ev_W$ for each window $W \in \windowfamily$ acting on homogeneous polynomials in the hook variables and $\beta$.
\[ \StDR_{\mu\nu,\lambda}(\windowfamily) = \bigcap_{W \in \windowfamily} \mathrm{Ker}\left( \mathrm{ev}_W : \Sym( \bh,\bbeta )  \to \BQ[\al] \right). \]

\end{definition}
By design, the evaluation map $\StS_{\mu\nu,\lambda} \to \mathbb{Z}[\al]$ factors through the hook space. Let $\iota^* : \StS_{\mu\nu,\lambda} \to \StDR_{\mu\nu,\lambda}(\windowfamily)$ be the pullback to the hook space.
\[
\begin{tikzcd}
\StS_{\mu\nu,\lambda} \arrow[r, "\iota^*"] \arrow[dr, "\mathrm{ev}"'] & \StDR_{\mu\nu,\lambda}(\windowfamily) \arrow[d, "\overline{\mathrm{ev}}"] \\
 & \BQ[\al]
\end{tikzcd}
\]

Since $h^\sU_\mu(b)-h^\sL_\mu(b) = \al-1$ for any box $b$, we always have the basic linear relations
\begin{equation}\label{betarelation}
\bh^\sU_{\sigma}(b)-\bh^\sL_{\sigma}(b)-\bbeta \in \StDR_{\mu\nu,\lambda}(\windowfamily).
\end{equation}
Note that if $\windowfamily' \subset \windowfamily$ is a smaller family of windows, then
\[ \StDR_{\mu\nu,\lambda}(\windowfamily) \twoheadrightarrow \StDR_{\mu\nu,\lambda}(\windowfamily') \twoheadrightarrow \StDR_{\mu\nu,\lambda}(\windowfamily_0)  \]
where $\windowfamily_0 := \{\mu\nu,\lambda\}$ is the family consisting of only the trivial window of the root triple. This is generated by the many relations between the usual hook lengths.

In general, there are many relations in the hook space. For example, for any two boxes $a<b$ such that $a\wedge b \in\sigma$, we have
\[ \bh^\sL_\sigma(a\vee b) - \bh^\sL_\sigma(a)-\bh^\sL_\sigma(b) + \bh^\sL_\sigma(a\wedge b) = 0.\]


\begin{example}
Returning to example \ref{ex:window1112} of $\{1,1,1^2\}$, where we have

\begin{equation}
\JackjLR_{1,1;1^2} = \begin{ytableau}
*(boxU)  \\
\end{ytableau}\,\,\begin{ytableau}
*(boxU)   \\
\end{ytableau} \,\, \begin{ytableau}
*(boxL) \\
*(boxL) 
\end{ytableau}  = \bh_\mu^\sU(a)\bh_\nu^\sU(b)\bh_\lambda^\sL(c)\bh_\lambda^\sL(d) \in \StD_{\mu\nu;\lambda} 
\end{equation}
The hook space is given by
\begin{equation}
\StDR_{1,1}^{1^2}(\windowfamily) = \{ \bh_\lambda^\sU(d),\bh_\lambda^\sU(c),\bh_\nu^\sU(b),\bh_\mu^\sU(a) \}[\beta]/ \langle \bh_\lambda^\sU(d) -\bh_\nu^\sU(b) -\bh_\mu^\sU(a)+ \bbeta \rangle
\end{equation}
so we have
\[ \JackjLR_{1,1;1^2} = \bh_\mu^\sU(a)\bh_\nu^\sU(b)(\bh_\lambda^\sU(c)-\beta)(\bh_\lambda^\sU(d)-\beta) \]

\end{example}

For the rest of this paper, we will be focusing on the following object.

\begin{definition}
The \emph{\bf{Jack Littlewood-Richardson Polynomial}} $\stanleypoly$ for a rule $\JackjLR_{\mu\nu}^{\lambda}$ that solves a window family $\windowfamily$ is the pullback of $\JackjLR_{\mu\nu;\lambda} := \bf{j}_\lambda \cdot \JackjLR_{\mu\nu}^{\lambda}$ to the hook space.
\[ \stanleypoly:= \iota^* \JackjLR_{\mu\nu;\lambda} \in \StDR_{\mu\nu}^{\lambda}(\windowfamily). \]
\end{definition} 

The objective of this paper is to demonstrate that the polynomial $\stanleypoly$ has properties that are not manifest in the Stanley sum $\JackjLR_{\mu\nu;\lambda}$.

\subsection{Particular window family}

The solution \eqref{eq:generalsolution} was computed for the following window family with root $\{21,21,321\}$.

\begin{example}
Consider the union of the two 7-parameter families of windows.
\ytableausetup{boxsize=1.1em}
\begin{equation}\label{eq:7paramwindowronly}
\windfact_{W_{12|3}} = \frac{\begin{ytableau}
\none[m_1\,]  &*(boxU)  & \none[1]  \\
\none   &*(boxS) a_1 &*(boxL) \\
\none[r_1]  &*(boxL)  & *(boxL) \\
\none[m_2\,]   &*(boxU)  & *(boxG)  & *(boxU) & \none[2] \\
\none    &*(boxS) a_2 &*(boxL)  &*(boxS)  a_3& *(boxL)   \\
\none[r_2] &*(boxL) & *(boxL)   &*(boxL)  & *(boxL)  \\
\none  &\none & \none[n_1]  & \none & \none[n_2]
\end{ytableau} \,\, \begin{ytableau}
\none \\
\none[m_3\,] & *(boxU) & \none[3] \\
\none & *(boxS) b_1  \\
\none[r_1] & *(boxL)  \\
\none[r_2] & *(boxL)  \\
\none &  *(boxS) b_2 & *(boxS) b_3  \\
\none & \none  & \none 
\end{ytableau}
}{\begin{ytableau}
\none & *(boxS) c_1\\
\none[m_3\,] & *(boxU) & \none[3] \\
\none[m_1\,] & *(boxU)  & \none[1]  \\
\none[r_1] & *(boxL) & *(boxL) & *(boxL) \\
\none & *(boxS) c_2 &*(boxL)  & *(boxS) c_3  \\
\none[m_2\,] & *(boxU) & *(boxG) &*(boxU) &\none[2]   \\
\none[r_2] & *(boxL)&*(boxL)  &*(boxL) & *(boxL)   & *(boxL)  \\
\none & *(boxS) c_4&*(boxL)  &*(boxS)  c_5 & *(boxL)  & *(boxS) c_6 \\
\none & \none  & \none[n_1] & \none  & \none[n_2] & \none
\end{ytableau}}\qquad
 \windfact_{W_{12|34}} = \frac{\begin{ytableau}
\none[m_1\,]  &*(boxU)  & \none[1]  \\
\none   &*(boxS) a_1 &*(boxL) \\
\none[r_1]  &*(boxL)  & *(boxL) \\
\none[m_2\,]   &*(boxU)  & *(boxG)  & *(boxU) & \none[2] \\
\none    &*(boxS) a_2 &*(boxL)  &*(boxS)  a_3& *(boxL)   \\
\none  &\none & \none[n_1]  & \none & \none[n_2] 
\end{ytableau} \,\, \begin{ytableau}
\none \\
\none[m_3\,\,] & *(boxU) & \none[3] \\
\none & *(boxS) b_1  \\
\none[r_1] & *(boxL)  & \none & \none[4]  \\
\none &  *(boxS) b_2  & *(boxS) b_3 & *(boxL) \\
\none & \none   & \none & \none[n_4]
\end{ytableau}
}{\begin{ytableau}
\none & *(boxS) c_1\\
\none[m_3\,\,] & *(boxU)& \none[3] \\
\none[m_1\,] & *(boxU)  & \none[1]  \\
\none[r_1] & *(boxL) & *(boxL) & *(boxL) \\
\none & *(boxS) c_2 &*(boxL)  & *(boxS) c_3  \\
\none[m_2\,\,] & *(boxU) & *(boxG) &*(boxU) &\none[2]&\none[4]   \\
\none & *(boxS) c_4&*(boxL)  &*(boxS)  c_5 & *(boxL)  &  *(boxL)  & *(boxS) c_6 \\
\none & \none  & \none[n_1] & \none  & \none[n_2] & \none[n_4] & \none 
\end{ytableau}}
\end{equation}

The hook lengths for the triple $\{21,21,321\}$ in the families \eqref{eq:7paramwindowronly} are given by (with $r_2n_4=0$): 
\begin{eqnarray}\label{lemma:hooklengths}
h_{a_1}^\sU: && m_1 + (n_1+1)\al  \nonumber \\
h_{a_2}^\sU: && m_1+m_2+r_1+1+(n_1+n_2+2)\al \nonumber\\
h_{a_3}^\sU: && m_2+(n_2+1)\al \nonumber\\
h_{b_1}^\sU: && m_3+\al \nonumber\\
h_{b_2}^\sU: && m_3+r_1+r_2+1+(n_4+2)\al \nonumber\\
h_{b_3}^\sU: && (n_4+1)\al \\
h_{c_1}^\sU: && \al \nonumber\\
h_{c_2}^\sU: && m_1+m_3+r_1+1+(n_1+2)\al \nonumber\\
h_{c_3}^\sU: && r_1+\al \nonumber\\
h_{c_4}^\sU: && m_1+m_2+m_3+r_1+r_2+2+(n_1+n_2+n_4+3)\al \nonumber\\
h_{c_5}^\sU: && m_2+r_1+r_2+1+(n_2+n_4+2)\al \nonumber\\
h_{c_6}^\sU: && r_2+\al. \nonumber
\end{eqnarray}
For the virtual hook $\tilde b_3$ motivated by \eqref{eq:tildeb3def}, we define
\[ h_{\tilde b_3}^\sU := \alpha^{-1} h_{b_3}^\sU h_{c_1}^\sL h_{c_6}^\sU \equiv r_2+ (n_4+1)\al   \mod r_2 n_4.\]
\end{example}

In particular, for our special rule $\bigstanleysum$ \eqref{eq:generalsolution} for this family \eqref{eq:7paramwindowronly}, we define $\bigstanleypoly:=\iota^* \bigstanleysum'$. We analyse the structure of this particular hook space by finding an abundance of four-term relations between the various hooks. 

\begin{example}
Consider the vertex $a_2$,  adjacent to $\{a_1,a_3, c_3\}$, we find
\[ -h^U_{a_2} + h^U_{a_1}+h^U_{a_3}+h^U_{c_3}  -\beta = 0. \]

\end{example}
We write $a\sim b$ for two boxes that are adjacent in the Petersen graph \eqref{diag:petersongraph}.
Let $x_b := (-1)^{\bmX_b/\sU}$, where $\bmX$ is given by \eqref{eq:Xdef}. Thus
\[ \bmX =\prod_b \bh_b^{\bmX_b} = \prod_b \bh_b^{x_b \sU} \]
\begin{proposition}
For each $b$ and each window in the family $W \in \windowfamily$ \eqref{eq:7paramwindowronly}, the relation $[\bmf_b]_W=0$ holds, where
\begin{equation}\label{eq:fourtermkernel}
\bmf_b := -x_b \bh_b^{\bmX_b} + \sum_{a:a\sim b} x_a \bh_a^{\bmX_a} -\beta.
\end{equation}
The ten relations $\bmf_b$ span $\StDR_{21,21;321}(\windowfamily)$, which is $5$-dimensional.
\end{proposition}
\begin{proof}
We verify this by direct calculation, referencing the table \eqref{lemma:hooklengths}. 

To show the spanning property, we use the claw relations to express the five hooks $\{a_2, b_2, c_2, c_4, c_5\}$ in terms of the five independent hooks $\{a_1, a_3, b_1, \tilde b_3, c_3\}$ and $\beta$. These five are clearly linearly independent. The relations at $\bmf_{a_2}$ and $\bmf_{b_2}$ immediately give
\begin{eqnarray*}
h^\sU_{a_2} &=& h^\sU_{a_1} + h^\sU_{a_3} + h^\sU_{c_3} - \beta,\\
h^\sU_{b_2} &=& h^\sU_{b_1} + h^\sU_{c_3} + h^\sU_{\tilde b_3} - \beta.
\end{eqnarray*}
Similarly, the remaining three are resolved from the relations:
\begin{eqnarray*}
h^\sU_{c_5} &=& h^\sU_{a_3} + h^\sU_{c_3} + h^\sU_{\tilde b_3} - \beta,\\
h^\sU_{c_4} &=& h^\sU_{a_1} + h^\sU_{a_3} + h^\sU_{b_1} + h^\sU_{c_3} + h^\sU_{\tilde b_3} - 2\beta,\\
h^\sU_{c_2} &=& h^\sU_{a_1} + h^\sU_{b_1} + h^\sU_{c_3} - \beta.
\end{eqnarray*}
Thus all ten hooks are determined by $\{h^\sU_{a_1}, h^\sU_{a_3}, h^\sU_{b_1}, h^\sU_{\tilde b_3}, h^\sU_{c_3}\}$ and $\beta$, so $\StDR_{21,21;321}(\windowfamily)$ is at most $5$-dimensional. Since the five independent hooks are unconstrained, the dimension is exactly $5$.
\end{proof}

The above result is a more concrete manifestation of the Petersen graph, as it captures the relations between the hooks.

\subsection{Summary}

The primary object of our attention is the particular Jack
Littlewood-Richardson polynomial that lives in the Hook Space,
\begin{equation}
\bigstanleypoly:=\iota^* \bigstanleysum \in \StDR_{21,21}^{321}(\windowfamily).
\end{equation}

\section{Symmetries}

We will also consider maps between hook rings. We start with automorphisms.

\begin{definition}\label{def:hookmap}
A \emph{\bf{hook automorphism}} $\psi \in \hkAut(\StDR)$
of a hook space $\StDR$ is a signed map between hook variables $\psi:\bh_b^\sA \mapsto \pm \bh_{b'}^{\sB}$ that preserves all kernel relations. Given the relations \eqref{betarelation}, for each $b\in \sigma, \sA \in \{\sU,\sL\}$ such a map must be of the form
\[ \rho : \bh^\sA_{\sigma}(b) \mapsto \bh^\sA_{\sigma'}(b') \]
or 
\[\rho :\bh^\sA_{\sigma}(b) \mapsto -\bh^{\bar\sA}_{\sigma'}(b'). \]

\end{definition}

\begin{lemma}
The action of $\Aut(\petersen) = S_5$ on boxes from Observation \ref{lemma:petersengraph} lifts to hook automorphisms given by
\begin{equation}
 g: \bh_b^\sA \mapsto x_{g.b/b} \bh_{g.b}^{x_{g.b/b}\sA},
\end{equation}
That is, 
\[ \Aut(\petersen) \subset \hkAut(\StDR_{21,21}^{321}(\windowfamily)).\]
\end{lemma}
\begin{proof}
We directly check that this $S_5$ action on $\bh_b$ permutes the relations \eqref{eq:fourtermkernel}, i.e. $g.\bmf_b = \bmf_{g.b}$, and thus descends to an action on the hook space.
\end{proof}

A set of generators of this action, each of which flip three or four pairs of boxes, are

\[ R_1 := 
\frac{\begin{ytableau}
a_1 \\
a_2 & a_3
\end{ytableau}\,\, \begin{ytableau}
\color{pGray}{b_1} \\
\color{pGray}{b_2} & \color{pGray}{b_2}
\end{ytableau} 
}{\begin{ytableau}
*(boxB)  \\
c_2   & \color{pGray}{c_2} \\
c_4 & c_5 & *(boxB) \nonumber\\
\end{ytableau}} 
\to 
\frac{\begin{ytableau}
*(boxS)\bar c_2 \\
*(boxS)\bar c_4 & *(boxS) \bar c_5
\end{ytableau}\,\, \begin{ytableau}
\color{pGray}{b_1} \\
\color{pGray}{b_2} & \color{pGray}{b_3}
\end{ytableau} 
}{\begin{ytableau}
*(boxB)  \\
*(boxS)\bar a_1   & \color{pGray}{c_2} \\
*(boxS)\bar a_2  & *(boxS)\bar a_3  & *(boxB) \nonumber\\
\end{ytableau}} 
\]

\[ R_2 := 
\frac{\begin{ytableau}
\color{pGray}{a_1} \\
\color{pGray}{a_2} & \color{pGray}{a_3}
\end{ytableau}\,\, \begin{ytableau}
b_1 \\
b_2 & b_3
\end{ytableau} 
}{\begin{ytableau}
*(boxB)  \\
c_2   & \color{pGray}{c_3} \\
c_4 & c_5 & *(boxB) \nonumber\\
\end{ytableau}} 
\to 
\frac{\begin{ytableau}
\color{pGray}{a_1} \\
\color{pGray}{a_2} & \color{pGray}{a_3}
\end{ytableau}\,\, \begin{ytableau}
*(boxS)\bar c_2 \\
*(boxS)\bar c_4 & *(boxS) \bar c_5
\end{ytableau} 
}{\begin{ytableau}
*(boxB)  \\
*(boxS)\bar b_1   & \color{pGray}{c_3} \\
*(boxS)\bar b_2  & *(boxS)\bar b_3  & *(boxB) \nonumber\\
\end{ytableau}} 
\]

\[ T := 
\frac{\begin{ytableau}
a_1 \\
\color{pGray}{a_2} & a_3
\end{ytableau}\,\, \begin{ytableau}
b_1 \\
\color{pGray}{b_2} & b_3
\end{ytableau} 
}{\begin{ytableau}
*(boxB)  \\
c_2   & \color{pGray}{c_3} \\
\color{pGray}{c_4} & c_5 & *(boxB) \nonumber\\
\end{ytableau}} 
\to 
\frac{\begin{ytableau}
a_3 \\
\color{pGray}{a_2} & a_1
\end{ytableau}\,\, \begin{ytableau}
b_3 \\
\color{pGray}{b_2} & b_1
\end{ytableau} 
}{\begin{ytableau}
*(boxB)  \\
c_5  & \color{pGray}{c_3} \\
\color{pGray}{c_4} & c_2 & *(boxB) \nonumber\\
\end{ytableau}} 
\]

\[ R_3 := 
\frac{\begin{ytableau}
a_1 \\
a_2 & \color{pGray}{a_3}
\end{ytableau}\,\, \begin{ytableau}
\color{pGray}{b_1} \\
b_2 & b_3
\end{ytableau} 
}{\begin{ytableau}
*(boxB)  \\
c_2   & c_3 \\
c_4   & c_5 & *(boxB) \nonumber\\
\end{ytableau}} 
\to 
\frac{\begin{ytableau}
b_3 \\
*(boxS)\bar c_2 & \color{pGray}{a_3}
\end{ytableau}\,\, \begin{ytableau}
\color{pGray}{b_1} \\
*(boxS) \bar c_5 & a_1
\end{ytableau} 
}{\begin{ytableau}
*(boxB)  \\
*(boxS)\bar a_2   & *(boxS)\bar c_4 \\
*(boxS)\bar c_3 & *(boxS)\bar b_2 & *(boxB) \nonumber\\
\end{ytableau}} 
\]

\begin{lemma}
The Stanley sum $\bigstanleysum'$ is invariant under this $S_5$ action, and thus so is the Jack LR polynomial $\bigstanleypoly:=\iota^* \bigstanleysum'$.
\end{lemma}
\begin{proof}
We look at each of the terms in the expression \eqref{eq:721sum}. Firstly, we show that $\bmX$ is invariant under this action.

\[ g.\bmX = g.\prod_b \bh_b^{x_b \sU} = \prod_b x_{g.b/b}\bh_b^{x_{g.b} \sU} = \left( \prod_b x_{g.b/b} \right)\bmX   \]

As each generator of $S_5$ flips hooks in pairs, we get $g.\bmX = (-1)^{2k}\bmX$. Similarly, $g.\bmX^{\{b\}} = \bmX^{\{g.b\}}$.
\end{proof}

Thus we can write the expression \eqref{eq:721sum} as sum of three orbits of this $S_5$ action
\begin{equation}\label{eq:s5orbitsum}
 \bigstanleysum = 7\CO_{\emptyset} -2 \CO_{\bullet} + 1 \CO_{\gPathTwo}.
 \end{equation}
where the subscript indicates the subgraph structure of $\petersen$ that has flipped hooks (relative to $\bmX$). Thus this Petersen structure is manifest in the Stanley sum $ \bigstanleysum'$ \eqref{eq:721sum}. However, we will find that the Stanley \emph{polynomial} $\bigstanleypoly:=\iota^* \bigstanleysum'$ has a \emph{non-manifest enlargement} of this Automorphism group to $S_6 \times \BZ_2$, which is the Automorphism group of the Johnson graph $J(6,3)$. To demonstrate this, we need a change of basis in the hook space.

\subsection{$\ell$ basis}


We consider the change of variables that is more suited to the Petersen symmetry.

\begin{definition}
For a fixed Stanley diagram $\bmD$, we define the $\ell$ basis relative to $\bmD$ as \
\begin{equation}
\ell_b := (-1)^{\bmD_b/\sU} (\bh_b^\sU+\bh_b^\sL) = (-1)^{\bmD_b/\sU}  (2\bh_b^\sA-(-1)^{\sA/\sU}\beta).
\end{equation}
If not mentioned, then we take the reference diagram $\bmD$ to be all upper hooks, $\bmD_b = \sU, \forall b$.
\end{definition}

When considering the main triple $\{21,21,321\}$, we work relative to the diagram $\bmX$, defined in \eqref{eq:Xdef}. We can express the hook variables as
\begin{equation}\label{eq:varsell321}
 x_b \bh_b^\sA = \tfrac{1}{2}(\ell_b+x_b(-1)^{\sA/\sU}\beta).
 \end{equation}

\begin{corollary}
With the change of coordinates \eqref{eq:varsell321}, the action of $S_5$ on the $\ell$ variables is the \emph{permutation representation} $\petvertrep := \oplus_{b\in \petersen} \BC\{\ell_b\}$ of the Petersen graph, 
\[ g.\ell_b = \ell_{g.b}.\]
The four-term hook relations \eqref{eq:fourtermkernel} are expressed as the \emph{\bf{claw}} relations
\begin{equation}\label{def:lkernel}
\bmf_b := \ell_b - \sum_{a\sim b} \ell_a= (1\text{-}A)\ell_b,
\end{equation}
where $A$ is the adjacency matrix of the Petersen graph $\petersen$,
\[ A.\ell_b :=  \sum_{a:a\sim b} \ell_a. \]
\end{corollary}

Thus, a function $F$ on $\petersen$ is in the linear span of the claw relations \eqref{def:lkernel} if $F = (1\text{-}A)G$, for another function $G$.

Any graph automorphism group preserves edges and so naturally commutes with the adjacency matrix. The spectrum of $A$ is well known (c.f. \cite[\S 9.1]{godsil01}), and the action of $S_5$ on the $\ell$ basis commutes with the action of the adjacency matrix $A$, and decomposes as $\mathrm{Im}(1\text{-}A)\oplus \ker(1\text{-}A)$ 

\begin{equation}\label{eq:petersenrepdecomp}
\petvertrep  = \left( V^{(3)}_{\{5\}} \oplus V^{(-2)}_{\{4,1\}}\right) \oplus V^{(1)}_{\{3,2\}}
\end{equation}
Where $V^k_\lambda$ carries $A$-eigenvalue $k$ and $S_5$-representation corresponding to the partition $\lambda$.
Here $V^{(3)}_{\{5\}}$ is the one-dimensional representation generated by 
\begin{equation}
\sum_b \bmf_b = -2\sum_b \ell_b.
\end{equation}

\begin{corollary}
The hook space of the triple $\{21,21,321\}$ is 
\begin{equation}\label{eq:petersenhookspace}
 \StDR_{21,21}^{321} \cong \Sym^{10}(V^{(1)}_{\{3,2\}}[\bbeta]).
\end{equation}
\end{corollary}
This simple algebraic description justifies the definition of the hook space. We remind the reader of the distinction between two polynomials,
\[ \bigstanleysum' \in \Sym^{10}(V_{\petersen}[\bbeta])^{S_5}, \qquad \bigstanleypoly \in \Sym^{10}(V^{(1)}_{\{3,2\}}[\bbeta])^{S_5} \]
We will demonstrate that $\bigstanleypoly$ has additional symmetries that are not shared with $\bigstanleysum'$.

\subsection{Conjugation Symmetry}
We've noted that the hook space $\StDR_{21,21}^{321}$ has a $S_5=\Aut(\petersen)$ symmetry. In addition to this, we observe all the relations \eqref{def:lkernel} are linear homogeneous in the $\ell$ variables.
\begin{corollary}
The hook space has the additional $\BZ_2$ conjugation symmetry 
\begin{equation}\label{eq:conjoperation}
 C: \ell_b \to -\ell_b.
\end{equation}
Thus, the hook automorphism group contains $S_5 \times \BZ_2 \subseteq \hkAut(\StDR_{21,21}^{321})$.
\end{corollary}

That is, flipping all boxes is a symmetry of the relations \eqref{eq:fourtermkernel}.
We will spend considerable effort proving the following theorem.
\begin{theorem}\label{thm:z2invs}
The Jack LR polynomial $\bigstanleypoly \in \StDR_{21,21}^{321}$ is invariant under this additional $\BZ_2$.
\end{theorem}
It is manifestly clear that $\bigstanleysum'$ does not have this symmetry, as $7\bmX \in \bigstanleysum$, however the coefficient of $\bar \bmX$ is zero.

We prove Theorem \eqref{thm:z2invs} via an equivalent statement \eqref{thm:oddorbitsumvanish} in a later section. First, we need to expand the Stanley sum $\bigstanleysum'$ in the $\ell$ basis.

\subsection{Orbit sums}

In this section we produce an expression of the form 
\begin{equation}
\bigstanleysum' = \sum_I \beta^{|I|} w_I \ell_{\allvertices/I},
\end{equation}
in terms of the monomials $\ell_J = \prod_{j\in J} \ell_j$, for $J$ a subset of the vertices of $\allvertices$.
In order to do this, we introduce some intermediate constructions.

\subsubsection{$K$-values}

It was noted in \cite{Mickler:2024stanley} that for a rule $\JackjLR_{\mu\nu}^{\lambda} = \sum_\bmD c_\bmD \, \bmD$, the sum of all the coefficients $c_\bmD$, henceforth denoted $K_\emptyset$, recovers the Schur Littlewood-Richardson coefficient. That is
\begin{eqnarray}
K_\emptyset(\JackjLR_{\mu\nu}^{\lambda}) &:=& \sum_{\bmD} c_{\bmD} \\
&=& c_{\mu\nu}^{\lambda}.
\end{eqnarray}
In the case of our Stanley sum \eqref{eq:721sum} we have
\[ K_\emptyset(\bigstanleysum) = 7(1)-2(10)+(15) = 2 = c_{21,21}^{321}. \]
We situate $K_\emptyset$ as the first term in a sequence of integers, computed via a linear transform.
\begin{definition} Fix a choice of a reference Stanley diagram $\bmB$. For each subset $I$ of the set of boxes $\allvertices$, define the $K$\emph{\bf{-values}} of a Stanley sum $\bmss = \sum_\bmD c_\bmD \, \bmD$ as
\begin{equation}\label{eq:ktransform} 
K_I^\bmB(\bmss) :=  \sum_{\bmD\,:\,\bmD_I = \bmB_I} c_\bmD.
\end{equation}
That is, sum all the coefficients $c_\bmD$ for those $\bmD$ that agree with the reference $\bmB$ on the subset $I$.
In this way, we map the integer coefficients $\{c_\bmD\}$ to the set of integer coefficients $\{K^\bmB_I\}$ indexed by the subsets of the set of boxes.
\end{definition}

We need to fix some notation before progressing. We use $\bar\bmB$ to denote the diagram with all flipped hooks of $\bmB$. Let $\bmB^J$ be the diagram with boxes $b\in J$ flipped relative to $\bmB$:
\begin{equation}\label{eq:flippeddiagram}
\bmB^J := \prod_{b \notin J} \bh_b^{\bmB_b} \prod_{b \in J} \bh_b^{\bar\bmB_b}.
\end{equation}
Note $\bar\bmB = \bmB^\allvertices$.  
Let $\bmB \wedge \bmD = \{ i \in \allvertices : \bmD_i = \bmB_i \}$ be the set of boxes where $\bmD$ agrees with $\bmB$. I.e. $ \bmB\wedge \bar\bmB^J  = J$.
Let $\bmB_I$ be just the subset hook choices associated to the subset of boxes $I \subseteq \allvertices$.
We introduce the sign notation
\begin{equation}
(-1)^{\bmA_{I}/\bmB_I} := (-1)^{|\{i \in I : \bmA_i \neq \bmB_i\}|}
\end{equation}
Recall the following basic result.
\begin{lemma}[M\"obius Inversion - c.f \protect{\cite[\S 3.8.3]{stanley1999enumerativev1}}]\label{lemma:mobinv}
Let $H$ be a finite set, and let $f, g : \mathcal{P}(H) \to \mathbb{R}$ be functions on the power set of $H$. If $g$ is the sum of $f$ over supersets:
\[ g(I) = \sum_{J \supseteq I} f(J), \]
then $f$ can be recovered from $g$ via the alternating sum:
\[ f(I) = \sum_{J \supseteq I} (-1)^{|J\setminus I|} g(J). \]
\end{lemma}
We use this to show
\begin{proposition}
For each Stanley diagram $\bmB$, the values 
\[K^\bmB : \{c_\bmD : \bmD \in \mathcal{H} \} \to \{K^\bmB_{I} : I \subseteq \allvertices \}\] is invertible. 
The Inverse map is given by
\begin{equation}\label{eq:inversektransform} 
c_{\bmD} = \sum_{ J \supseteq \bmB \wedge \bmD } (-1)^{|J| - |\bmB \wedge \bmD|} K_{J}^{\bmB}
\end{equation}
\end{proposition}

\begin{proof}
We have
\[ K_I^\bmB = \sum_{\bmD : \bmB \wedge \bmD \supseteq I} c_\bmD = \sum_{J:J \supseteq I} c_{\bar\bmB^J}, \] 
Applying Lemma \eqref{lemma:mobinv} with $g(I) = K_I^\bmB$ and $f(J) = c_{\bar\bmB^J}$ yields the result. I.e.
\[ c_{\bar\bmB^I} =  \sum_{J \supseteq I} (-1)^{|J\setminus I|}K^B_J.\]
\end{proof}
The dependence of $K_I$ on $\bmB$ is expressed in the following Lemma. 
\begin{lemma}\label{lemma:deponB}
Let $\bmA^b$ be equal to $\bmA$ but with a single hook at $b$ flipped, then
\begin{equation}  K_{I}^{\bmA^b} = \begin{cases}
- K_{I}^{\bmA} + K_{I\setminus b}^{\bmA}, \quad \text{ for } b \in I \\
 K_{I}^{\bmA}, \quad\quad\quad\quad\quad \text{ for } b \notin I \\
\end{cases} \end{equation}
Thus, the values $\{K_{J}^{\bmA}\}$ can be determined from $\{K_{I}^{\bmB}\}$ by the formula
\begin{equation}\label{eq:Kinversion}
K_{J}^{\bmA} =\sum_{I \subseteq J \,: \,J\setminus I \subseteq \bmA/\bmB}   (-1)^{\bmA_{I}/\bmB_{I}} K_{I}^{\bmB},
\end{equation}
where
\[ \bmA/\bmB := \{i: \bmA_i \neq \bmB_i \}.\]
\end{lemma}
\begin{proof}
Then by using the definition \eqref{eq:ktransform}, for $I$ with $b \in I$, we have
\[ K_{I\setminus b}^{\bmA} = K_{I}^{\bmA} + K_{I}^{\bmA^b} \]
So
\[ K_{I}^{\bmA^b} = - K_{I}^{\bmA} + K_{I\setminus b}^{\bmA} \]
However, for $I$ with $b \notin I$, we have
\[ K_{I}^{\bmA^b} = K_{I}^{\bmA} \]
The result follows from multiple single flips.
\end{proof}

\begin{corollary}\label{cor:inversek}
A Stanley sum $\bmss$ can be expressed in terms of its $K$-values as
\begin{eqnarray}\label{eq:inversek}
\bmss &=& \sum_{\bmD} \left( \sum_{ I \supseteq \bmB \wedge \bmD } (-1)^{|I| - |\bmB \wedge \bmD|} K_{I}^{\bmB}(\bmss) \right) \bmD\\
&=& \sum_{J \subseteq \allvertices}\left( \sum_{I\supseteq J}(-1)^{|I\setminus J|}  K_{I}^{\bmB}(\bmss) \right) \bar \bmB^J.
\end{eqnarray}
\end{corollary}

\subsubsection{Kernel functions}

\begin{theorem}[Kernel functions]\label{thm:kernelidentity}
Fix a Stanley diagram $\bmB$ and a linear polynomial in hook variables and $\beta$,
\[ \bmf := \sum_{b \in Y} a_b \bh_b^{\bmB_b} + d\,\beta, \qquad Y \subseteq \allvertices. \]
Set $\sigma_b := (-1)^{\bmB_b/\sU}$ ($+1$ for $\bmB_b=\sU$, $-1$ for $\bmB_b=\sL$), encoding the $\beta$-shift $\bh_b^{\bar\bmB_b} = \bh_b^{\bmB_b} + \sigma_b\beta$, and define the \emph{kernel function} as the $Y$-supported Stanley sum
\begin{equation}\label{eq:kernelsumformula}
 \bmk_\bmf^\bmB \;:=\; \sum_{L\subseteq Y} c_L\,\bmB^L, \qquad c_L := (-1)^{|L|}\Bigl(d - \sum_{b\in Y\setminus L}\sigma_b a_b\Bigr).
\end{equation}
Then, the following relation holds as polynomials in the variables $\bh,\beta$:
\begin{equation}\label{eq:kernelidentity}
\bmk_\bmf^\bmB = (-1)^{|Y|}\,\sigma_Y\,\bh_{\bar Y}^\bmB \cdot \beta^{|Y|-1} \cdot \bmf,
\end{equation}
where $\sigma_Y = \prod_{b\in Y}\sigma_b$ and $\bh_{\bar Y}^\bmB = \prod_{b\in\bar Y}\bh_b^{\bmB_b}$.

In particular, if $[\bmf] = 0$ in the evaluation ring, then $[\bmk_\bmf^\bmB]=0$ as well.
\end{theorem}

Its $K$-values follow from the definition: since $K^\bmB_I(\bmB^L) = [I\cap L = \emptyset]$, only $L\subseteq Y\setminus I$ contribute, and the resulting alternating sum vanishes for $|Y\setminus I|\ge 2$. Concretely,
\begin{equation}\label{eq:kernelKvalues}
 K^\bmB_I(\bmk_\bmf^\bmB) = \begin{cases} d - \sum_{b\in Y}\sigma_b a_b & I\supseteq Y,\\ -\sigma_b a_b & I\cap Y = Y\setminus\{b\}\ (b\in Y),\\ 0 & \text{otherwise}.\end{cases}
\end{equation}

\begin{proof}[Proof of Theorem~\ref{thm:kernelidentity}]
Write $\bmk := \bmk_\bmf^\bmB$.

Using $\bh_b^{\bar\bmB_b} = \bh_b^{\bmB_b} + \sigma_b\beta$, each diagram $\bmB^L$ ($L\subseteq Y$) factors as
\[ \bmB^{L} = \bh_{\bar Y}^\bmB \cdot \prod_{b\in Y\setminus L} \bh_b^{\bmB_b} \cdot \prod_{b\in L}(\bh_b^{\bmB_b} + \sigma_b \beta). \]

Expanding the $L$-product:
\[ \prod_{b\in L}(\bh_b^{\bmB_b} + \sigma_b\beta) = \sum_{T\subseteq L} \beta^{|T|}\sigma_T \prod_{b\in L\setminus T} \bh_b^{\bmB_b}, \]
so
\[ \frac{\bmk}{\bh_{\bar Y}^\bmB} = \sum_{L\subseteq Y} c_{L} \sum_{T\subseteq L} \beta^{|T|}\sigma_T \prod_{b\in Y\setminus T} \bh_b^{\bmB_b} = \sum_{T\subseteq Y} \beta^{|T|}\sigma_T\,\tilde c^{(T)}\prod_{b\in Y\setminus T} \bh_b^{\bmB_b}, \]
where we swapped the order of summation and collected the sum $\tilde c^{(T)} = \sum_{L\supseteq T} c_{L}$.

By Lemma~\ref{lemma:mobinv} on the power set of $Y$, the sum $\tilde c^{(T)}$ is uniquely characterised by the inverse relation
\[ c_{L} = \sum_{T\supseteq L}(-1)^{|T\setminus L|}\,\tilde c^{(T)}. \]
We claim:
\begin{equation}\label{eq:tildec-claim}
\tilde c^{(T)} = \begin{cases} 0 & |T|\le |Y|-2,\\ (-1)^{|Y|}\,\sigma_b a_b & T = Y\setminus\{b\},\\ (-1)^{|Y|}\,d & T = Y. \end{cases}
\end{equation}
By the uniqueness of Möbius inversion, it suffices to verify that \eqref{eq:tildec-claim} satisfies the inverse relation, i.e.\ reproduces $c_L$. For $L\subseteq Y$, only $T = Y$ and $T = Y\setminus\{b\}$ for $b\notin L$ contribute:
\begin{align*}
\sum_{T\supseteq L}(-1)^{|T\setminus L|}\tilde c^{(T)}
&= (-1)^{|Y\setminus L|}(-1)^{|Y|}d + \sum_{b\notin L}(-1)^{|Y\setminus L|-1}(-1)^{|Y|}\sigma_b a_b\\
&= (-1)^{|L|}\Bigl(d - \!\!\sum_{b\in Y\setminus L}\!\sigma_b a_b\Bigr) = c_L.
\end{align*}
Hence only $|T|\in\{|Y|-1,|Y|\}$ contribute to $\bmk/\bh_{\bar Y}^\bmB$, giving
\[ \beta^{|Y|-1}\!\sum_{b\in Y}\sigma_{Y\setminus\{b\}}(-1)^{|Y|}\sigma_b a_b\,\bh_b^{\bmB_b} + \beta^{|Y|}\sigma_Y(-1)^{|Y|}d = (-1)^{|Y|}\sigma_Y\beta^{|Y|-1}\Bigl(\sum_{b\in Y}a_b\bh_b^{\bmB_b} + d\,\beta\Bigr), \]
using $\sigma_{Y\setminus\{b\}}\sigma_b = \sigma_Y$. The bracket is $\bmf$, proving~\eqref{eq:kernelidentity}.
\end{proof}

\begin{lemma}\label{cor:clawkernel}
For the claw relations \eqref{eq:fourtermkernel} with $Y = \claw(b) = \{b\}\cup\{a : a\sim b\}$ and any $\bmB$ with $\bmB_Y = \bmX_Y$, the kernel sum $\bmk^\bmB_{\bmf_b}$ is given by the ten-term Stanley sum
\begin{equation}\label{def:kernelfunction}
 \bmk^\bmB_{\bmf_b} \;:=\; k_f^0 - (k_f^0)^{\claw(b)}, \qquad k_f^0 \;:=\; -2\bar\bmB^{\{b\}}+\sum_{a \sim b} \bar\bmB^{\{b,a\}} -\bar\bmB^{Y}.
\end{equation}
\end{lemma}
\begin{proof}
Apply Theorem~\ref{thm:kernelidentity} with $Y = \claw(b)$. Under the natural frame $\bmB_Y = \bmX_Y$, the claw relation \eqref{eq:fourtermkernel} gives $\sigma_b a_b = +1$ and $\sigma_a a_a = -1$ for $a\sim b$, with $d = -1$. Evaluating the coefficients \eqref{eq:kernelsumformula} on each $L\subseteq Y$, write $\epsilon := [b\in L]$ and $k := |\{a\sim b : a\notin L\}|$:
\[ \sum_{b'\in Y\setminus L}\!\sigma_{b'}a_{b'} = (1-\epsilon) - k, \qquad c_L = (-1)^{|L|}(\epsilon + k - 2). \]
Hence $c_L = 0$ exactly when $\epsilon + k = 2$: the three leaf-singletons $L = \{a\}$ ($\epsilon=0,k=2$) and the three triples $L = Y\setminus\{a\}$ ($\epsilon=1,k=1$), $a\sim b$. The remaining ten $L\subseteq Y$ carry the coefficients
\[ c_\emptyset = +1,\quad c_{\{b\}} = -2,\quad c_{\{b,a\}} = +1,\quad c_{\{a,a'\}} = -1,\quad c_{Y\setminus\{b\}} = +2,\quad c_Y = -1, \]
which assemble as $k_f^0 - (k_f^0)^{\claw(b)}$ in~\eqref{def:kernelfunction}. 
\end{proof}

\begin{example}
We may take $\bmB = \bmX$ and for $b = a_2$ with kernel sum 
\[ \bmf_{a_2} = h^\sU_{a_2} - h^\sU_{a_1} - h^\sU_{a_3} - h^\sL_{c_3}, \]
the kernel Stanley sum \eqref{def:kernelfunction} reduces to the $10$ non-zero terms shown below
\begin{eqnarray*}
\ytableausetup{boxsize=0.6em}
\bmk_{\bm{f}_{a_2}}^\bmX &=&
- 2\cdot
\frac{ \begin{ytableau}
*(boxU)  \\
*(boxL)    &  *(boxU) 
\end{ytableau}\,\,\begin{ytableau}
  \\
  &  
\end{ytableau} }{ \begin{ytableau}
 \\
    &  *(boxL) \\
   &    &  
\end{ytableau} }
+\,1 \cdot \frac{ \begin{ytableau}
*(boxL)  \\
*(boxL)    &  *(boxU) 
\end{ytableau}\,\,\begin{ytableau}
  \\
  &  
\end{ytableau} }{ \begin{ytableau}
 \\
    &  *(boxL) \\
   &    &  
\end{ytableau} }  
+\,1 \cdot \frac{ \begin{ytableau}
*(boxU)  \\
*(boxL)    &  *(boxL) 
\end{ytableau}\,\,\begin{ytableau}
  \\
  &  
\end{ytableau} }{ \begin{ytableau}
 \\
    &  *(boxL) \\
   &    &  
\end{ytableau} }  
+1 \cdot
\frac{ \begin{ytableau}
*(boxU)  \\
*(boxL)    &  *(boxU) 
\end{ytableau}\,\,\begin{ytableau}
  \\
  &  
\end{ytableau} }{ \begin{ytableau}
 \\
    &  *(boxU) \\
   &    &  
\end{ytableau} }
-1 \cdot \frac{ \begin{ytableau}
*(boxL)  \\
*(boxL)    &  *(boxL) 
\end{ytableau}\,\,\begin{ytableau}
  \\
  &  
\end{ytableau} }{ \begin{ytableau}
 \\
    &  *(boxU) \\
   &    &  
\end{ytableau} } \\
&&
+2 \cdot
\frac{ \begin{ytableau}
*(boxL)  \\
*(boxU)    &  *(boxL) 
\end{ytableau}\,\,\begin{ytableau}
  \\
  &  
\end{ytableau} }{ \begin{ytableau}
 \\
    &  *(boxU) \\
   &    &  
\end{ytableau} }
-1 \cdot \frac{ \begin{ytableau}
*(boxU)  \\
*(boxU)    &  *(boxL) 
\end{ytableau}\,\,\begin{ytableau}
  \\
  &  
\end{ytableau} }{ \begin{ytableau}
 \\
    &  *(boxU) \\
   &    &  
\end{ytableau} }
-1 \cdot \frac{ \begin{ytableau}
*(boxL)  \\
*(boxU)    &  *(boxU) 
\end{ytableau}\,\,\begin{ytableau}
  \\
  &  
\end{ytableau} }{ \begin{ytableau}
 \\
    &  *(boxU) \\
   &    &  
\end{ytableau} }
-1\cdot
\frac{ \begin{ytableau}
*(boxL)  \\
*(boxU)    &  *(boxL) 
\end{ytableau}\,\,\begin{ytableau}
  \\
  &  
\end{ytableau} }{ \begin{ytableau}
 \\
    &  *(boxL) \\
   &    &  
\end{ytableau} }  
+ 1 \cdot \frac{ \begin{ytableau}
*(boxU)  \\
*(boxU)    &  *(boxU) 
\end{ytableau}\,\,\begin{ytableau}
  \\
  &  
\end{ytableau} }{ \begin{ytableau}
 \\
    &  *(boxL) \\
   &    &  
\end{ytableau} } 
\end{eqnarray*}
\ytableausetup{boxsize=0.8em}
where the empty boxes are filled with the fixed choice $\bmX_{\bar Y}$.
\end{example}

Note: It is not  hard to show that the $K$-value notation extends naturally to our setting with the blacked-out boxes $c_1, c_6$ (which are determined by $\tilde b_3$ via \eqref{eq:tildeb3def}). The $K$-transform acts only on the ten remaining free boxes, and the reference diagram $\bmB$ assigns hook choices to these ten boxes.

We return now to our main object of focus, the Stanley sum $\bigstanleysum$. 
\begin{lemma}\label{lemma:Kvaluess}
The K-values of $\bigstanleysum$ with reference to $\bar \bmX$ are
\[ K_I^{\bar\bmX}(\bigstanleysum) = \begin{cases} 2 & I = \emptyset \\ 1 & |I| = 1 \\ 1 & I =\{a,b\} \text{ with } \, a\sim b\\ 0 & \text{otherwise}. \end{cases} \]
\end{lemma}
\begin{proof}
We have
\[ K_\emptyset = \sum_\bmD c_\bmD = 7(1)-2(10)+1(15) = 2 \]
\[ K_{\{b\}}^{\bar\bmX} = \sum_{\bmD:\bmD_b =\bar\bmX_b}  c_\bmD = 7(0)-2(1)+1(\sum_{a:b\sim a} 1) = 1 \]
\[ K_{\{a,b\}}^{\bar\bmX} = \sum_{\bmD:...} c_\bmD = 7(0)-2(0)+1(\delta_{a \sim b}) = \delta_{a\sim b}. \]
As no diagram in the sum \eqref{eq:721sum} differs from $\bmX$ at more than two boxes, all higher $K^{\bar \bmX}$ vanish.
\end{proof}

For the conjugate Stanley sum $\bar \bigstanleysum$, we find a simple formula.
\begin{lemma}
The $K$-values of the conjugate $\bar\bigstanleysum$ is
\[ K_I^{\bar \bmX}(\bar \bigstanleysum) = 2-|I|+e_I, \]
where $e_I$ is the number of edges of the induced graph on $I$.
\end{lemma}
\begin{proof}
Using the inversion formula \eqref{eq:Kinversion}, we have
\begin{align*}
K_I^{\bar\bmX}(\bar\bigstanleysum) &=K_I^{\bmX}(\bigstanleysum)  \\
&= \sum_{J \subseteq I} (-1)^{|J|} K_J^{\bar\bmX}(\bigstanleysum) \\
&= (-1)^0\cdot2   + \sum_{b \in I}(-1)^1\cdot  1  + \sum_{\{a,b\} \subseteq I, \{a,b\} \in \Theta}  (-1)^2 \cdot 1  \\
&= 2 - |I| + e_I. \qedhere
\end{align*}
\end{proof}

Comparing with \eqref{lemma:Kvaluess}, we find
\begin{corollary} For $|I|<3$, we have
\[ K_I^{\bar\bmX}(\bigstanleysum) = K_I^{\bar\bmX}(\bar \bigstanleysum). \] 
\end{corollary}

However, for $|I|\geq 3$, there is a difference in the $K$-values of the two Stanley sums. One of the objectives of this work is to prove that not only are $\bigstanleysum$ and $\bar\bigstanleysum$ \emph{both} solutions to \eqref{eq:bigequation}, they are indeed the \emph{same} solution once descended to the hook space.

\begin{proposition}
For a Stanley sum $\bmss$ with $K$-values $K^\bmB_I(\bmss)$ relative to $\bmB$, the polynomial form $\bmss' = \sum_\bmD c_\bmD \bmD'$ admits the $\tilde\ell_b = \bh_b^\sU + \bh_b^\sL$ expansion
\begin{equation}\label{eq:ellexp}
\bmss' = 2^{-|\allvertices|}\sum_{I\subseteq \allvertices} \beta^{|I|}\, w_I\, \ell_{\allvertices/I},
\end{equation}
with weights
\[ w_J(K_\bullet^\bmB)  = (-1)^{|J|}\sum_{I \subseteq J}  K_I^\bmB(\bmss) \cdot (-2)^{|I|}.  \]
Furthermore,
\[ w_J(K_\bullet^{\bmA})= (-1)^{\bmA_J/\bmB_J} w_J(K_\bullet^{\bmB}). \]
\end{proposition}
\begin{proof}
Fix the reference $\bmB$. By \eqref{eq:varsell321}, each hook factor satisfies $x_b \bh_b^\sA = \tfrac{1}{2}(\ell_b + x_b(-1)^{\sA/\sU}\beta)$, with the convention that $x_b := (-1)^{\bmB_b/\sU}$ so that $x_b\bh_b^{\bmB_b} = \tfrac{1}{2}(\ell_b+\beta)$ and $x_b\bh_b^{\bar\bmB_b} = \tfrac{1}{2}(\ell_b-\beta)$.

For $J\subseteq\allvertices$, let $\bar\bmB^J$ denote the diagram agreeing with $\bmB$ on $J$ and with $\bar\bmB$ outside $J$. Substituting into the polynomial form $(\bar\bmB^J)'$, each factor becomes $\tfrac{1}{2}(\ell_b + \epsilon_b^J\beta)$ with $\epsilon_b^J = +1$ for $b\in J$ and $-1$ otherwise, so
\[ (\bar\bmB^J)' = 2^{-|\allvertices|}\prod_b(\ell_b + \epsilon_b^J\beta) = 2^{-|\allvertices|}\sum_{I\subseteq\allvertices}\beta^{|I|}\,(-1)^{|I\setminus J|}\,\ell_{\allvertices/I}, \]
using $\prod_{b\in I}\epsilon_b^J = (-1)^{|I\setminus J|}$.

By Corollary~\ref{cor:inversek}, $\bmss' = \sum_J c_{\bar\bmB^J}(\bar\bmB^J)'$ with $c_{\bar\bmB^J} = \sum_{L\supseteq J}(-1)^{|L|-|J|}K_L^\bmB$. Collecting the coefficient of $\beta^{|I|}\ell_{\allvertices/I}$:
\[ w_I = \sum_J c_{\bar\bmB^J}(-1)^{|I\setminus J|} = \sum_L K_L^\bmB \sum_{J\subseteq L}(-1)^{|L|-|J|+|I\setminus J|}. \]
For the inner sum, decompose $J = A\sqcup B$ with $A\subseteq L\cap I$, $B\subseteq L\setminus I$. Then $|J| = |A|+|B|$ and $|I\setminus J| = |I|-|A|$, so the exponent is $|L|+|I|-2|A|-|B|$, and
\[ \sum_{J\subseteq L}(-1)^{|L|-|J|+|I\setminus J|} = (-1)^{|L|+|I|}\,\underbrace{\sum_{A\subseteq L\cap I}1}_{2^{|L\cap I|}}\cdot\underbrace{\sum_{B\subseteq L\setminus I}(-1)^{|B|}}_{[L\subseteq I]}. \]
The second factor vanishes unless $L\subseteq I$, in which case $|L\cap I|=|L|$. Hence
\[ w_I = (-1)^{|I|}\sum_{L\subseteq I} K_L^\bmB(-2)^{|L|}. \]

For the sign covariance, let $D := \{b : \bmA_b \neq \bmB_b\}$, so $|J\cap D| = \bmA_J/\bmB_J$ counts the disagreements on $J$. By \eqref{eq:Kinversion}, $K_L^\bmA = \sum_{I\subseteq L,\,L\setminus I\subseteq D}(-1)^{|I\cap D|}K_I^\bmB$. Substituting and reindexing $L = I\sqcup S$ with $S\subseteq (J\setminus I)\cap D$:
\[ w_J(K_\bullet^\bmA) = (-1)^{|J|}\sum_{I\subseteq J}(-1)^{|I\cap D|}K_I^\bmB(-2)^{|I|}\!\sum_{S\subseteq (J\setminus I)\cap D}\!(-2)^{|S|}. \]
The inner sum equals $(1-2)^{|(J\setminus I)\cap D|} = (-1)^{|(J\setminus I)\cap D|}$, and combines with $(-1)^{|I\cap D|}$ to give $(-1)^{|J\cap D|}$, independent of $I$. Therefore $w_J(K_\bullet^\bmA) = (-1)^{\bmA_J/\bmB_J}\,w_J(K_\bullet^\bmB)$.
\end{proof}

Thus if we changed reference diagram $\bmB$ by flipping a single box at $b$, these weights only accrue a sign, compared to the more complicated behaviour of $K_I^{\bmB}$ given by \eqref{lemma:deponB}.

\begin{lemma}
For $K^I_X(\bigstanleysum)$ given by \eqref{lemma:Kvaluess}, the weights $w_J$ are given by 
\begin{equation}\label{eq:cvalues}
w_J = 2(-1)^{|J|}(2e_J-|J|+1).
\end{equation}
\end{lemma}
\begin{proof}
We substitute the $K$-values from Lemma \ref{lemma:Kvaluess} into the weight formula $w_J = (-1)^{|J|}\sum_{I \subseteq J} K_I^{\bar\bmX} (-2)^{|I|}$. Since $K_I^{\bar\bmX}=0$ for $|I|\geq 3$, only subsets of size $0,1,2$ contribute:
\begin{align*}
w_J &= (-1)^{|J|}\left[K_\emptyset\cdot(-2)^{0} + \sum_{b \in J} K_{\{b\}} \cdot (-2)^1 + \sum_{\substack{\{a,b\} \subseteq J \\ a\sim b}} K_{\{a,b\}}\cdot(-2)^2\right] \\
&= (-1)^{|J|}\left[2\cdot (1) + |J|\cdot (-2) + e_J\cdot 4\right] \\
&= 2(-1)^{|J|}(1 - |J| + 2e_J). \qedhere
\end{align*}
\end{proof}

\subsubsection{Summary}

We summarize the results of the above as follows. For $\petersen$ the Petersen graph, let $\ell_b$ be a variable for each of the 10 vertices  $b\in P$, $\petvertrep=\BZ[P]$ is the dimension 10 permutation module of $Aut(P)=S_5$. Consider the following  $S_5$-invariant polynomial, in the variables $\ell_b, \beta$, expressed as the sum over all subsets $I$ of the vertices of $\petersen$.
\[ 2^{|\petersen|}\bigstanleysum' := \sum_{I\subseteq \petersen} w_{P/I} \ell_I \beta^{|\petersen/I|} \,\, \in \Sym^{10}(V_\petersen[\bbeta])^{S_5}, \]
where $\ell_I = \prod_{b\in I} \ell_b$, and the $S_5$-invariant weight $c_J \in \BZ$ for a subset $J$ is given by 
\begin{equation}\label{weightsdef}
w_{J}:=2(2e_J - |J|+1),
\end{equation}
where $e_J$ is the number of edges in the subgraph of $\petersen$ on the vertices $J$. We have that
\[ 2^{|\petersen|} \bigstanleysum' = 2(\prod_{ b\in \petersen} \ell_b) \beta^0 + \ldots + 42 \beta^{10}. \]
Note that $f$ has coefficients of both signs, e.g. for the subset $I\subset \petersen$ of four isolated vertices, then $c_I = -6$.

\subsection{Orbit Sums}

Since the weights $w_I$ are invariant under $S_5$, we group the terms of \eqref{eq:ellexp} according to the $S_5$ orbits of subsets of vertices.  Let $\CO_k$ be the set of $S_5$-orbits of $k$-subsets of vertices of $\petersen$. For $O \in \CO_k$ an orbit, define the $S_5$-invariant \emph{orbit contribution}
\begin{equation}
L_{O}:= \sum_{I \in O} \ell_I.
\end{equation}
We then define the degree-$k$ \emph{orbit sum} as the weighted sum of these contributions,
\[ \orbitsum_k := \sum_{O\in \CO_{k}} w_{\bar O} L_O, \]
and express \eqref{eq:ellexp} as the expansion in these orbit sums 
\begin{eqnarray*}
2^{10} \bigstanleysum' &=& \sum_{k=1}^{10} \beta^{10-k} \orbitsum_k\\
&=&  2\ell_\petersen \,\,+ \,\,\cdots\,\, +\,\, 42 \beta^{10} \ell_{\emptyset} 
\end{eqnarray*}
where $42 = w_{\petersen}$. The ellipsis represents a complicated collection of terms.

We note that the conjugation involution \eqref{eq:conjoperation} acts as $C:\ell_b \to -\ell_b$, and so $C:\orbitsum_k \to (-1)^{k} \orbitsum_k$. Thus, theorem \eqref{thm:z2invs} is equivalent to the following.

\begin{theorem}\label{thm:oddorbitsumvanish}
For $k$ odd, the orbit sum $\orbitsum_k$ lies in the kernel of the pullback to the hook space. That is
\begin{equation}
\orbitsum_k \in \ker\left(\iota : \Sym(\petvertrep) \to \Sym(V^{(1)}_{\{3,2\}})\right), \quad \text{ for } k \text{ odd }.
\end{equation}

\end{theorem}

In the following subsections, we prove \eqref{thm:oddorbitsumvanish} this statement in each degree individually (\eqref{prop:degree1}, \eqref{prop:degree3}, \eqref{prop:degree5}, \eqref{prop:degree7}), by showing  formulas of the form (for each odd $k$)
\[ \orbitsum_{k} := \sum_{O\in \CO_{k}} w_{\bar O} L_O = \sum_i Q_i
\]
where $w_{\bar O}:= w_{\petersen/I}$, for any representative subset $I \in O$, and each $Q_i \in \ker(\iota)$ is a kernel element given by the following construction.  
For $V\in \CO_{k}$, let $F$ be a function on $\petersen$ with support on $\bar V$ that is invariant under the $V$-stabilizer $\Aut(\petersen)^V$. Then the following degree $k+1$ polynomial the $\ell$ variables is $\Aut(\petersen)$ invariant
\begin{equation} Q(V,F):= \sum_{I\in V} \ell_{ V} \sum_{b\in \bar V} F_b \ell_b = \sum_{O\in \CO_{k+1}: V\subset O } F_{O/V}\, [O\!:\! V]\, L_O 
\end{equation}
where $[O\!:\! V]$ counts the number of inclusions $V\hookrightarrow  O$. This is easily expressed in terms of the orbit sums. Note, if such a function $F \in \mathrm{image}(1\text{-}A)$, then
\[\sum_{b\in \bar V} F_b \ell_b \in \ker(\iota), \]
We denote the space of these functions $F$ as $H(V)$.

\begin{corollary}
For $F \in H(V)$, the function $Q(V,F)$ lies in the kernel of $\iota$
\end{corollary}
This corollary will be used extensively in what follows. The functions $F$ we will use will be in the span one of the following kernel functions,
\[
f_1:=\begin{tikzpicture}[scale=0.6,transform shape, every node/.style={circle, draw, inner sep=1.2pt}, baseline={([yshift=-1.0ex]current bounding box.center)}]
  \node (o1) at ( 90:1.6) {$1$};
  \node (o2) at ( 18:1.6) {$0$};
  \node (o3) at (-54:1.6) {$0$};
  \node (o4) at (-126:1.6) {$0$};
  \node (o5) at (162:1.6) {$0$};
  \node (i1) at ( 90:0.9) {$\text{-}1$};
  \node (i2) at ( 18:0.9) {$0$};
  \node (i3) at (-54:0.9) {$1$};
  \node (i4) at (-126:0.9) {$1$};
  \node (i5) at (162:0.9) {$0$};
  \draw (o1)--(o2)--(o3)--(o4)--(o5)--(o1);
  \draw (o1)--(i1);
  \draw (o2)--(i2);
  \draw (o3)--(i3);
  \draw (o4)--(i4);
  \draw (o5)--(i5);
  \draw (i1)--(i3)--(i5)--(i2)--(i4)--(i1);
\end{tikzpicture}
=(1\text{-}A)
\begin{tikzpicture}[scale=0.6,transform shape, every node/.style={circle, draw, inner sep=1.2pt}, baseline={([yshift=-1.0ex]current bounding box.center)}]
  \node (o1) at ( 90:1.6) {$0$};
  \node (o2) at ( 18:1.6) {$0$};
  \node (o3) at (-54:1.6) {$0$};
  \node (o4) at (-126:1.6) {$0$};
  \node (o5) at (162:1.6) {$0$};
  \node (i1) at ( 90:0.9) {$\text{-}1$};
  \node (i2) at ( 18:0.9) {$0$};
  \node (i3) at (-54:0.9) {$0$};
  \node (i4) at (-126:0.9) {$0$};
  \node (i5) at (162:0.9) {$0$};
  \draw (o1)--(o2)--(o3)--(o4)--(o5)--(o1);
  \draw (o1)--(i1);
  \draw (o2)--(i2);
  \draw (o3)--(i3);
  \draw (o4)--(i4);
  \draw (o5)--(i5);
  \draw (i1)--(i3)--(i5)--(i2)--(i4)--(i1);
\end{tikzpicture}
\]

\[
f_2:=\begin{tikzpicture}[scale=0.6,transform shape, every node/.style={circle, draw, inner sep=1.2pt}, baseline={([yshift=-1.0ex]current bounding box.center)}]
  \node (o1) at ( 90:1.6) {$0$};
  \node (o2) at ( 18:1.6) {$1$};
  \node (o3) at (-54:1.6) {$0$};
  \node (o4) at (-126:1.6) {$0$};
  \node (o5) at (162:1.6) {$1$};
  \node (i1) at ( 90:0.9) {$0$};
  \node (i2) at ( 18:0.9) {$0$};
  \node (i3) at (-54:0.9) {$1$};
  \node (i4) at (-126:0.9) {$1$};
  \node (i5) at (162:0.9) {$0$};
  \draw (o1)--(o2)--(o3)--(o4)--(o5)--(o1);
  \draw (o1)--(i1);
  \draw (o2)--(i2);
  \draw (o3)--(i3);
  \draw (o4)--(i4);
  \draw (o5)--(i5);
  \draw (i1)--(i3)--(i5)--(i2)--(i4)--(i1);
\end{tikzpicture}
=(1\text{-}A)
\begin{tikzpicture}[scale=0.6,transform shape, every node/.style={circle, draw, inner sep=1.2pt}, baseline={([yshift=-1.0ex]current bounding box.center)}]
  \node (o1) at ( 90:1.6) {$\text{-}1$};
  \node (o2) at ( 18:1.6) {$0$};
  \node (o3) at (-54:1.6) {$0$};
  \node (o4) at (-126:1.6) {$0$};
  \node (o5) at (162:1.6) {$0$};
  \node (i1) at ( 90:0.9) {$\text{-}1$};
  \node (i2) at ( 18:0.9) {$0$};
  \node (i3) at (-54:0.9) {$0$};
  \node (i4) at (-126:0.9) {$0$};
  \node (i5) at (162:0.9) {$0$};
  \draw (o1)--(o2)--(o3)--(o4)--(o5)--(o1);
  \draw (o1)--(i1);
  \draw (o2)--(i2);
  \draw (o3)--(i3);
  \draw (o4)--(i4);
  \draw (o5)--(i5);
  \draw (i1)--(i3)--(i5)--(i2)--(i4)--(i1);
\end{tikzpicture}
\]
\[f_3:=
\begin{tikzpicture}[scale=0.6,transform shape, every node/.style={circle, draw, inner sep=1.2pt}, baseline={([yshift=-1.0ex]current bounding box.center)}]
  \node (o1) at ( 90:1.6) {$1$};
  \node (o2) at ( 18:1.6) {$0$};
  \node (o3) at (-54:1.6) {$1$};
  \node (o4) at (-126:1.6) {$1$};
  \node (o5) at (162:1.6) {$0$};
  \node (i1) at ( 90:0.9) {$1$};
  \node (i2) at ( 18:0.9) {$1$};
  \node (i3) at (-54:0.9) {$0$};
  \node (i4) at (-126:0.9) {$0$};
  \node (i5) at (162:0.9) {$1$};
  \draw (o1)--(o2)--(o3)--(o4)--(o5)--(o1);
  \draw (o1)--(i1);
  \draw (o2)--(i2);
  \draw (o3)--(i3);
  \draw (o4)--(i4);
  \draw (o5)--(i5);
  \draw (i1)--(i3)--(i5)--(i2)--(i4)--(i1);
\end{tikzpicture} = (1\text{-}A)\begin{tikzpicture}[scale=0.6,transform shape, every node/.style={circle, draw, inner sep=1.2pt}, baseline={([yshift=-1.0ex]current bounding box.center)}]
  \node (o1) at ( 90:1.6) {$\text{-}1$};
  \node (o2) at ( 18:1.6) {$\text{-}1$};
  \node (o3) at (-54:1.6) {$0$};
  \node (o4) at (-126:1.6) {$0$};
  \node (o5) at (162:1.6) {$\text{-}1$};
  \node (i1) at ( 90:0.9) {$0$};
  \node (i2) at ( 18:0.9) {$0$};
  \node (i3) at (-54:0.9) {$0$};
  \node (i4) at (-126:0.9) {$0$};
  \node (i5) at (162:0.9) {$0$};
  \draw (o1)--(o2)--(o3)--(o4)--(o5)--(o1);
  \draw (o1)--(i1);
  \draw (o2)--(i2);
  \draw (o3)--(i3);
  \draw (o4)--(i4);
  \draw (o5)--(i5);
  \draw (i1)--(i3)--(i5)--(i2)--(i4)--(i1);
\end{tikzpicture}.
\]
\begin{equation} f_\Gamma:= \begin{tikzpicture}[scale=0.6,transform shape, every node/.style={circle, draw, inner sep=1.2pt}, baseline={([yshift=-1.0ex]current bounding box.center)}]
  \node (o1) at ( 90:1.6) {$1$};
  \node (o2) at ( 18:1.6) {$1$};
  \node (o3) at (-54:1.6) {$1$};
  \node (o4) at (-126:1.6) {$1$};
  \node (o5) at (162:1.6) {$1$};
  \node (i1) at ( 90:0.9) {$1$};
  \node (i2) at ( 18:0.9) {$1$};
  \node (i3) at (-54:0.9) {$1$};
  \node (i4) at (-126:0.9) {$1$};
  \node (i5) at (162:0.9) {$1$};
  \draw (o1)--(o2)--(o3)--(o4)--(o5)--(o1);
  \draw (o1)--(i1);
  \draw (o2)--(i2);
  \draw (o3)--(i3);
  \draw (o4)--(i4);
  \draw (o5)--(i5);
  \draw (i1)--(i3)--(i5)--(i2)--(i4)--(i1);
\end{tikzpicture} = (1\text{-}A)\tfrac{\text{-}1}{2} f_\Gamma
\end{equation}
Note that $A f_\Gamma=3 f_\Gamma$.
\begin{example}\label{ex:qvf}
For example we look at $Q(\gPathTwo, f_2)$, which corresponds to
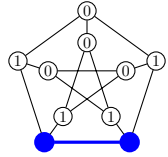
\[ \begin{tikzpicture}[scale=0.6,transform shape, every node/.style={circle, draw, inner sep=1.2pt}, baseline={([yshift=-1.0ex]current bounding box.center)}]
  \node (o1) at ( 90:1.6) {$0$};
  \node (o2) at ( 18:1.6) {$1$};
  \node[blue, thick, fill] (o3) at (-54:1.6) {$0$};
  \node[blue, thick, fill] (o4) at (-126:1.6) {$0$};
  \node (o5) at (162:1.6) {$1$};
  \node (i1) at ( 90:0.9) {$0$};
  \node (i2) at ( 18:0.9) {$0$};
  \node (i3) at (-54:0.9) {$1$};
  \node (i4) at (-126:0.9) {$1$};
  \node (i5) at (162:0.9) {$0$};
  \draw (o1)--(o2)--(o3) (o4)--(o5)--(o1);
  \draw[blue, very thick] (o3)--(o4);
  \draw (o1)--(i1);
  \draw (o2)--(i2);
  \draw (o3)--(i3);
  \draw (o4)--(i4);
  \draw (o5)--(i5);
  \draw (i1)--(i3)--(i5)--(i2)--(i4)--(i1);
\end{tikzpicture} \]
We see that the function $f_2$ has support away from the (blue) subgraph $V=\gPathTwo$, and is invariant under the automorphism subgroup that fixes $V$, and thus $f_2 \in H(V)$.

Including each of the vertices in the support of the function $f_2$, we get the four contributions
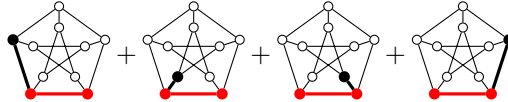
\[  \begin{tikzpicture}[scale=0.4, every node/.style={circle, draw, inner sep=1.2pt}, baseline={([yshift=-1.0ex]current bounding box.center)}]
  \node (o1) at ( 90:1.6) {$$};
  \node (o2) at ( 18:1.6) {$$};
  \node[red, thick, fill] (o3) at (-54:1.6) {$$};
  \node[red, thick, fill] (o4) at (-126:1.6) {$$};
  \node[thick, fill] (o5) at (162:1.6) {$$};
  \node (i1) at ( 90:0.9) {$$};
  \node (i2) at ( 18:0.9) {$$};
  \node (i3) at (-54:0.9) {$$};
  \node (i4) at (-126:0.9) {$$};
  \node (i5) at (162:0.9) {$$};
  \draw (o1)--(o2)--(o3) (o5)--(o1);
  \draw[red, very thick] (o3)--(o4);
  \draw[ very thick] (o4)--(o5);
  \draw (o1)--(i1);
  \draw (o2)--(i2);
  \draw (o3)--(i3);
  \draw (o4)--(i4);
  \draw (o5)--(i5);
  \draw (i1)--(i3)--(i5)--(i2)--(i4)--(i1);
\end{tikzpicture} + 
\begin{tikzpicture}[scale=0.4, every node/.style={circle, draw, inner sep=1.2pt}, baseline={([yshift=-1.0ex]current bounding box.center)}]
  \node (o1) at ( 90:1.6) {$$};
  \node (o2) at ( 18:1.6) {$$};
  \node[red, thick, fill] (o3) at (-54:1.6) {$$};
  \node[red, thick, fill] (o4) at (-126:1.6) {$$};
  \node (o5) at (162:1.6) {$$};
  \node (i1) at ( 90:0.9) {$$};
  \node (i2) at ( 18:0.9) {$$};
  \node (i3) at (-54:0.9) {$$};
  \node[thick, fill] (i4) at (-126:0.9) {$$};
  \node (i5) at (162:0.9) {$$};
  \draw (o1)--(o2)--(o3) (o4)--(o5)--(o1);
  \draw[red, very thick] (o3)--(o4);
    \draw[ very thick] (o4)--(i4);
  \draw (o1)--(i1);
  \draw (o2)--(i2);
  \draw (o3)--(i3);
  \draw (o4)--(i4);
  \draw (o5)--(i5);
  \draw (i1)--(i3)--(i5)--(i2)--(i4)--(i1);
\end{tikzpicture}+
\begin{tikzpicture}[scale=0.4, every node/.style={circle, draw, inner sep=1.2pt}, baseline={([yshift=-1.0ex]current bounding box.center)}]
  \node (o1) at ( 90:1.6) {$$};
  \node (o2) at ( 18:1.6) {$$};
  \node[red, thick, fill] (o3) at (-54:1.6) {$$};
  \node[red, thick, fill] (o4) at (-126:1.6) {$$};
  \node (o5) at (162:1.6) {$$};
  \node (i1) at ( 90:0.9) {$$};
  \node (i2) at ( 18:0.9) {$$};
  \node[thick, fill] (i3) at (-54:0.9) {$$};
  \node (i4) at (-126:0.9) {$$};
  \node (i5) at (162:0.9) {$$};
  \draw (o1)--(o2)--(o3) (o4)--(o5)--(o1);
  \draw[red, very thick] (o3)--(o4);
  \draw[ very thick] (o3)--(i3);
  \draw (o1)--(i1);
  \draw (o2)--(i2);
  \draw (o3)--(i3);
  \draw (o4)--(i4);
  \draw (o5)--(i5);
  \draw (i1)--(i3)--(i5)--(i2)--(i4)--(i1);
\end{tikzpicture}+
\begin{tikzpicture}[scale=0.4, every node/.style={circle, draw, inner sep=1.2pt}, baseline={([yshift=-1.0ex]current bounding box.center)}]
  \node (o1) at ( 90:1.6) {$$};
  \node[thick, fill] (o2) at ( 18:1.6) {$$};
  \node[red, thick, fill] (o3) at (-54:1.6) {$$};
  \node[red, thick, fill] (o4) at (-126:1.6) {$$};
  \node (o5) at (162:1.6) {$$};
  \node (i1) at ( 90:0.9) {$$};
  \node (i2) at ( 18:0.9) {$$};
  \node (i3) at (-54:0.9) {$$};
  \node (i4) at (-126:0.9) {$$};
  \node (i5) at (162:0.9) {$$};
  \draw (o1)--(o2)--(o3) (o4)--(o5)--(o1);
  \draw[red, very thick] (o3)--(o4);
  \draw[ very thick] (o3)--(o2);
  \draw (o1)--(i1);
  \draw (o2)--(i2);
  \draw (o3)--(i3);
  \draw (o4)--(i4);
  \draw (o5)--(i5);
  \draw (i1)--(i3)--(i5)--(i2)--(i4)--(i1);
\end{tikzpicture} \]
Adding in each of these vertices augments $\gPathTwo \to \gPathThree$, and thus we conclude that 
\[ Q(\gPathTwo,f_2) = (1) \times [\,\gPathTwo:\gPathThree\, ] \times \gPathThree = 2\, \gPathThree. \]
\end{example}

We now repeat this process of finding such linear combinations of orbit sums that lie in the kernel one degree at a time.

NOTE: From here on we rescale the weights \eqref{weightsdef} by a factor of $\tfrac{1}{2}$, for cleaner formulas.

\subsubsection{Degree 1}

\[
\centering
\begin{tabular}{c|c|c|c|c}
Orbit $O$ & $|O|$ & $e(O)$   & $w_{\bar O}$\\
\hline
$\bullet$ & 10 & 0 & 2(15-3)-(10-1)+1=16 
\end{tabular}
\]

\begin{proposition}\label{prop:degree1}
The degree 1 orbit sum, 
\begin{equation}
\orbitsum_1:=\sum_{O\in \CO_1} w_{\bar O} L_O = 16\,\bullet,
\end{equation}
 lies in the kernel.
\end{proposition}
\begin{proof}

Here we only have one orbit $O\in \CO_1$ corresponding to $ \bullet$, for which 
\begin{equation}
w_{\bar \bullet} = 2(15-3)-9+1 =16
\end{equation}

There is only one orbit of size $k-1=0$, namely $\emptyset$. There is only a single function on $\bar \emptyset=\Gamma$ that is invariant under $G^{\emptyset} = G$, that is, the constant function $f_\Gamma$. 
\begin{equation}
-2 \, \bullet  = -2 \sum_b \ell_b = (1\text{-}A) f_\Gamma.
\end{equation}
Thus we have shown that
\[ \sum_{O \in \CO_1} w_{\bar O} L_O = -8 \times Q(\emptyset, \,f_\Gamma). \]
\end{proof}

\subsubsection{Degree 3}

Here we have the weights

\[
\begin{tabular}{c|c|c|c|c|c}
Orbit $O$ & $|O|$ & $e(O)$ & $w_{\bar O}$\\
\hline
\gPathThree & 30 & 2  & 10 \\
\gEdgeIsoThreeGhost & 60 & 1  & 8\\
\gStarGhostThree & 10 & 0 &  6\\
\gthreeIsoGhost & 20 & 0 &  6\\
\end{tabular}
\]

\begin{proposition}\label{prop:degree3}
The degree 3 orbit sum lies in the kernel, specifically
\begin{eqnarray}
\orbitsum_3 &:=& 10\,\gPathThree + 8\,\gEdgeIsoThreeGhost+6\,\gStarGhostThree+6\,\gthreeIsoGhost\\
&=& 2\,Q(\gPathTwo, f_2)+2\,Q(\gIsoTwo, f_1)+4\,Q(\gIsoTwo, f_3).
\end{eqnarray}
\end{proposition}

\begin{proof}

First, we determine the $V_i$ and $F_i \in H(V_i)$. There are two $V_i$, 
\[ V_1 =\gPathTwo, \qquad V_2 =\gIsoTwo. \]
We find $H(\gPathTwo) = \mathrm{span}\{ f_2 \}$. Similarly, $H(\gIsoTwo) = \mathrm{span}\{ f_1, f_3 \}$. So we have a total of three dimensions of kernel sums, whilst we have four orbit terms $L_{O_j}$. Thus, the system 
\begin{equation}\label{eq:lsumeqn}
 \sum_{O\in \CO_3} w_{\bar O} L_O= \sum_{i=1}^{3} Q(V_i,f_i) \times v_i 
 \end{equation}
is over determined.

In the example \eqref{ex:qvf}, we showed that
\[ Q(\gPathTwo, f_2) =  2\, \gPathThree. \]

Next, the graph for $Q(\gIsoTwo, f_1)$ is

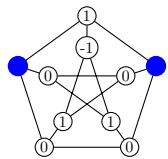
\begin{equation}
\begin{tikzpicture}[scale=0.6,transform shape, baseline=(current bounding box.center), every node/.style={circle, draw, inner sep=1.2pt}]
  \node (o1) at ( 90:1.6) {$1$};
  \node[blue, thick, fill] (o2) at ( 18:1.6) {$0$};
  \node (o3) at (-54:1.6) {$0$};
  \node (o4) at (-126:1.6) {$0$};
  \node[blue, thick, fill] (o5) at (162:1.6) {$0$};
  \node (i1) at ( 90:0.9) {$\text{-}1$};
  \node (i2) at ( 18:0.9) {$0$};
  \node (i3) at (-54:0.9) {$1$};
  \node (i4) at (-126:0.9) {$1$};
  \node (i5) at (162:0.9) {$0$};
  \draw (o1)--(o2)--(o3)--(o4)--(o5)--(o1);
  \draw (o1)--(i1);
  \draw (o2)--(i2);
  \draw (o3)--(i3);
  \draw (o4)--(i4);
  \draw (o5)--(i5);
  \draw (i1)--(i3)--(i5)--(i2)--(i4)--(i1);
\end{tikzpicture}
\qquad
Q(\gIsoTwo, F_1) = 1\,\gPathThree+0\,\gEdgeIsoThreeGhost-3\,\gStarGhostThree+3\,\gthreeIsoGhost.
\end{equation}

Lastly, we look at the graph for $Q(\gIsoTwo, f_3)$, which is
\begin{equation}
\begin{tikzpicture}[scale=0.6,transform shape, baseline=(current bounding box.center), every node/.style={circle, draw, inner sep=1.2pt}]
  \node (o1) at ( 90:1.6) {$1$};
  \node[blue, thick, fill] (o2) at ( 18:1.6) {$0$};
  \node (o3) at (-54:1.6) {$1$};
  \node (o4) at (-126:1.6) {$1$};
  \node[blue, thick, fill] (o5) at (162:1.6) {$0$};
  \node (i1) at ( 90:0.9) {$1$};
  \node (i2) at ( 18:0.9) {$1$};
  \node (i3) at (-54:0.9) {$0$};
  \node (i4) at (-126:0.9) {$0$};
  \node (i5) at (162:0.9) {$1$};
  \draw (o1)--(o2)--(o3)--(o4)--(o5)--(o1);
  \draw (o1)--(i1);
  \draw (o2)--(i2);
  \draw (o3)--(i3);
  \draw (o4)--(i4);
  \draw (o5)--(i5);
  \draw (i1)--(i3)--(i5)--(i2)--(i4)--(i1);
\end{tikzpicture}
\qquad
Q(\gIsoTwo, F_3) = 1\,\gPathThree+2\,\gEdgeIsoThreeGhost+3\,\gStarGhostThree+0\,\gthreeIsoGhost.
\end{equation}

With these three vectors determined we find that the overdetermined matrix equation \eqref{eq:lsumeqn}, and its unique solution $v$, are given by 
\begin{equation}
\left(\begin{matrix}
10 \\
8 \\
6\\
6
\end{matrix}\right) = 
\left(\begin{matrix}
2 & 1 & 1 \\
0 & 0 & 2 \\
0 & -3 & 3\\
0 & 3 & 0 
\end{matrix}\right)
\left(\begin{matrix}
2 \\
2 \\
4
\end{matrix}\right).
\end{equation}
Thus we have shown that the degree three orbit sum lies in the span of the kernel

\begin{equation}
\sum_{O\in \CO_3} w_{\bar O} L_O = 2\,Q(\gPathTwo, f_2)+2\,Q(\gIsoTwo, f_1)+4\,Q(\gIsoTwo, f_3).
\end{equation}

\end{proof}

\subsubsection{Degree 5}

Here, we find

\[
\centering
\begin{tabular}{c|c|c|c|c}
Orbit $O$ & $|O|$ & $e(O)$ &$\bar O$ & $w_{\bar O}$\\
\hline
\gFiveCycle & 12 & 5 &  \gFiveCycle & 6 \\
\gFivePath & 60 & 4 &  \gFiveTreeA & 4\\
\gFiveTreeA & 60 & 4 &  \gFivePath & 4\\
\gFivePFourPlusIso & 60 & 3 &  \gFivePFourPlusIso & 2\\
\gPthreeTwoIso & 30 & 2 & \gTwoEdgesIso & 0\\
\gTwoEdgesIso & 30 & 2 &  \gPthreeTwoIso & 0 \\
\end{tabular}
\]

\begin{proposition}\label{prop:degree5}
The degree 5 orbit sum lies in the kernel. Specifically,
\begin{eqnarray*}
\orbitsum_5 &:=& 6\,\gFiveCycle + 4\,\gFivePath+4\,\gFiveTreeA+2\,\gFivePFourPlusIso+0\,\gPthreeTwoIso+0\,\gTwoEdgesIso \\
&=& \tfrac{1}{5} \big(\, 6 \,Q(\gPFour, f_1) + 0\, Q(\gPthreeIso, f_1)  + 8 \,Q(\gPtwoIsotwo, f_2) + \\
&&  \quad\,\, 20\, Q(\gPtwoIsotwo, f_2) + 16\, Q(\gFourIso,f_2) \,\big).
\end{eqnarray*}

\end{proposition}
\begin{proof}
Here, we compute five kernel sums, 
\begin{eqnarray}
\begin{tikzpicture}[scale=0.6,transform shape, baseline=(current bounding box.center), every node/.style={circle, draw, inner sep=1.2pt}]
  \node (o1) at ( 90:1.6) {$1$};
  \node[blue, thick, fill] (o2) at ( 18:1.6) {$1$};
  \node[blue, thick, fill] (o3) at (-54:1.6) {$0$};
  \node[blue, thick, fill] (o4) at (-126:1.6) {$0$};
  \node[blue, thick, fill] (o5) at (162:1.6) {$1$};
  \node (i1) at ( 90:0.9) {$\text{-}1$};
  \node (i2) at ( 18:0.9) {$0$};
  \node (i3) at (-54:0.9) {$1$};
  \node (i4) at (-126:0.9) {$1$};
  \node (i5) at (162:0.9) {$0$};
  \draw (o1)--(o2) (o5)--(o1);
  \draw[blue, very thick] (o2)--(o3)--(o4)--(o5);
  \draw (o1)--(i1);
  \draw (o2)--(i2);
  \draw (o3)--(i3);
  \draw (o4)--(i4);
  \draw (o5)--(i5);
  \draw (i1)--(i3)--(i5)--(i2)--(i4)--(i1);
\end{tikzpicture} \quad & Q(\gPFour, f_1) &= 5\,\gFiveCycle + 2\, \gFiveTreeA - \,\gFivePFourPlusIso \\
\begin{tikzpicture}[scale=0.6,transform shape, baseline=(current bounding box.center), every node/.style={circle, draw, inner sep=1.2pt}]
  \node (o1) at ( 90:1.6) {$1$};
  \node (o2) at ( 18:1.6) {$0$};
  \node[blue, thick, fill] (o3) at (-54:1.6) {$0$};
  \node[blue, thick, fill] (o4) at (-126:1.6) {$0$};
  \node[blue, thick, fill] (o5) at (162:1.6) {$1$};
  \node (i1) at ( 90:0.9) {$\text{-}1$};
  \node[blue, thick, fill] (i2) at ( 18:0.9) {$0$};
  \node (i3) at (-54:0.9) {$1$};
  \node (i4) at (-126:0.9) {$1$};
  \node (i5) at (162:0.9) {$0$};
  \draw (o1)--(o2)--(o3) (o5)--(o1);
  \draw[blue, very thick] (o3)--(o4)--(o5);
  \draw (o1)--(i1);
  \draw (o2)--(i2);
  \draw (o3)--(i3);
  \draw (o4)--(i4);
  \draw (o5)--(i5);
  \draw (i1)--(i3)--(i5)--(i2)--(i4)--(i1);
\end{tikzpicture} \quad& Q(\gPthreeIso, f_1) &= \,\gFiveTreeA+ 2\,\gFivePFourPlusIso  - 2\,\gPthreeTwoIso \\
\begin{tikzpicture}[scale=0.6,transform shape, baseline=(current bounding box.center), every node/.style={circle, draw, inner sep=1.2pt}]
  \node (o1) at ( 90:1.6) {$0$};
  \node (o2) at ( 18:1.6) {$1$};
  \node[blue, thick, fill] (o3) at (-54:1.6) {$0$};
  \node[blue, thick, fill] (o4) at (-126:1.6) {$0$};
  \node (o5) at (162:1.6) {$1$};
  \node (i1) at ( 90:0.9) {$0$};
  \node[blue, thick, fill] (i2) at ( 18:0.9) {$0$};
  \node (i3) at (-54:0.9) {$1$};
  \node (i4) at (-126:0.9) {$1$};
  \node[blue, thick, fill] (i5) at (162:0.9) {$0$};
  \draw (o1)--(o2)--(o3) (o4)--(o5)--(o1);
  \draw[blue, very thick] (o3)--(o4) (i2)--(i5);
  \draw (o1)--(i1);
  \draw (o2)--(i2);
  \draw (o3)--(i3);
  \draw (o4)--(i4);
  \draw (o5)--(i5);
  \draw (i1)--(i3)--(i5) (i2)--(i4)--(i1);
\end{tikzpicture} \quad & Q(\gtwotwo, f_2) &= \gFivePath
\\
\begin{tikzpicture}[scale=0.6,transform shape, baseline=(current bounding box.center), every node/.style={circle, draw, inner sep=1.2pt}]
  \node (o1) at ( 90:1.6) {$0$};
  \node (o2) at ( 18:1.6) {$1$};
  \node[blue, thick, fill] (o3) at (-54:1.6) {$0$};
  \node[blue, thick, fill] (o4) at (-126:1.6) {$0$};
  \node (o5) at (162:1.6) {$1$};
  \node[blue, thick, fill] (i1) at ( 90:0.9) {$0$};
  \node[blue, thick, fill] (i2) at ( 18:0.9) {$0$};
  \node (i3) at (-54:0.9) {$1$};
  \node (i4) at (-126:0.9) {$1$};
  \node (i5) at (162:0.9) {$0$};
  \draw (o1)--(o2)--(o3) (o4)--(o5)--(o1);
  \draw[blue, very thick] (o3)--(o4);
  \draw (o1)--(i1);
  \draw (o2)--(i2);
  \draw (o3)--(i3);
  \draw (o4)--(i4);
  \draw (o5)--(i5);
  \draw (i1)--(i3)--(i5)--(i2)--(i4)--(i1);
\end{tikzpicture} \quad& Q(\gPtwoIsotwo, f_2) &= \,\gFiveTreeA + 2\,\gFivePFourPlusIso  + 2\,\gPthreeTwoIso \\
\begin{tikzpicture}[scale=0.6,transform shape, baseline=(current bounding box.center), every node/.style={circle, draw, inner sep=1.2pt}]
  \node (o1) at ( 90:1.6) {$1$};
  \node[blue, thick, fill] (o2) at ( 18:1.6) {$0$};
  \node (o3) at (-54:1.6) {$1$};
  \node (o4) at (-126:1.6) {$1$};
  \node[blue, thick, fill] (o5) at (162:1.6) {$0$};
  \node (i1) at ( 90:0.9) {$1$};
  \node (i2) at ( 18:0.9) {$1$};
  \node[blue, thick, fill] (i3) at (-54:0.9) {$0$};
  \node[blue, thick, fill] (i4) at (-126:0.9) {$0$};
  \node (i5) at (162:0.9) {$1$};
  \draw (o1)--(o2)--(o3)--(o4)--(o5)--(o1);
  \draw (o1)--(i1);
  \draw (o2)--(i2);
  \draw (o3)--(i3);
  \draw (o4)--(i4);
  \draw (o5)--(i5);
  \draw (i1)--(i3)--(i5)--(i2)--(i4)--(i1);
\end{tikzpicture}
\quad & Q(\gFourIso,f_2) &= \gPthreeTwoIso
\end{eqnarray}

We observe that $H(\gFourStar) = \emptyset$, i.e. there are no invariant kernel functions in its complement, thus it doesn't produce any kernel relations. 

The above five equations in the six variables combine into the following linear system, with a one parameter family of solutions.
\begin{equation}
\left(\begin{matrix}
6 \\
4 \\
4\\
2\\
0\\
0
\end{matrix}\right) = 
\left(\begin{matrix}
5  & 0 & 0 & 0 &  0 \\
0  & 0 & 0 & 1 & 0\\
2  & 1 & 1 & 0 & 0\\
-1 & 2 & 2 & 0 & 0 \\
0  & 2 & -2& 0 & 1\\
0  & 0 & 0  & 0 & 0 \\
\end{matrix}\right)\frac{1}{5}
\left(\begin{matrix}
6\qquad\, \\
0\,\,+t \\
8\,\,-t \\
20 \,\qquad\,\\
16\,-4t
\end{matrix}\right).
\end{equation}
Thus, the degree $5$ orbit sum is in the image of the kernel elements.
\end{proof}

\subsubsection{Degree 7}

\[
\centering
\begin{tabular}{c|c|c|c}
Orbit $O$ & $|O|$ & $e(O)$ &  $w_{\bar O}$\\
\hline
$\overline{\gthreeIsoGhost}$ & 20 & 0  & -1 \\
$\overline{\gStarGhostThree}$ & 10 & 0 & -1 \\

$\overline{\gEdgeIsoThree}$ & 60 & 1  & 0\\

$\overline{\gPathThree}$ & 30 & 2   & 1\\
\end{tabular}
\]


\begin{proposition}\label{prop:degree7}
The degree 7 orbit sum lies in the kernel. Specifically,
\begin{eqnarray*}
\orbitsum_7   &=& -1\,\overline{\gStarGhostThree}-1\,\overline{\gthreeIsoGhost}+0\,\gEdgeIsoThree  +1\,\overline{\gPathThree} \\
&=&  Q(\overline{\gFourStar},f_1)-Q(\overline{\gFourIso},f_3).
\end{eqnarray*}

\end{proposition}
\begin{proof}
There are only two $6$ orbits with kernel functions on their complement, and they produce the two relations
\begin{equation}
Q(\overline{\gFourStar},f_1) = -\overline{\gStarGhostThree} + \overline{\gPathThree}.
\end{equation}
\begin{equation}
Q(\overline{\gFourIso},f_3) = \overline{\gthreeIsoGhost}.
\end{equation}
\end{proof}

\subsubsection{Degree 9}

\begin{lemma}\label{prop:degree9}
The degree 9 orbit sum $\orbitsum_9$ vanishes identically.
\end{lemma}
\begin{proof}
The only weight is $w(\bullet)=2(0)-1+1=0$, so there is no contribution at this degree.
\end{proof}

\subsubsection{Even degrees}

We note that the above cancellation does not happen in the even degrees. 

\begin{proposition}
The degree 2 orbit sum $\orbitsum_2$ does not lie in the kernel.
\end{proposition}
\begin{proof}
We start from
\[ \orbitsum_2 = 13\,\gPathTwo + 11\,\gIsoTwo. \]
The only kernel function in degree 2 is
\[ Q(\bullet, f_\Gamma+f_1) = 2\, \gPathTwo+ 1\,\gIsoTwo. \]
Let 
\[ \bullet_2 :=  \sum_{b} \ell_b^2,\quad \bullet :=  \sum_{b} \ell_b =  2Q(\emptyset, \,f_\Gamma). \]
We note that the action of $A$ on a basis of degree 2 monomials is
\[
A \begin{pmatrix} \gPathTwo \\ \gIsoTwo \\ \bullet_2 \end{pmatrix} = \begin{pmatrix} 5 & 2 & 0 \\ 4 & 6 & 3 \\ 0 & 2 & 3 \end{pmatrix} \begin{pmatrix} \gPathTwo \\ \gIsoTwo \\ \bullet_2 \end{pmatrix}.
\]
This system admits three eigenvectors $v_\lambda$,
\[ v_1= 2\gPathTwo-2\gIsoTwo+3\bullet_2  \]
\[ v_4=  -4\gPathTwo+1\gIsoTwo+3\bullet_2  \]
\[ v_9= 2\gPathTwo+2\gIsoTwo+\bullet_2 \]
We see the kernel contains
\[ v_4 = 3 (\bullet)^2 - 5 Q(\bullet, f_\Gamma+f_1) \]
\[v_9 = (\bullet)^2  = 4Q(\emptyset, \,f_\Gamma)^2\]
Thus,  we have
\[ \orbitsum_2 = 5Q(\bullet, f_\Gamma+f_1)+9Q(\emptyset, \,f_\Gamma)^2-\tfrac{3}{4} v_1 \]
As the vector $v_1$ is in the kernel of $A\text{-}1$ and thus does not vanish in $\Sym^2(V_{32})$.

\end{proof}

\subsection{Summary}

We have thus shown that 
\begin{equation}
\bigstanleypoly \in \Sym^{10}(V^{(1)}_{\{3,2\}}[\bbeta])^{S_5 \times \BZ_2}.
\end{equation}

\section{hyperplane factorization}

We now utilize the extra $\BZ_2$ invariance discussed in the previous section to reveal an unexpected \emph{hyperplane factorization} phenomenon.

\subsection{Boundary data}

The hook space has various \emph{boundary hyperplanes} defined by the equations $\bh_b^\sA = 0$. Relative to these hyperplanes, we need a subtle notion of equivalence for elements of the hook space.

\begin{definition}
Consider two Stanley sums $\bmss_1,\bmss_2 \in \StS_{\mu\nu}^{\lambda}$.

We write 
\begin{equation}
\bmss_1 \simeq \bmss_2,
\end{equation}
if they become equal when pulled back to the hook space, i.e. $\iota^*\bmss_1 = \iota^* \bmss_2\in \StDR_{\mu\nu}^{\lambda}$.

We write
\begin{equation}
\bmss_1 \simeq_b^\sA \bmss_2,
\end{equation}
for a box $b$ and hook choice $\sA$, to indicate when
\[ \iota^*\bmss_1 =   \iota^*\bmss_2 \in \StDR_{\mu\nu}^{\lambda}/\bh_b^\sA \]
That is, they agree in the hook space upon further imposing the relation $\bh_b^\sA = 0$.
\end{definition}

It is easy to check the following result.
\begin{lemma}\label{lemma:propofsim}
If $\bmss_1 \simeq \bmss_2$ and $\bmss_1 \simeq_b^{\sA} \bmss_3$ then $\bmss_2 \simeq_b^{\sA} \bmss_3$.
If $\bmss_1 \simeq_b^\sA \bmss_2$, then $\bar\bmss_1 \simeq_b^{\bar\sA} \bar\bmss_2$.
\end{lemma}

Returning back to $\bigstanleypoly \in (\StDR_{21,21}^{321})^{S_5 \times \BZ_2}$. In terms of the above notation, the $\BZ_2$ invariance \eqref{thm:z2invs} is expressed as
\begin{equation}
\bigstanleysum \simeq \bar\bigstanleysum.
\end{equation}

The following property is perhaps the most surprising of this paper, and is the property that we later claim generalizes for all Jack
Littlewood-Richardson polynomials, see \eqref{conj:adjacency}.
\begin{theorem}\label{prop:hyperplanesimp}
For each box $b$ and hook choice $\sA$, there exists a single diagram $\simpterm(b,\sA) \in \StD_{21,21}^{321}$, so called {\emph{\bf boundary datum}}, such that
\[ \bigstanleypoly \simeq_{b}^{\sA} \simpterm(b,\sA) . \]
Specifically, for each box $b$ we have
\begin{equation}
 \simpterm(b,\bmX_b) = \bmX^{\claw(b)}
\end{equation}
and
\[ \simpterm(b,\bar\bmX_b)  = \overline{\simpterm(b,\bmX_b)} = \bar{ \bmX}^{\claw(b)}\]
\end{theorem}
This theorem says that when restricted to any boundary hyperplane $\bh_b^\sA=0$, the sum of 26 diagrams in $\bigstanleysum$ collapses to a single diagram, \emph{up to the hook relations}. Note that \emph{none} of the individual boundary datum $\simpterm(b,\sA)$ are individually in the support in the sum $\bigstanleysum$, i.e. this is purely a hook space phenomena.

\begin{proof}
We first treat the boundary value at $b$ corresponding to the hook choice $\bar\bmX_b$, for which we need to show that $ \bigstanleysum \simeq_b^{\bar\bmX_b} \simpterm(b,\bar\bmX_b) := \bar \bmX^{\claw(b)}$. From the formula \eqref{eq:721sum}, the terms of the Stanley sum agreeing with $\bmX_b$ at the box $b$ are
\[
\bigstanleysum \simeq \bar \bigstanleysum \simeq_b^{\bar\bmX_b} \sum_{\bmD:\bmD_b = \bmX_b} c_\bmD \cdot \bar \bmD = -2 \bar \bmX^{\{b\}} + \sum_{a\sim b}\bar \bmX^{\{b,a\}}.
\]
We recall the half-kernel sum expression \eqref{def:kernelfunction} of
\[
k_f^0 := -2 \bar \bmX^{\{b\}}+\sum_{a \sim b} \bar \bmX^{\{b,a\}} -\bar\bmX^{\claw(b)}.
\]
Thus we have
\[
\bigstanleysum \simeq_b^{\bar\bmX_b} k_f^0+\bar\bmX^{\claw(b)}.
\]
Recall the full kernel function \eqref{def:kernelfunction}, 
\[
\bmk_f = k_f^0 - (k_f^0)^{\claw(b)},
\]
which vanishes in the hook space by Theorem~\ref{thm:kernelidentity}, $\bmk_f \simeq 0$. We can see that $\bmk_f \simeq_b^{\bmX_b} k_f^0$, and so by \eqref{lemma:propofsim} we can conclude that $k_f^0\simeq_b^{\bmX_b}0$. Therefore, modulo hook relations, the only surviving boundary contribution is the single diagram $\bar\bmX^{\claw(b)}$, which proves
\[
\bar\bigstanleysum \simeq_b^{\bar\bmX_b} \bar\bmX^{\claw(b)} =: \simpterm(b,\bar\bmX_b) .
\]

For the opposite hook choice, we use the extra conjugation symmetry. Since $\bigstanleysum \simeq \bar\bigstanleysum$, Lemma \ref{lemma:propofsim} implies
\[
\simpterm(b,\bmX_b)
=
\overline{\simpterm(b,\bar \bmX_b)}
=
\bmX^{\claw(b)}.
\]
\end{proof}

The proof of the above theorem demonstrates that one half of the boundary data follows quite easily from the structure of $\bigstanleysum$. However, determining the other half required use of the non-manifest $\BZ_2$ symmetry.

\begin{corollary}
The boundary data is equivariant under $S_5 \times \BZ_2$, i.e. 
\begin{equation}
g.\simpterm(b,\bmX_b) = \simpterm(g.b,g.\bmX_{b}).
\end{equation}
\end{corollary}
\begin{proof}
This follows from $g.\claw(b) = \claw(g.b)$.
\end{proof}

\begin{figure}
\[
 \frac{\begin{ytableau}
*(boxL) \bullet   \\
*(boxL) \* & *(boxU) 
\end{ytableau}\,\, \begin{ytableau}
*(boxU)  \\
*(boxU)   & *(boxL) \* 
\end{ytableau} 
}{\begin{ytableau}
*(boxB)  \* \\
*(boxU)  & *(boxL)   \\
*(boxU)   & *(boxL) \* & *(boxB) \*  \\
\end{ytableau}}
 \frac{\begin{ytableau}
*(boxL)  \* \\
*(boxL)  \bullet & *(boxL) \*
\end{ytableau}\,\, \begin{ytableau}
*(boxU) \\
*(boxU)   & *(boxU)
\end{ytableau} 
}{\begin{ytableau}
*(boxB) \\
*(boxU) & *(boxU)  \* \\
*(boxU)   & *(boxU) & *(boxB)  \\
\end{ytableau}}
 \frac{\begin{ytableau}
*(boxU)   \\
*(boxL) \*  & *(boxL)\bullet 
\end{ytableau}\,\, \begin{ytableau}
*(boxL) \* \\
*(boxU)   & *(boxU)
\end{ytableau} 
}{\begin{ytableau}
*(boxB) \\
*(boxL) \* & *(boxL)   \\
*(boxU)   & *(boxU) & *(boxB) \\
\end{ytableau}}
 \frac{\begin{ytableau}
*(boxU)   \\
*(boxU)   & *(boxL) \*
\end{ytableau}\,\, \begin{ytableau}
*(boxL) \bullet \\
*(boxL) \*  & *(boxU)  
\end{ytableau} 
}{\begin{ytableau}
*(boxB)  \\
*(boxU)  & *(boxL)   \\
*(boxU)   & *(boxL)\*  & *(boxB)  \\
\end{ytableau}}
 \frac{\begin{ytableau}
*(boxU)   \\
*(boxU)   & *(boxU) 
\end{ytableau}\,\, \begin{ytableau}
*(boxL) \* \\
*(boxL) \bullet  & *(boxL)  \* 
\end{ytableau} 
}{\begin{ytableau}
*(boxB) \*  \\
*(boxU)  & *(boxU) \*  \\
*(boxU)   & *(boxU)  & *(boxB) \* \\
\end{ytableau}}
 \frac{\begin{ytableau}
*(boxL) \*  \\
*(boxU)   & *(boxU) 
\end{ytableau}\,\, \begin{ytableau}
*(boxU)  \\
*(boxL) \*  & *(boxL)  \bullet
\end{ytableau} 
}{\begin{ytableau}
*(boxB) \*  \\
*(boxL) \* & *(boxL)   \\
*(boxU)   & *(boxU) & *(boxB) \*  \\
\end{ytableau}}
\]
\[
\frac{\begin{ytableau}
*(boxU)  \\
*(boxU) & *(boxL) \*
\end{ytableau}\,\, \begin{ytableau}
*(boxU) \\
*(boxU) & *(boxL)\*  
\end{ytableau} 
}{\begin{ytableau}
*(boxB) \*  \\
*(boxL)  \bullet & *(boxL)  \\
*(boxL) \* & *(boxU) & *(boxB) \* \\
\end{ytableau}} 
 \frac{\begin{ytableau}
*(boxU)  \\
*(boxL)  \* & *(boxU)
\end{ytableau}\,\, \begin{ytableau}
*(boxU) \\
*(boxL)  \* & *(boxU)
\end{ytableau} 
}{\begin{ytableau}
*(boxB)  \\
*(boxU) & *(boxU) \bullet \\
*(boxL) \*  & *(boxU) & *(boxB)  \\
\end{ytableau}}
 \frac{\begin{ytableau}
*(boxU)  \\
*(boxU) & *(boxU)
\end{ytableau}\,\, \begin{ytableau}
*(boxU) \\
*(boxU) & *(boxU)
\end{ytableau} 
}{\begin{ytableau}
*(boxB)\\
*(boxL)\* & *(boxU)\*  \\
*(boxL) \bullet & *(boxL)\* & *(boxB) \\
\end{ytableau}}
\frac{\begin{ytableau}
*(boxL)\*  \\
*(boxU) & *(boxU)
\end{ytableau}\,\, \begin{ytableau}
*(boxL) \*\\
*(boxU) & *(boxU) 
\end{ytableau} 
}{\begin{ytableau}
*(boxB)  \\
*(boxU)   & *(boxL)  \\
*(boxL) \* & *(boxL)\bullet & *(boxB) \nonumber\\
\end{ytableau}} 
\]
\caption{\label{fig:boundarydatum}Ten of the boundary datum $\simpterm(b,\sA)$ of $\bigstanleypoly$. Here the box $b$ is represented by a solid dot ($\bullet$), and the 3 adjacent vertices in $\claw(b)$ are stars ($\*$). The other ten datum are compliments of these. }
\end{figure}

\begin{corollary}
The Jack Littlewood-Richardson polynomial $\bigstanleypoly \in \StDR$, when considered as a polynomial in the hook space variables variables $\ell_b \in V_{32} \subset V_\petersen$ , factors into a product of linear factors along each of the 20 lines $\ell_b = \pm \beta$.
\end{corollary}
We continue with this factorization property in section \eqref{sect:hyperplanesagain}.

\section{ Extension to $S_6$ }

In this section, we will demonstrate that the $S_5 \times \BZ_2$ symmetry group of the Jack Littlewood-Richardson polynomial $\bigstanleypoly$ has an enhancement to $S_6 \times \BZ_2$. For this we need to understand the structure of the exotic outer automorphism of $S_6$.

\subsection{The outer automorphism of $S_6$}
We draw heavily on the constructions in \cite{Howard:2008aa}.

A \emph{mystic pentagon} $P$ is a 2-coloring of the edges of $K_5$, such that the edges of one color form a 5-cycle (and thus the other colored edges do also). There are six distinct such pentagons, up to exchanging of the two colors.

$P$ determines a bijection between edges in $K_5 = K([6]/q)$, where, white edge $A$ is sent to black edge $B$ when $A$ and $B$ don't share any vertices. The remaining vertex is $e$. Such a bijection determines a Pentad $\Sigma_P$ consisting of the 5 synthemes $\{A/B/\{e,q\}\}$.

\begin{example}
The mystic pentagon $(1,2),(2,3),(3,4),(4,5),(1,5)$ corresponds to the pentad,
\begin{eqnarray*}
\{ (1,2) / (3,5) / (4,6) \} \\
\{ (2,3) / (1,4) / (5,6) \} \\
\{ (3,4) / (2,5) / (1,6) \} \\
\{ (4,5) / (1,3) / (2,6) \} \\
\{ (1,5) / (2,4) / (3,6) \}
\end{eqnarray*}
\end{example}

\subsection{A $\BZ_4$ action on $3$-subsets}

For a choice of mystic pentagon $P$ we construct a permutation $\rho_P \in \mathrm{Perm}\binom{[6]}{3}$ on the set of $3$-subsets of six elements. Let $q \in [6]$ be the vertex not included in $P$. 

For $T \in \binom{[6]}{3}$, either $q \in T$, or $q \in T^c$, so denote by $(-1)^{q \notin T} T$ the one that contains $q$. From this, define the duad
\[ d_T := ((-1)^{q \notin T} T)\setminus\{q\}.  \]
Note $d_T = d_{T^c}$. $d_T$ is a member of a unique syntheme $S \in \Sigma_P$, of the form
\[ S =\{d_T/d'/\{e,q\}\}. \]
We define $D(T) :=  \{q\} \cup d' \in \binom{[6]}{3}$. Then we see that $D(D(T)) = \{q\} \cup d_T = (-1)^{q \notin T} T$, and the only four 3-subsets that map to $S$ are $d^{-1}(S) = \{T,D(T),T^c,D(T)^c\}$. The map $\rho$ will cycle these four 3-subsets, the only question is in which direction. For a duad $d$, let $(-1)^{P(d)} = 1$ if $P|_d$ is white, and $(-1)$ otherwise. We define
\[ \rho_P(T) := (-1)^{q \notin T}(-1)^{P({d_T})} D(T). \]

\begin{corollary}
\[ \rho^2_P = c\]
Thus $\rho_P$ is order $4$, and splits up the 20 3-subsets into 5 4-cycles of the form $(T, U, T^c , U^c)$.
\end{corollary}

\begin{figure}[h]
\centering
\begin{tikzpicture}[
    every node/.style={draw, minimum size=0.9cm},
    arrow/.style={->, thick, >=stealth},
    syntheme/.style={draw=none, align=center}
]

\begin{scope}[shift={(0,0)}]
    \node (A1) at (0,2) {$126$};
    \node (B1) at (2,2) {$124$};
    \node (C1) at (2,0) {$345$};
    \node (D1) at (0,0) {$356$};
    \node[syntheme] at (1,1) {$(1\,2)$\\$(3\,5)$\\$(4\,6)$};
    \draw[arrow] (A1) -- (B1);
    \draw[arrow] (B1) -- (C1);
    \draw[arrow] (C1) -- (D1);
    \draw[arrow] (D1) -- (A1);
\end{scope}

\begin{scope}[shift={(4.5,0)}]
    \node (A2) at (0,2) {$136$};
    \node (B2) at (2,2) {$456$};
    \node (C2) at (2,0) {$245$};
    \node (D2) at (0,0) {$123$};
    \node[syntheme] at (1,1) {$(4\,5)$\\$(1\,3)$\\$(2\,6)$};
    \draw[arrow] (A2) -- (B2);
    \draw[arrow] (B2) -- (C2);
    \draw[arrow] (C2) -- (D2);
    \draw[arrow] (D2) -- (A2);
\end{scope}

\begin{scope}[shift={(9,0)}]
    \node (A3) at (0,2) {$146$};
    \node (B3) at (2,2) {$236$};
    \node (C3) at (2,0) {$235$};
    \node (D3) at (0,0) {$145$};
    \node[syntheme] at (1,1) {$(2\,3)$\\$(1\,4)$\\$(5\,6)$};
    \draw[arrow] (A3) -- (B3);
    \draw[arrow] (B3) -- (C3);
    \draw[arrow] (C3) -- (D3);
    \draw[arrow] (D3) -- (A3);
\end{scope}

\begin{scope}[shift={(2.25,-4.5)}]
    \node (A4) at (0,2) {$156$};
    \node (B4) at (2,2) {$135$};
    \node (C4) at (2,0) {$234$};
    \node (D4) at (0,0) {$246$};
    \node[syntheme] at (1,1) {$(1\,5)$\\$(2\,4)$\\$(3\,6)$};
    \draw[arrow] (A4) -- (B4);
    \draw[arrow] (B4) -- (C4);
    \draw[arrow] (C4) -- (D4);
    \draw[arrow] (D4) -- (A4);
\end{scope}

\begin{scope}[shift={(6.75,-4.5)}]
    \node (A5) at (0,2) {$256$};
    \node (B5) at (2,2) {$346$};
    \node (C5) at (2,0) {$134$};
    \node (D5) at (0,0) {$125$};
    \node[syntheme] at (1,1) {$(3\,4)$\\$(2\,5)$\\$(1\,6)$};
    \draw[arrow] (A5) -- (B5);
    \draw[arrow] (B5) -- (C5);
    \draw[arrow] (C5) -- (D5);
    \draw[arrow] (D5) -- (A5);
\end{scope}

\end{tikzpicture}
\caption{The five 4-cycles of $\rho_P$ on the 20 triples. Each square shows $T \to \rho(T) \to T^c \to \rho(T)^c$. Inside each cycle is the corresponding syntheme in $\Sigma_P$.}
\end{figure}
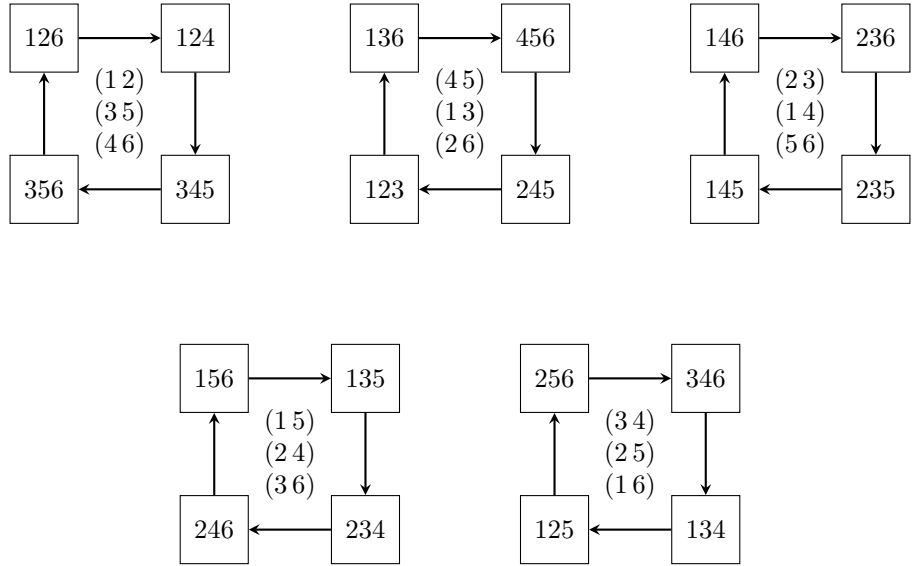

\subsubsection{The Johnson Graphs $J(n,k,i)$}

For an $n$-element set, $[n]$, consider the set of all $k$-subsets of elements, denoted $\binom{[n]}{k}$, and define the distance between two $k$-subsets as $d(T,T') := k - |T\cap T'|$. The set $\binom{[n]}{k}$ forms the vertices of the Johnson graph $J(n,k,i)$, with an edge $T \sim T'$ between two $k$-subsets $T,T'$ when $|T\cap T'| = i$, i.e. $T$ and $T'$ share $i$ elements, so $d(T,T') = 1$. Clearly, this distance agrees with the usual graph theoretic distance.

We will mainly be interested in the Johnson Graph, $\johnson := J(6,3) := J(6,3,2)$, 
which has $\binom{6}{3}=20$ vertices, and each vertex has $9$ neighbours differing by one element. 


\begin{theorem}[\protect{\cite[Theorem 9.1.2]{Brouwer:1989aa}}] 
\[ \Aut(J(2k,k)) = S_{2k} \times \BZ_2 \]
where $S_{2k}$ acts in the standard way on the set of $2k$ elements, and $\BZ_2$ is generated by conjugation $C$.
\end{theorem}

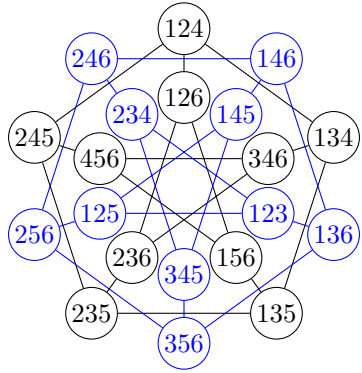
\begin{figure}
\begin{equation}
\begin{tikzpicture}[scale=1.3, every node/.style={circle, draw, inner sep=1.2pt}]
  \node (o1) at ( 90:1.6) {$124$};
  \node (o2) at ( 18:1.6) {$134$};
  \node (o3) at (-54:1.6) {$ 135$};
  \node (o4) at (-126:1.6) {$ 235$};
  \node (o5) at (162:1.6) {$245$};
  \node[blue] (O1) at ( 126:1.6) {$246$};
  \node[blue] (O2) at ( 54:1.6) {$146$};
  \node[blue] (O3) at (-18:1.6) {$ 136$};
  \node[blue] (O4) at (-90:1.6) {$ 356$};
  \node[blue] (O5) at (198:1.6) {$256$};
  \node (i1) at ( 90:0.9) {$126$};
  \node (i2) at ( 18:0.9) {$346$};
  \node (i3) at (-54:0.9) {$156$};
  \node (i4) at (-126:0.9) {$236$};
  \node (i5) at (162:0.9) {$456$};
  \node[blue] (I1) at ( 126:0.9) {$234$};
  \node[blue] (I2) at ( 54:0.9) {$145$};
  \node[blue] (I3) at (-18:0.9) {$ 123$};
  \node[blue] (I4) at (-90:0.9) {$ 345$};
  \node[blue] (I5) at (198:0.9) {$125$};
  \draw (o1)--(o2)--(o3)--(o4)--(o5)--(o1);
  \draw[blue] (O1)--(O2)--(O3)--(O4)--(O5)--(O1);
  \draw (o1)--(i1);
  \draw (o2)--(i2);
  \draw (o3)--(i3);
  \draw (o4)--(i4);
  \draw (o5)--(i5);
  \draw[blue] (O1)--(I1);
  \draw[blue] (O2)--(I2);
  \draw[blue] (O3)--(I3);
  \draw[blue] (O4)--(I4);
  \draw[blue] (O5)--(I5);
  \draw (i1)--(i3)--(i5)--(i2)--(i4)--(i1);
  \draw[blue] (I1)--(I3)--(I5)--(I2)--(I4)--(I1);
\end{tikzpicture}
\end{equation}
\caption{The 2-coloring of the vertices of $J(6,3)$ that partitions it into two copies of the Petersen graph. The copies are exchanged under complementation. $\rho$ permutes the four elements on any line that goes through the origin, and exchanges the inner ring with the outer ring.}\label{fig:johnsontwopeterson}
\end{figure}

\begin{lemma}
\[ d(T,U) = 2 \leftrightarrow d(\rho(T),\rho(U)) = 1\]
Thus $\rho$ is not an automorphism of the Johnson graph, but rather implements an isomorphism between $J(6,3,2)$ and $J(6,3,1)$.
\end{lemma}

As noted in \cite{Howard:2008aa} section 1.2, there is a bijection between mystic pentagons and 2-colorings of $\johnson$ such that each $T$ and $T^c$ have different colors and every $K_4$ subgraph has two vertices of each color. The next construction agrees with this bijection.

\begin{corollary}
For $(ij)$ in $\binom{[5]}{2}$, the map 
\begin{equation}\label{eq:embedpettojohn}
\gamma : (ij) \mapsto \rho_P(ij6)
\end{equation} gives an embedding of the Petersen graph $\petersen$ into the Johnson graph $\johnson$. Composing with conjugation gives a complementary embedding $\gamma^c$, which provides a partition of the vertices $\johnson = \gamma(\petersen) \cup {\color{blue}{\gamma^c(\petersen)}}$. As we vary the choice of pentad $P$, we find all 12 embeddings $\petersen \hookrightarrow \johnson$.
\end{corollary}

The stabiliser of this $\gamma(\petersen)$ is the $S_5=\Aut(\petersen)$ conjugated by $\rho$.

\begin{lemma}\label{lemma:intertwiningaction}
The embedding $\sigma_{ij} \mapsto s_{ij}$ of $\Aut(\petersen)=S_5 \hookrightarrow \Aut(\johnson)= S_6 \times \BZ_2$ that intertwines $\gamma$ takes the form
\[ s_{12}  = -(46)(15)(23),\]
\[ s_{23}  = -(56)(12)(34),\]
\[ s_{34}  = -(16)(23)(45),\]
\[ s_{45}  = -(26)(15)(34).\]
That is,
\[ s_{ij} \gamma = \gamma \sigma_{ij}. \]
This is the embedding $(\tau, \sgn)$, where $\tau$ is the exotic inclusion of $S_5 \hookrightarrow S_6$ determined by the pentad $P$.
\end{lemma}

\subsection{Vertex Representations}

Let $\tilde \johnvertrep$ be the vertex representation of $\Aut(\johnson) = S_6\times \BZ_2$, i.e.  the vector space with a generator $y_T$ for each $3$-subset $T$. $\Aut(\johnson)$ acts on $\tilde \johnvertrep$ by $g.y_T =  y_{g.T}$. The adjacency matrix $ A_\johnson$ commutes with $\Aut(\johnson)$, under which we find
\begin{lemma}
Under $\Aut(\johnson) \times \BC[A_\johnson]$, we have
\begin{equation}\label{eq:VJdecomp}
\tilde \johnvertrep \cong V_{6,\id}^{(9)} \oplus V_{42,\id}^{(-1)} \oplus V_{33,\sgn}^{(-3)} \oplus V_{51,\sgn}^{(3)},
\end{equation}
with dimensions
\[ (20) =  (1) + (9) + (5) + (5).\]
\end{lemma}
\begin{proof}
The spectrum and multiplicities are given in \protect{\cite[Theorems 9.1.2, 8.4.2-3]{Brouwer:1989aa}}. We describe each of these two $5$-dimensional sub-representations.
The map $\iota:V_{51} \to \johnvertrep$ is given by $\{x_i\}  \mapsto y_T = x_{T} := \sum_{t\in T} x_t$, which satisfies $x_{T} + x_{T^c} = 0$. 

The sub-representation $V_{33} \subset \johnvertrep$ is spanned by the elements
\begin{equation}\label{prop:v33inm3}
\gamma(cl_b) := y_{\gamma(b)} - \sum_{a\sim b} y_{\gamma(a)} 
\end{equation}
where $\gamma$ is the embedding \eqref{eq:embedpettojohn} of the Petersen graph into the Johnson graph.
 These are clearly invariant under $\gamma(\Aut(\petersen))$.

\end{proof}

We will consider the 10-dimensional anti-symmetric vertex representation, $\johnvertrep:= (\tilde \johnvertrep)_{\sgn}$, that is, we impose $y_T = - y_{T^c}$. Following from \eqref{eq:VJdecomp}, this is given as
\[ \johnvertrep \cong V_{33,\sgn}^{(-3)} \oplus V_{51,\sgn}^{(3)}. \]

\subsubsection{Twisting} 

The outer automorphism $\tau$ of $S_6$ twists its representation theory. The four $5$-dimensional representations $\lambda$ of $S_6$ have character tables entries for cycle types $\mu$ given by

\begin{center}
\begin{tabular}{c|ccc}
$\lambda$ $\backslash$ $\mu$ & $1^6$ & $21^4$ & $2^3$  \\
\hline 
$51$ & $5$ & $3$ & $-1$ \\ 
$21111$ & $5$ & $-3$ & $1$ \\ 
$222$ & $5$ & $-1$ & $3$ \\ 
$33$ & $5$ & $1$ & $-3$ \\ 
\end{tabular}
\end{center}
As $\tau$ sends an element of cycle type $(21^4)$ to one of type $(2^3)$, we find
\[ \tau : V_{33} \leftrightarrow V_{21111}, \quad  V_{222} \leftrightarrow V_{51} \]
On the other hand, twisting by the sign character induces
\[ \sgn : V_{51} \leftrightarrow V_{21111}, \quad  V_{222} \leftrightarrow V_{33} \]
and thus
\[ \sgn\otimes\tau : V_{33} \leftrightarrow V_{51}. \]
Consider the map $\tilde \rho \in \mathrm{End}(\johnvertrep)$ which acts by
\begin{equation}
\tilde\rho : x_T \mapsto x_{\rho(T)}, 
\end{equation}

\begin{lemma}
The map $\tilde \rho$ is an $S_6$-module isomorphism of $\johnvertrep$ with itself twisted by $\sgn \otimes \tau$.
\[ \tilde \rho : \johnvertrep \to \johnvertrep^{\sgn \otimes \tau}. \]
\end{lemma}
\begin{proof}
\[ \tilde \rho : \johnvertrep=V_{33}\oplus V_{51} \rightarrow \johnvertrep^{\sgn\otimes\tau}=V_{33}^{\sgn\otimes\tau}\oplus V_{51}^{\sgn\otimes\tau}=  V_{51}\oplus V_{33} \cong \johnvertrep \]
\end{proof}

\subsubsection{The Petersen Graph $K(5,2)$}

Consider the $2$-subsets $b$ of $5$ elements $[5]$. We form the basis $\ell_{b}$ for $b=\{i,j\}$. Recall the decomposition \eqref{eq:petersenrepdecomp}
\[ \petvertrep = \left( V_{5,\sgn}^{(3)} \oplus V_{41,\sgn}^{(-2)}\right) \oplus V_{32,\sgn}^{(1)}. \]

\begin{theorem}
The map $\psi : \petvertrep \to \johnvertrep$, defined by 
\[ \psi : \ell_{\{ij\}} \mapsto  x_{\rho(\{ij6\})} \]
\begin{eqnarray*}
  \ell_{12}  \to -x_{356}, &   \ell_{23}  \to -x_{146}, &   \ell_{34}  \to -x_{256}, \\
  \ell_{13}  \to +x_{456}, &   \ell_{24}  \to +x_{156}, &   \ell_{35}  \to +x_{126}, \\
  \ell_{14}  \to +x_{236}, &   \ell_{25}  \to +x_{346}, &   \ell_{45}  \to -x_{136}, \\
  \ell_{15}  \to -x_{246}. \\
\end{eqnarray*}
is an intertwining isomorphism of $S_5$ representations, twisted by $\tau \otimes \sgn$, i.e.
\begin{equation}
\psi(\sigma \ell_b) = \sgn(\sigma) \tau(\sigma) \psi(\ell_b).
\end{equation}
\[ s_{ij} := \psi\sigma_{ij}\psi^{-1} \]
\end{theorem}
This corresponds to the $S_6 \to S_5$ branching
\[ \johnvertrep^{ \tau\otimes\sgn} = V_{33}^{\tau\otimes\sgn} \oplus V_{51}^{\tau\otimes\sgn} = V_{51}\oplus V_{33} \rightarrow (V_{5} \oplus V_{41})\oplus V_{32} = \petvertrep \]

\subsubsection{Lift to $S_6$}

With the choice of intertwining representation as given in \eqref{lemma:intertwiningaction}, we find the 5th braid generator of $S_6$ to be
\[  s_{56} = -(65)(14)(23). \]

\begin{lemma}
The action of $s_{56}$ on $\johnvertrep$ descends to an action of $\sigma_{56} := \psi^{-1}s_{56}\psi$ on $\petvertrep$ as follows
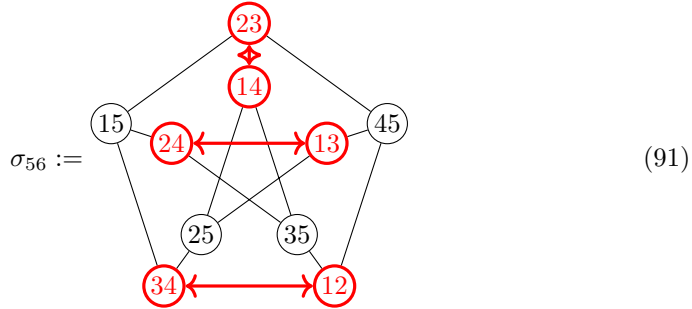
\begin{equation}\label{img:actionofsigma}
\sigma_{56} :=
\begin{tikzpicture}[scale=1.2, every node/.style={circle, draw, inner sep=1.2pt}, baseline={([yshift=-1.0ex]current bounding box.center)}]
  \node[red, very thick] (o1) at ( 90:1.6) {$23$};
  \node (o2) at ( 18:1.6) {$45$};
  \node[red, very thick] (o3) at (-54:1.6) {$\color{red}{ 12}$};
  \node[red, very thick] (o4) at (-126:1.6) {$\color{red}{ 34}$};
  \node (o5) at (162:1.6) {$15$};
  \node[red, very thick] (i1) at ( 90:0.9) {$\color{red}{14}$};
  \node[red, very thick] (i2) at ( 18:0.9) {$\color{red}{ 13}$};
  \node (i3) at (-54:0.9) {$35$};
  \node (i4) at (-126:0.9) {$25$};
  \node[red, very thick] (i5) at (162:0.9) {$\color{red}{ 24}$};
  \draw (o1)--(o2)--(o3) (o4)--(o5)--(o1);
  \draw[<->, red, very thick] (o3)--(o4);
  \draw[<->, red, very thick] (o1)--(i1);
  \draw (o2)--(i2);
  \draw (o3)--(i3);
  \draw (o4)--(i4);
  \draw (o5)--(i5);
  \draw (i1)--(i3)--(i5)  (i2)--(i4)--(i1);
  \draw[<->, red, very thick] (i5)--   (i2);  
\end{tikzpicture}
\end{equation}
where the basis elements moved along the red arrow incur a sign.
E.g.
\[ \sigma_{56}\ell_{23} = -\ell_{14},\, \sigma_{56}\ell_{15} = \ell_{15}. \]

This action preserves the $5\oplus 5$ decomposition of $V_\petersen = (V_5 \oplus V_{41})\oplus V_{32}$. 
The $S_4 \subset S_5$ acting on $[4] \subset [5]$ commutes with the action of $\sigma_{56}$
  
\end{lemma}

\begin{proof}
For this, we look at the claw relation $cl_b := \ell_b-\sum_{a\sim b} \ell_a=0$ of any vertex $b$ and show that$\sigma$ maps it to the linear span of such claws, hence the subspace $V_{32}$ is invariant. Up to $S_4$ invariance, there are only two such claws. Either $b$ is fixed or $b$ is swapped by $\sigma$. Take the fixed vertex $(15)$ in the graph \ref{img:actionofsigma}. We find $\sigma (cl_{(15)}) = (15) + (14) + (13) + (12) = (15)- cl_{(25)} + (25) - (34) + (12) = -cl_{(25)} - cl_{(34)}$. For the non-fixed vertex $24$, we find $\sigma (cl_{(24)}) = cl_{(24)}$, which finishes the proof.
\end{proof}

In terms of the action on hooks, we have
\ytableausetup{boxsize=1.4em}
\[ \sigma_{56} = 
\frac{\begin{ytableau}
a_1 \\
\color{pGray}{ a_2} & a_3
\end{ytableau}\,\, \begin{ytableau}
\color{pGray}{b_1} \\
\color{pGray}{b_2} & \color{pGray}{b_2}
\end{ytableau} 
}{\begin{ytableau}
*(boxB)  \\
\color{pGray}{c_2}   & c_3 \\
c_4 & \color{pGray}{c_5} & *(boxB) \nonumber\\
\end{ytableau}} 
\to 
\frac{\begin{ytableau}
*(boxS)\bar b_3  \\
\color{pGray}{ a_2} &  *(boxS)\bar b_1
\end{ytableau}\,\, \begin{ytableau}
*(boxS) \bar a_3 \\
\color{pGray}{b_2} & *(boxS)\bar a_1 
\end{ytableau} 
}{\begin{ytableau}
*(boxB)  \\
\color{pGray}{c_2}   &  c_4 \\
 c_3  & \color{pGray}{c_5}  & *(boxB) \nonumber\\
\end{ytableau}} 
\]

\subsection{Extra invariance}

\begin{proposition}
$\bigstanleypoly \in \Sym^{10}(V_{32}[\bbeta])^{S_5\times \BZ_2}$ is invariant under $\sigma_{56} \in \mathrm{End}(V_{32})$.
\end{proposition}

\begin{proof}

Under the branching $S_5$ to this $S_4$, we have $V_{32} \to V_{22} \oplus V_{31}$, a $2\oplus 3$ split, on which $\sigma_{56}$ must act as block diagonal $\sigma_{56} = (-1)^a \oplus (-1)^b$. Since $V_{32} \cong V_{33}$ as an $S_6$ representation, $\mathrm{Tr}_{V_{33}}(\sigma_{56})=1=2(-1)^a + 3 (-1)^b$. Thus $\sigma_{56}$ acts as $(-1) \oplus (1)$.

Since $S_4$ acts as reflections on $V_{22}$, $\BC[V_{22}]^{S_4} \cong \BC[f_2,f_3]$ is a polynomial ring. By standard theory, $\BC[V_{31}]^{S_4} \cong \BC[e_2,e_3]$,  So
\[ \BC[V_{32}]^{S_4} = \BC[f_2,f_3^\sgn,e_2,e_3] \]
On which $\sigma_{56}$ acts as $-1$ on $f_3$, and the other generators are invariant. Thus, the only polynomials of degree 10 or less that are anti-invariant are
\[ f_3e_3 \{1, e_2, f_2, e_2^2, e_2f_2,f_2^2 \} \]
Thus, any anti-invariant must have at degree $6, 8$ or $10$. We prove that $\bigstanleypoly$ has no anti-invariant component in each of these degrees. 
First, in degree 10, $(\bigstanleypoly)_{(10)} = 2\prod_{b} \ell_b$, which under $\sigma_{56}$ produces a factor of $(-1)^6$, and hence is invariant.

In degree 8, we  split up the two $S_5$ orbits into $S_4$ orbits, then compute the action of $\sigma_{56}$ on those orbits.

\begin{figure}
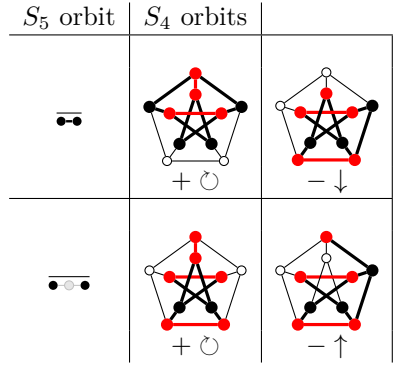

\begin{tabular}{c|c|c|c|c|c}
$S_5$ orbit & $S_4$ orbits  \\
\hline & & \\
$\overline{\gPathTwo}$ & \geighta  & \geightii \\
       & $+\circlearrowright$ & $-\downarrow$ \\
\hline & & \\
$\overline{\gIsoTwo}$ & \geightaa & \geightiii  \\
       & $+\circlearrowright$ & $-\uparrow$ \\
\end{tabular}
\caption{\label{tab:deg6s4orbits}The two orbits under $S_5$ of eight vertices of the Petersen graph decomposed into their four orbits under $S_4$.  Under each orbit is the sign incurred under $\sigma_{56}$, and if two orbits are interchanged, which in this case the two in the last column are.}
\end{figure}

We show that 
\[ \tfrac{1}{2}(1\text{-}\sigma_{56}) \bigstanleypoly = \sum_{pp' \in P_2} \tfrac{1}{2}(w_{p}+w_{p'}) ( O_{p}+O_{p'} ) =: \sum_i \tfrac{1}{2}\, W_i \cdot \mathcal{O}_i \]
We write $e_I = \tfrac{1}{2}\, W_i \cdot \mathcal{O}_i$, distinguishing the sum of weights $W_i$ from the anti-invariant sum of orbits $\mathcal{O}_i$.

There are 2 invariant orbits, and one pair with sign $-1$. This one signed pair contributes $\tfrac{1}{2}(1\text{-}\sigma_{56})(\bigstanleypoly)_{(8)} = E_0$ where
\[ E_0 = \tfrac{1}{2}\left(w(\gPathTwo) + w(\gIsoTwo) \right) \left( \geightii + \geightiii \right) \]

We require $E_0$ to vanish, which can only happen via one of two ways: either the sum of weights $(W_0)$ vanishes, or the $S_4$-orbit sum $(\mathcal{O}_0)$ vanishes under the pullback to $V_{32}$. In degree 8, there are no kernel functions on $\petersen$ of the form $(1\text{-}A_\petersen)f$ that are supported on only two vertices. Thus, we cannot show that the orbit sums $\mathcal{O}_0$ vanish when pulled back to $V_{32}$. So, we are required to force the weight sum $W_0$ to vanish. This is
\[ W_0 := w(\gPathTwo) + w(\gIsoTwo) = 0\]
This holds for our specific weights given by \eqref{eq:cvalues}, $w(\gPathTwo) = 1$, $w(\gIsoTwo)=-1$, and so there is no anti-invariant contribution in degree 8.

In degree 6, we require more work. There are 6 $S_5$-orbits which split into $17$ $S_4$-orbits. 

Referring to \eqref{img:actionofsigma}, we split the set of vertices into 6 red and 4 black. Similarly, there are three red edges, and twelve black.

The first row, type "C", there is one choice for the color of the isolated vertex, and then what color the center of the 4-chain is. A total of four orbits. Similarly for the second row, type "E", there's two choices for the color of the center vertex is, and then two choices for the color of the isolated edge attached to that vertex. Total of four.

Of these, we see there are seven invariant orbits, one pair that exchanges with sign $(+1)$, and four pairs with sign $-1$. 

The $\sigma:= \sigma_{56}$ action on the $S_4$ orbits is displayed in figure \eqref{table:s4orbits}.

\begin{figure}
\begin{tabular}{c|c|c|c|c|c}
$S_5$ orbit & $S_4$ orbits  \\
\hline & & & & \\
\gsixC & \gsixCione & \gsixCii & & \gsixCaii & \gsixCiii \\
   60    & 12 $-\downarrow$ &  12 $-\downarrow$ & & 24 $+\downarrow$ & 12 $+\circlearrowright$\\
\hline & & & & \\
\gsixE & \gsixEi & \gsixEii & \gsixEiii & & \gsixEaaa\\
    60   & 12 $-\uparrow$ & 12 $-\uparrow$  & 24 $-\downarrow$ & & 12 $+\circlearrowright$\\
\hline & & & & \\
\gsixQ & & & \gsixQi & \gsixQaa & \gsixQiii \\
    60   & & &  24 $-\uparrow$ & 24 $+\uparrow$  & 12 $+\circlearrowright$ \\
\hline & & & &  \\
\gsixO & \gsixOi & \gsixOa & & & \\
   10    & 6 $+\circlearrowright$ &4 $-\downarrow$ & & \\
\hline & & & & \\
\gsixXi & \gsixXIi & \gsixXIa & & & \\
   5    & 1 $+\circlearrowright$&4 $-\uparrow$& & \\
\hline & & & & \\
\gsixH & \gsixHIi & \gsixHIii   & &\\
  15     & 3 $+\circlearrowright$ & 12 $+\circlearrowright$  & & 
\end{tabular}
\caption{\label{table:s4orbits}The six $S_5$-orbits of six vertices of the Petersen graph decomposed into their seventeen $S_4$-orbits. Under each orbit is (orbit size, sign incurred under $\sigma_{56}$, and if orbit interchange occurs.}
\end{figure}

The four contributions are $\tfrac{1}{2}(1\text{-}\sigma)(\bigstanleypoly)_{(6)} = E_1+E_2+E_3+E_4$, where
\begin{eqnarray}
E_1 &=& \tfrac{1}{2}\left(w(\gFourStar) + w(\gFourIso) \right) \left( \gsixOa +\gsixXIa \right)
\\
E_2 &=&  \tfrac{1}{2}\left(w(\gPthreeIso) + w(\gPtwoIsotwo) \right) \left( \gsixEi + \gsixCione \right)
\\
E_3 &=&  \tfrac{1}{2}\left(w(\gPthreeIso) + w(\gPtwoIsotwo) \right) \left( \gsixEii + \gsixCii \right)
\\
E_4 &=&  \tfrac{1}{2}\left(w(\gPFour) + w(\gPthreeIso) \right) \left( \gsixQi +\gsixEiii \right)
\end{eqnarray}
We require this to vanish, similarly to the degree 8 case. We will show that the last three orbit sums vanish, and one can easily check that the first three weight sums vanish.

In degree 6, we find only two $S_5$ kernel functions
\[Q_1 := \begin{tikzpicture}[scale=0.6,transform shape, baseline=(current bounding box.center), every node/.style={circle, draw, inner sep=1.2pt}]
  \node (o1) at ( 90:1.6) {$0$};
  \node (o2) at ( 18:1.6) {$1$};
  \node[blue, thick, fill] (o3) at (-54:1.6) {$0$};
  \node[blue, thick, fill] (o4) at (-126:1.6) {$0$};
  \node (o5) at (162:1.6) {$1$};
  \node[blue, thick, fill] (i1) at ( 90:0.9) {$0$};
  \node[blue, thick, fill] (i2) at ( 18:0.9) {$0$};
  \node (i3) at (-54:0.9) {$1$};
  \node (i4) at (-126:0.9) {$1$};
  \node[blue, thick, fill] (i5) at (162:0.9) {$0$};
  \draw (o1)--(o2)--(o3) (o4)--(o5)--(o1);
  \draw[blue, very thick] (o3)--(o4);
  \draw (o1)--(i1);
  \draw (o2)--(i2);
  \draw (o3)--(i3);
  \draw (o4)--(i4);
  \draw (o5)--(i5);
  \draw (i1)--(i3)--(i5)  (i2)--(i4)--(i1);
  \draw[blue, very thick] (i2)--(i5);
\end{tikzpicture} = 1\,\gsixE + 1\,\gsixC  + 0\,\gsixQ. \]

\[Q_2 := \begin{tikzpicture}[scale=0.6,transform shape, baseline=(current bounding box.center), every node/.style={circle, draw, inner sep=1.2pt}]
  \node (o1) at ( 90:1.6) {$1$};
  \node[blue, thick, fill] (o2) at ( 18:1.6) {$1$};
  \node[blue, thick, fill] (o3) at (-54:1.6) {$0$};
  \node[blue, thick, fill] (o4) at (-126:1.6) {$0$};
  \node[blue, thick, fill] (o5) at (162:1.6) {$1$};
  \node (i1) at ( 90:0.9) {$\text{-}1$};
  \node (i2) at ( 18:0.9) {$0$};
  \node (i3) at (-54:0.9) {$1$};
  \node (i4) at (-126:0.9) {$1$};
  \node[blue, thick, fill] (i5) at (162:0.9) {$0$};
  \draw (o1)--(o2) (o5)--(o1);
  \draw[blue, very thick] (o2)--(o3)--(o4)--(o5);
  \draw (o1)--(i1);
  \draw (o2)--(i2);
  \draw (o3)--(i3);
  \draw (o4)--(i4);
  \draw[blue, very thick] (o5)--(i5);
  \draw (i1)--(i3)--(i5)--(i2)--(i4)--(i1);
\end{tikzpicture} = 1\,\gsixE -1\,\gsixC  +2\,\gsixQ. \]

Since $\sigma$ commutes with $\iota^*$ (i.e. $\sigma$ preserves $V_{32}$), then $\iota^* Q_i = 0$ implies $\iota^*\tfrac{1}{2}(1\text{-}\sigma)Q_i = 0$. Upon splitting from $S_5$ to $S_4$, we look at the rows of table \ref{table:s4orbits}, to see
\[ \tfrac{1}{2}(1\text{-}\sigma)\gsixE \to \mathcal{O}_2 + \mathcal{O}_3 +\mathcal{O}_4, \quad \tfrac{1}{2}(1\text{-}\sigma)\gsixC \to \mathcal{O}_2 + \mathcal{O}_3, \quad \tfrac{1}{2}(1\text{-}\sigma)\gsixQ \to \mathcal{O}_4 \]
where, as before, $\mathcal{O}_i$ is the anti-invariant orbit pair in $e_I$. Hence, 
\[ Q_1 = 2\mathcal{O}_2 + 2\mathcal{O}_3 +\mathcal{O}_4,\quad Q_2= 3 \mathcal{O}_4. \]
Both of these combinations vanish in the pullback to $V_{32}$, i.e. 
\[ \iota^*\mathcal{O}_4 = 0= \iota^*(\mathcal{O}_2 + \mathcal{O}_3) \,\in \Sym^{6}(V_{32}). \]
Now, we have $W_2=W_3$ so our combination $E_2 + E_3 = \tfrac{1}{2}w_2(\mathcal{O}_2 + \mathcal{O}_3)$ vanishes, and so does $E_4 \sim \mathcal{O}_4$. Thus the last three orbit sums contributions vanish, $E_2+E_3+E_4=0$.

However, our orbit sum $\mathcal{O}_1$ does not vanish when restricted to $V_{32}$, so it is genuinely anti-invariant (i.e. proportional to $f_3e_3$), which means that the weight sum $w_1$ must vanish. So we find just a single equation in degree 6 that must be satisfied in order for $(\bigstanleypoly)_{(6)}$ to be invariant under $\sigma_{56}$, and that is
\[ W_1 := w(\gFourStar) + w(\gFourIso) = 0. \]
This holds for our specific weights given by \eqref{eq:cvalues}, $w(\gFourStar) = 3$, $w(\gFourIso)=-3$, and so $(\bigstanleypoly)_{(6)}$ is invariant under $\sigma_{56}$. Thus $\bigstanleypoly$ is invariant.
\end{proof}

\subsubsection{Boundary Hyperplanes Again}\label{sect:hyperplanesagain}
Let $T$ denote the 3-subsets of $[6]$, which carries the action of $S_6 \times \BZ_2$. Consider the 20 hyperplanes $H_T := x_T - \beta = 0$, which we easily see are permuted by $S_6 \times \BZ_2$. The polytope defined by these 20 hyperplanes contains the 12 points in the orbit of
\[ \tfrac{\beta}{3}\{ 5, -1,-1,-1,-1,-1\} \in V_{51} \]
Each such point lies on $10$ of the hyperplanes (the 10 choices of 3 $T$ that include $+5$, or the 10 choices that include three $+1$'s), and each hyperplane contains 6 such points.

\begin{lemma}
The polynomial $\bigstanleypoly$ has the following properties.
\begin{itemize}
\item $\bigstanleypoly \in \Sym^{10}(V_{51}[\bbeta])^{S_6 \times \BZ_2}$
\item For each triple $T$, consider the degree 10 polynomial $J_T:= \prod_{S:S\cap T \leq 1} H_S$.  Each $J_T$ is invariant under $x_S\to x_{S^c}$, $\beta \to -\beta$. Then, on each hyperplane $H_T=0$, we have a factorization
\[ \bigstanleypoly|_{H_T=0} = J_T|_{H_T=0} .\]
\end{itemize}
\end{lemma}

\begin{proof}
By Theorem~\ref{prop:hyperplanesimp}, on each hook hyperplane $\bh_b^\sA = 0$ we have $\bigstanleypoly \simeq \simpterm(b,\sA)$, a product of $10$ hooks. Under the embedding $\gamma$, the hook $\bh_b^{\bmX_b}$ corresponds to $H_{\gamma(b)}$ and $\bh_b^{\bar\bmX_b}$ to $H_{\gamma^c(b)}$. The boundary datum $\bmX^{\claw(b)}$ is the diagram with hooks flipped on $\claw(b)$, so its $10$ factors are $H_{\gamma(a)}$ for $a\notin\claw(b)$ and $H_{\gamma^c(a)}$ for $a\in\claw(b)$. Setting $T := \gamma(b)$,

For $a\notin\claw(b)$: since $\gamma$ is a graph embedding, Petersen non-adjacency gives $|\gamma(a)\cap\gamma(b)|\neq 2$, and $a\neq b$ gives $\gamma(a)\neq\gamma(b)$, so $|\gamma(a)\cap T|\leq 1$.

For $a\in\claw(b)$: $|\gamma^c(a)\cap T| = |\gamma(a)^c\cap\gamma(b)| = 3 - |\gamma(a)\cap\gamma(b)| \leq 1$, since $a = b$ or $a\sim b$ gives $|\gamma(a)\cap\gamma(b)|\geq 2$.

\end{proof}

\subsubsection{Invariant Rings}
As the $S_5$ representation $V_{32}$ is not generated by reflections, the invariant ring $\BC[V_{32}]^{S_5}$ is not a free polynomial ring\footnote{From computations in Sage, it appears to be the case that
$C[V_{32}]^{S_5} =$$ (C[e_2,e_3,e_4,e_5, e_6] \otimes C[r_6,r_7,r_8,r_9,r_{15}])/ (r_i \cdot r_j = a_{ij}^0(e)+ a_{ij}^k(e) r_k)$, where the $e_i$ are the elementary polynomials pulled back from $V_{51}$.
}. 
Thus any expression for $\bigstanleypoly \in \Sym^{10}(V_{32}[\bbeta])^{S_5\times \BZ_2}$ in terms of  $\BC[V_{32}]^{S_5}$ would likely contain redundancies. However, now we have `untwisted' from $V_{32}$ to $V_{51}$, whose invariant ring is generated by elementary symmetric polynomials
\[ \BC[V_{51}]^{S_6 \times \BZ_2} = \BC[e_2,e_3,e_4,e_5,e_6]^{\BZ_2}. \]
\begin{observation}
In the invariant ring $\BC[V_{51}]^{S_6 \times \BZ_2}$, we find
\begin{align}
\bigstanleypoly &= 42\beta^{10} + 54 e_2 \beta^8 + (34 e_4 + 14 e_2^2) \beta^6 \nonumber \\
&\quad + (42 e_6 + 30 e_3^2 + 4 e_2 e_4 + 2 e_2^3) \beta^4 \\
&\quad + (-8 e_4^2 - 18 e_3 e_5 + 12 e_2 e_6 + 6 e_2 e_3^2 + 2 e_2^2 e_4) \beta^2 \nonumber\\
&\quad + (2 e_5^2 - 8 e_4 e_6 + 2 e_3^2 e_4 - 2 e_2 e_3 e_5 + 2 e_2^2 e_6).\nonumber
\end{align}
\end{observation}
We can also look at the expansion in monomial symmetric functions
\[ \bigstanleypoly = \sum_{\lambda:|\lambda|\leq 10} a_\lambda m_\lambda \beta^{10-|\lambda|}. \]
\begin{lemma}\label{lemma:monomialbound}
The monomial coefficient $a_\lambda$ can be non-zero only if $\lambda_1\leq 3$.
\end{lemma}
\begin{proof}
We use the expression of $F$ in terms of the Johnson graph variables $y_T := x_{T_1}+x_{T_2}+x_{T_3}$ on the anti-symmetric quotient $\johnvertrep$ (where $y_T=-y_{T^c}$). Under the embedding $\psi:\petvertrep \to \johnvertrep$, the polynomial $F$ is written as a degree-10 polynomial in these $y_T$ variables. Each $y_T$ is a \emph{linear} combination of the $x_i$ with coefficients $0$ or $\pm 1$, and each $y_T$ involves exactly $3$ of the $6$ variables $x_i$.

The exponent of $x_i$ in a monomial $\prod_{j=1}^{k} y_{T_j}$ equals the number of $T_j$ containing $i$. Among the $10$ anti-symmetric representatives, element $i$ appears in exactly $\binom{5}{2}=10$ triples, but by $S_6$ symmetry the maximum exponent of any $x_i$ in the support of $F$ is constrained. The key observation is that in the $\ell$-variable formula, each factor $(\ell_b \pm \beta)$ contributes, via the map $\psi$, a linear form in at most $3$ of the $x_i$'s. Since $F$ is $S_6$-invariant and degree $10$, expanding in monomial symmetric functions $m_\lambda(x_1,\ldots,x_6)$ on $\sum x_i=0$, the leading part $\lambda_1$ cannot exceed $3$ because each linear factor $y_T$ distributes its contribution equally among $3$ variables.
\end{proof}

We can average the orbit sum \eqref{eq:s5orbitsum} over the full $\Aut(\johnson)$ to yield a manifestly $\Aut(\johnson)$-invariant Stanley sum, which thus is also a rule,
\begin{equation}\label{eq:s6orbitsum}
(\bigstanleysum)^{\Aut(\johnson)} =\frac{|\Aut(\petersen)|}{|\Aut(\johnson)|} \left( 7\tilde \CO_{\emptyset} -2 \tilde\CO_{\bullet} + 1 \tilde\CO_{\gPathTwo} \right).
\end{equation}
where before the $S_5$ orbits were in reference to a fixed Petersen graph, these $S_6$ orbits are referencing any of the $12 = \frac{|\Aut(\johnson)|}{|\Aut(\petersen)|}$ distinct embeddings of $\petersen \hookrightarrow \johnson$.

\begin{observation}\label{conj:monomialnonnegative}
The monomial coefficients satisfy $a_\lambda \geq 0$.
\end{observation}

In fact, each of the $S_6$-orbit sums $(-1)^{|J|}\tilde\CO_{J}$ has monomial expansions with all non-negative coefficients.

\subsection{Coordinate transformation}

At last, we compose the map $V_{32} \to V_{51}$ to find
\begin{eqnarray}\label{eq:xcoords}
x_0 &=& \tfrac{1}{3}(-z_1-2z_2+z_3-z_4) \\
x_1 &=& \tfrac{1}{3}(2z_1+z_2+z_3+2z_4+3z_5) \\
x_2 &=& \tfrac{1}{3}(-z_1+z_2+z_3+2z_4) \\
x_3 &=& \tfrac{1}{3}(-z_1+z_2-2z_3-z_4) \\
x_4 &=& \tfrac{1}{3}(-z_1-2z_2-2z_3-z_4-3z_5) \\
x_5 &=& \tfrac{1}{3}(2z_1+z_2+z_3-z_4).
\end{eqnarray}
where $z_i = 2w_i + (1+\al)$.

\subsection{Stanley Non-negativity}

To conclude this section, we state a non-negativity condition that generalises a conjecture Stanley \cite[Conj 8.3]{Stanley:1989}. Label the vertices of $\petersen$ by
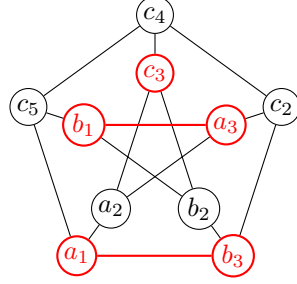
\begin{equation}
\begin{tikzpicture}[scale=1.1, every node/.style={circle, draw, inner sep=1.2pt}]
  \node (o1) at ( 90:1.6) {$c_4$};
  \node (o2) at ( 18:1.6) {$c_2$};
  \node[red, thick] (o3) at (-54:1.6) {$\color{red}{ b_3}$};
  \node[red, thick] (o4) at (-126:1.6) {$\color{red}{ a_1}$};
  \node (o5) at (162:1.6) {$c_5$};
  \node[red, thick] (i1) at ( 90:0.9) {$\color{red}{ c_3}$};
  \node[red, thick] (i2) at ( 18:0.9) {$\color{red}{ a_3}$};
  \node (i3) at (-54:0.9) {$b_2$};
  \node (i4) at (-126:0.9) {$a_2$};
  \node[red, thick] (i5) at (162:0.9) {$\color{red}{ b_1}$};
  \draw (o1)--(o2)--(o3) (o4)--(o5)--(o1);
  \draw[red, thick] (o3)--(o4);
  \draw (o1)--(i1);
  \draw (o2)--(i2);
  \draw (o3)--(i3);
  \draw (o4)--(i4);
  \draw (o5)--(i5);
  \draw (i1)--(i3)--(i5)  (i2)--(i4)--(i1);
  \draw[red, thick] (i5)--(i2);
\end{tikzpicture}
\end{equation}

Consider the following map $\phi$ from $\mathbb{R}[x_1,\ldots,x_5] \to V_\petersen$ given by

\begin{eqnarray*}
\ell_{\color{red}{a_1}}= && 2(1+x_1) + \beta \\
\ell_{\color{red}{a_3}}= && 2(1+x_2)+\beta \\
\ell_{\color{red}{b_1}}= && 2(1+x_3)+\beta \\
\ell_{\color{red}{b_3}}= && 2(1+x_4)+\beta\\
\ell_{\color{red}{c_3}}= && 2(1+x_5)+\beta \\
\ell_{b_2}= && \ell_{b_1}+\ell_{b_3} +\ell_{c_3} \\
\ell_{a_2}= && \ell_{a_1}+\ell_{a_3} +\ell_{c_3} \\
\ell_{c_2}= && -(\ell_{a_1}+\ell_{b_1} +\ell_{c_3}) \\
\ell_{c_5}= && -(\ell_{a_3}+\ell_{b_3} +\ell_{c_3}) \\
\ell_{c_4}= && -(\ell_{a_1}+\ell_{a_3}+\ell_{b_1}+\ell_{b_3} +\ell_{c_3}) 
\end{eqnarray*}
 We can verify that the image of $\phi$ lands in $V_{32}^{(1)}\subset V_\petersen$, i.e. $\ell_b = \sum_{a\sim b}\ell_a$ and hence $\sum_{b} \ell_b = 0$.
These coordinates ${x}$ measure the distance from 5 of the negative hyperplanes $H_{v}^{-} : \ell_{v}-\beta=0$.

\begin{observation}\label{obs:stanleynonneg}
The degree 10 polynomial $(-1)\bigstanleypoly(\phi(x),\beta)|_{\beta\to \al-1}$ has non-negative coefficients as a polynomial in $\{x\},\al$. 
\end{observation}

We can expand out to
\begin{eqnarray*}
\ell_{\color{red}{a_1}}= && 2m_1+1 + (2n_1+1)\al \\
\ell_{\color{red}{a_3}}= && 2m_2+1+(2n_2+1)\al \\
\ell_{\color{red}{b_1}}= && 2m_3+1+\al \\
\ell_{\color{red}{b_3}}= && 2r_2+1+(2n_4+1)\al\\
\ell_{\color{red}{c_3}}= && 2r_1+1+\al 
\end{eqnarray*}

From the observation \eqref{obs:stanleynonneg}, it follows that \emph{all} the Stanley coefficients of the whole family are non-negative polynomials in $\alpha$, which was one of the motivating conjecture of Stanley.

\begin{conjecture}[Stanley1989] The expression
\[ \jackjLR:=(-1)[\bigstanleypoly(\phi(\{m,n,r\}),\al-1)]\] is a non-negative polynomial in $\al$ for all $\{m,n,r\} = (m_1,m_2,m_3,n_1,n_2,n_4,r_1,r_2) \in \BZ_{\geq0}^8$ with $r_2 n_4=0$.
\end{conjecture}

However, of course, this remains a conjecture as we have not proven that our Stanley Sum $\bigstanleysum$ captures the Jack LR coefficients.

In a similar direction, if we write
\begin{eqnarray*}
\ell_{a_1}= && y_1+\beta \\
\ell_{a_3}= && y_2+\beta \\
\ell_{b_1}= && y_3+\beta \\
\ell_{b_3}= && y_4+\beta\\
\ell_{c_3}= && y_5+\beta, 
\end{eqnarray*}
then we find a different form of non-negativity

\begin{conjecture}
 $\rho(\{y\},\beta) = (-1)F(\phi(\{y\}),\beta)$ has non-negative coefficients as a homogeneous polynomial in $y_i,\beta$.  Furthermore, it is degree 7 in $\beta$.
\end{conjecture}

\section{Discussion}\label{section:discussion}

\subsection{Hyperplane simplification}

In theorem \eqref{prop:hyperplanesimp} we proved that, when working in the hook space, the sum of 26 Stanley diagrams that constitute the Jack LR polynomial $\bigstanleysum$ simplify down to a \emph{single} Stanley diagram $\simpterm(b,\sA)$ along each of the 20 hyperplanes $\bh_b^{\sA}  = 0$. That is, $\bigstanleysum \simeq_{b}^{\sA}\simpterm(b,\sA)$.

One such example of this is given by the longest hook of the triple, $b=c_4=(0,0) \in \lambda=\{321\}$, for which we saw in figure \eqref{fig:boundarydatum} is

\ytableausetup{boxsize=1.0em}

\begin{equation}\label{eq:stdfor321c4u}
\bigstanleysum \simeq_{c_4}^{\sU}\simpterm(c_4,\sU) := \frac{
\begin{ytableau}
*(boxU)  \\
*(boxU)  & *(boxU)
\end{ytableau}\,\, \begin{ytableau}
*(boxU) \\
*(boxU)   & *(boxU)
\end{ytableau} }{
\begin{ytableau}
*(boxU)  \\
*(boxL) & *(boxU)  \\
*(boxL) c_4 & *(boxL) & *(boxL)
\end{ytableau}}.
\end{equation}

We propose a conjectural explanation for this hyperplane simplification result, which we will develop further in a subsequent work \cite{Mickler2026b}, however we give an outline here. We begin by observing that for the triple $\{21,21,321\}$, there is an \emph{adjacent} triple of partitions that contains an \emph{equal} hook, namely $\{21,21,2211\}$, and the \emph{upper} hook for $c_4=b=(0,0) \in 2211$. We find,
\[ h_{321}^{\sL}(c_4) = h_{2211}^{\sU}(b) = 3+2\al. \] 

This is because the two partitions $\lambda_1 = 321$ and $\lambda_2=2211$ differ by a simple box move,
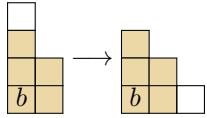
\[ \begin{ytableau}
 \\
*(boxS)  \\
*(boxS)   & *(boxS) \\
*(boxS) b & *(boxS)  \\
\end{ytableau} 
\longrightarrow
\begin{ytableau}
  \none[] \\
*(boxS)  \\
*(boxS)   & *(boxS) \\
*(boxS) b & *(boxS) &  \\
\end{ytableau} 
\]

Moving a single box in one of the three partitions (i.e. Young adjacency) is a simple form of adjacency of two triples of partitions\footnote{There are more complicated versions which we'll explain in the follow up work}.

We computed earlier (example \ref{example:g21212211}) that the Jack LR rule (now with Schur LR coefficient $c=1$) for this adjacent triple is
\begin{equation}
\JackjLR_{21,21}^{2211} = \frac{\begin{ytableau}
*(boxU)  \\
*(boxU)  & *(boxU)
\end{ytableau}\,\, \begin{ytableau}
*(boxU) \\
*(boxU)   & *(boxU)
\end{ytableau} }{
\begin{ytableau}
*(boxU)  \\
*(boxU)  \\
*(boxU)   & *(boxU) \\
*(boxU) b & *(boxU) \\
\end{ytableau}}.
\end{equation}
We observe now that there exists a map $\psi$ between the hooks of these adjacent triples, in the \emph{form} of a hook morphism as in \eqref{def:hookmap}, which is the identity map on $\mu$ and $\nu$, and on the $\lambda_i$ is given by
\begin{equation}\label{eq:hookmap2}
 \psi: \bh_{\lambda_2}^\sA(a) \mapsto \pm \bh_{\lambda_1}^{\pm\sA}(a'),
 \end{equation}
given by,
\ytableausetup{boxsize=1.4em}
\begin{equation}\label{fig:basicpivot}
\psi:\,\,\begin{ytableau}
*(boxU) a \\
*(boxU) c \\
*(boxU) d  & *(boxU) e\\
*(boxU) b & *(boxU) f\\
\none[\lambda_2] 
\end{ytableau}
\,\,\longrightarrow\,\,
\begin{ytableau}
\none  \\
*(boxU) a \\
*(boxL) \bar f  & *(boxU) e \\
*(boxL) \bar b & *(boxL) \bar c & *(boxL) \bar d\\
\none[\lambda_1] 
\end{ytableau}.
\end{equation}
That is, $\psi$ maps between these two Stanley diagrams,
\[ \psi : \JackjLR_{21,21}^{2211} \to \simpterm(c_4,\sU). \]
Although, this map \emph{does not} preserve the relations between these hooks. 
But, rather, there exists a modification $\tilde \psi$ of $\psi$ given by,
\begin{eqnarray*}\label{eq:localmap}
\bh^\sU_{\lambda_2}(a)  \mapsto &   \bh^\sL_{\lambda_1}(a)&\qquad \,\,\,\\ 
\bh^\sU_{\lambda_2}(c)  \mapsto&  -\bh^\sL_{\lambda_1}(c)&{\color{red}{+\bh^\sL_{\lambda_1}(b)}}\\ 
\tilde\psi:\qquad \bh^\sU_{\lambda_2}(d)  \mapsto&  -\bh^\sL_{\lambda_1}(d)&{\color{red}{+\bh^\sL_{\lambda_1}(b)}}\\
\bh^\sU_{\lambda_2}(e)  \mapsto&  \bh^\sU_{\lambda_1}(e)&\qquad \,\,\,\\\
\bh^\sU_{\lambda_2}(f)  \mapsto&  -\bh^\sL_{\lambda_1}(f)&{\color{red}{+\bh^\sL_{\lambda_1}(b)}}\\
\bh^\sU_{\lambda_2}(b)  \mapsto& -\bh^\sL_{\lambda_1}(b)& {\color{red}{+\bh^\sL_{\lambda_1}(b)}}.
\end{eqnarray*}
such that 
\[ [\tilde\psi(\bh_a)] = [\bh_a] , \,\, \forall a\neq b\]
\[ [\tilde\psi(\bh_b)] = 0 \]
Crucially, this map is only of the form \eqref{eq:hookmap2} modulo ${\color{red}{\bh^\sU_{\lambda_1}(b)}}$.

For example $\tilde \psi$ induces an equality of \emph{hook lengths} modulo $3+2\alpha = h^\sU_{\lambda_2}(b)$. e.g. for box $d$, 
\[ \tilde \psi : 2+2\al \mapsto -(1+0\al) \mod 3+2\al. \] 
We now formalize this construction.

\begin{definition}
A \emph{\bf{hook correspondence}} between two adjacent triples $T_2\to T_1$ (via a pivot at $b \in \sigma_1 \cap \sigma_2$) is a linear map $\tilde \psi$ between hook variables such that
\begin{itemize}
\item
$\tilde\psi$ has the form $\bh_{T_2}^\sA(a) \mapsto \pm \bh_{T_1}^{\pm\sA}(a')$ modulo $\bh_{\sigma_2}^\sL(b)$. 
\item
$\tilde\psi: \bh_{\sigma_2}^\sA(b) \mapsto 0$.
\item
$\tilde\psi$ preserves the image of the evaluation map, 
\[ [\tilde\psi(\bh_a)] = [\bh_a] , \,\, \forall a \neq b\]
\end{itemize}
\end{definition}

In other words, when working on the hyperplane $\bh^\sU_{\lambda_2}(b)=0$, we find a bijection of the hooks (with flips) that preserves evaluation. 

\subsubsection{General correspondences}

The above example is of a simple form, which we describe as an action on Stanley diagrams. Given partition $\lambda$, pick two outer corners $a,b \in \addset_\lambda$, let $d = a \vee b$. The two partitions $\lambda+a$, $\lambda+b$ thus differ by a single box move. We will describe a hook correspondence $\psi_d$ on Stanley diagrams from $\lambda+a$ to $\lambda+b$, which acts as identity on the other two partitions ($\mu,\nu$), and as identity on the row-column set of $a$ and $b$.

\ytableausetup{boxsize=0.9em}
\begin{equation}
\psi_d :\begin{ytableau}
*(boxS) &*(boxS)&*(boxS) a\\
 &&*(boxS)     \\
 &&*(boxS)     \\
 &&*(boxS)  &&  \\
 &&*(boxS)   && & \\
*(boxS) &*(boxS)&*(boxU) d &*(boxS)&*(boxS)  &*(boxS)&\none[b] \\
 &&*(boxS)   && &&*(boxS)& \\
 &&*(boxS)   &&& &*(boxS)& \\
\end{ytableau} \to 
\begin{ytableau}
*(boxS) &*(boxS)&\none[a]\\
 &&*(boxS)     \\
 &&*(boxS)     \\
 &&*(boxS)  &&  \\
 &&*(boxS)   && & \\
*(boxS) &*(boxS)&*(boxL) d &*(boxS)&*(boxS)  &*(boxS)&*(boxS)b \\
 &&*(boxS)   && &&*(boxS)& \\
 &&*(boxS)   &&& &*(boxS)& \\
\end{ytableau}
\end{equation}
Note $h^\sU_{\lambda+a}(d) = h^\sL_{\lambda+b}(d)$. The rule for $\psi_d$ is as follows. If a box is to the left of $d$, its hook assignment is swapped with the box on the corresponding other row. Similarly for those boxes beneath $d$, they are exchanged with the box in the other column. For the remaining highlighted boxes something more subtle occurs. If $b$ is in a row with a outer corner (i.e. a place where a box can be added to $\lambda$), then the hook is sent to the box beneath that outer corner, and it is flipped (lower$\leftrightarrow$upper). If $b$ is not in a row with an outer corner, then it moves down one. Next, if a box is to the right of $b$, then if it is a column with an inner corner (i.e. where a box can be removed) then it is swapped with the box to the left of that outer corner (and flipped). If there is no such inner corner, it is moved to the right. All other boxes (i.e. non-highlighted) remain unchanged.
The final outcome is
\begin{equation}
\psi_d :\begin{ytableau}
*(boxS) \scriptstyle  10 &*(boxS) \scriptstyle 11 &*(boxS)\scriptstyle  1\\
 &&*(boxS) \scriptstyle 2    \\
 &&*(boxS) \scriptstyle 3    \\
 &&*(boxS) \scriptstyle 4 &&  \\
 &&*(boxS) \scriptstyle 5  && & \\
*(boxS) \scriptstyle 12 &*(boxS) \scriptstyle 13 &*(boxU) \scriptstyle 6 &*(boxS) \scriptstyle 7 &*(boxS) \scriptstyle 8  &*(boxS) \scriptstyle 9  \\
 &&*(boxS)\scriptstyle 14   && &&*(boxS) \scriptstyle 16& \\
 &&*(boxS)\scriptstyle 15   &&& &*(boxS) \scriptstyle 17& \\
\end{ytableau} \to 
\begin{ytableau}
*(boxS) \scriptstyle 12 &*(boxS) \scriptstyle 13&\none[]\\
 &&*(boxS) \scriptstyle 1    \\
 &&*(boxS) \scriptstyle 2    \\
 &&*(boxS) \scriptstyle \bar{8}  &&   \\
 &&*(boxS) \scriptstyle \bar{9}  &&  & \\
*(boxS)\scriptstyle 10 &*(boxS)\scriptstyle 11&*(boxL)\scriptstyle  \bar{6} &*(boxS) \scriptstyle \bar{3}&*(boxS)\scriptstyle  7  &*(boxS) \scriptstyle \bar{4}&*(boxS)\scriptstyle  \bar{5} \\
 &&*(boxS) \scriptstyle 16  && &&*(boxS)\scriptstyle 14 & \\
 &&*(boxS) \scriptstyle 17  &&& &*(boxS)\scriptstyle 15& \\
\end{ytableau}
\end{equation}
\ytableausetup{boxsize=1.0em}
The above map is always a hook correspondence, as is demonstrated by the following lemma.
\begin{lemma}
For $a,b \in \addset_\mu$, with $a<b$. Let $x := h_{\mu+a}^{\sU}(a \vee b)=h_{\mu+b}^{\sL}(a \vee b)$.
For $c \in \addset_\mu$, we have,
if $a<c<b$
\[  h_{\mu+b}^\sL(b \vee c) + h_{\mu+a}^{\sU}(a\vee c) = x \]
if $c<a$
\[  h_{\mu+b}^\sL(b \vee c) - h_{\mu+a}^{\sL}(a\vee c) = x \]
if $b<c$
\[  -h_{\mu+b}^\sU(b \vee c) + h_{\mu+a}^{\sU}(a\vee c) = x \]

For $c \in \remset_\mu$, we have, if  $a<c<b$ 
\[ h_{\mu+b}^\sU(a\vee c) + h_{\mu+a}^\sL(b\vee c) = x \] 
if $c<a$
\[  h_{\mu+b}^\sU(a\vee c) - h_{\mu+a}^\sU(b\vee c) = x \] 
if $b<c$
\[ -h_{\mu+b}^\sL(a\vee c) + h_{\mu+a}^\sL(b\vee c) = x \] 

\end{lemma}

Including the previous example \eqref{fig:basicpivot}, there are four such hook correspondences of this form (up to basic symmetries) for $\{21,21,321\}$, which are given by

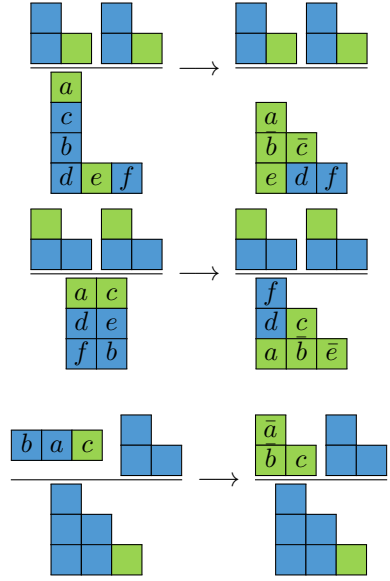
\ytableausetup{boxsize=1.1em}
\[
\frac{\begin{ytableau}
*(boxU) \\
*(boxU) & *(boxU)
\end{ytableau}\,\, \begin{ytableau}
*(boxU) \\
*(boxU) & *(boxU) 
\end{ytableau} 
}{\begin{ytableau}
*(boxU) a \\
*(boxU) c \\
*(boxU) d  & *(boxU) e\\
*(boxU) {b} & *(boxU) f
\end{ytableau}}
\,\longrightarrow\,
\frac{\begin{ytableau}
*(boxU) \\
*(boxU) & *(boxU)
\end{ytableau}\,\, \begin{ytableau}
*(boxU) \\
*(boxU) & *(boxU) 
\end{ytableau} 
}{
\begin{ytableau}
\none  \\
*(boxU) a \\
*(boxL) \bar f  & *(boxU) e \\
*(boxL) \bar b & *(boxL) \bar c & *(boxL) \bar d
\end{ytableau}
}
\]
\[  
\frac{\begin{ytableau}
*(boxU) \\
*(boxU) & *(boxL)
\end{ytableau}\,\, \begin{ytableau}
*(boxU) \\
*(boxU) & *(boxL) 
\end{ytableau} 
}{\begin{ytableau}
*(boxL) a\\
*(boxU) c \\
*(boxU) b \\
*(boxU) d & *(boxL) e & *(boxU) f
\end{ytableau}} 
\,\longrightarrow\,
\frac{\begin{ytableau}
*(boxU) \\
*(boxU) & *(boxL)
\end{ytableau}\,\, \begin{ytableau}
*(boxU) \\
*(boxU) & *(boxL) 
\end{ytableau} 
}{\begin{ytableau}
\none \\
*(boxL) a \\
*(boxL) \bar b  & *(boxL) \bar c\\
*(boxL)  e & *(boxU)  d & *(boxU) f
\end{ytableau}} 
\]
\[ 
\frac{\begin{ytableau}
*(boxL) \\
*(boxU) & *(boxU)
\end{ytableau}\,\, \begin{ytableau}
*(boxL) \\
*(boxU) & *(boxU) 
\end{ytableau} 
}{\begin{ytableau}
*(boxL)a  & *(boxL)c\\
*(boxU)d   & *(boxU) e \\
*(boxU)f  & *(boxU)b  
\end{ytableau}} 
\,\longrightarrow\,
\frac{\begin{ytableau}
*(boxL) \\
*(boxU) & *(boxU)
\end{ytableau}\,\, \begin{ytableau}
*(boxL) \\
*(boxU) & *(boxU) 
\end{ytableau} 
}{\begin{ytableau}
*(boxU)f \\
*(boxU)d & *(boxL) c \\
*(boxL)a & *(boxL) \bar b & *(boxL) \bar e
\end{ytableau}} 
\]
\[
\frac{\begin{ytableau}
\none \\
 *(boxU) b &  *(boxU)a &  *(boxL)c
\end{ytableau}\,\, \,\,\begin{ytableau}
\none \\
*(boxU) \\
*(boxU)   & *(boxU)
\end{ytableau} 
}{\begin{ytableau}
*(boxU) \\
*(boxU) & *(boxU)   \\
*(boxU) & *(boxU) & *(boxL)  
\end{ytableau}}
\,\longrightarrow\,
\frac{\begin{ytableau}
*(boxL)  \bar a \\
*(boxL)  \bar b & *(boxL) c
\end{ytableau}\,\, \begin{ytableau}
*(boxU) \\
*(boxU)   & *(boxU)
\end{ytableau} 
}{\begin{ytableau}
*(boxU) \\
*(boxU) & *(boxU)   \\
*(boxU)   & *(boxU) & *(boxL)  
\end{ytableau}}
\]

On the left hand side, we have $c=1$ diagrams, whose structure conforms with the strong form of the Stanley conjecture, on the right we have the boundary datum $\simpterm(b,\sA)$ for the box $b$ (see figure \eqref{fig:boundarydatum}).

\subsubsection{Extension to window families}
Let $b=(0,0)$, and $\windowfamily'$ be the window family for $\{21,21,2211\}$ defined by \eqref{eq:window21212211}. For each $W' \in \windowfamily'$, we claim there exists a window $W \in \windowfamily$ of $\{21,21,321\}$ and a hook correspondence $\tilde\psi_{W',W} : W' \to W$ relative to $b$.
This data defines a map 
\begin{equation}
 \StDR_{21,21}^{2211}(\windowfamily')/\{ \bh_{c_4}^\sU \} \to \StDR_{21,21}^{321}(\windowfamily)/\{ \bh_{c_4}^\sL \}
\end{equation}

Thus, in summary, we interpret the simplification \eqref{eq:stdfor321c4u} in terms of the hook correspondences $\tilde \psi$ inducing a local equivalence between Jack LR coefficients,
\begin{equation}\label{eq:2121321er}
\tilde \psi : \JackjLR_{21,21;2211} \mod \bh_{2211}^{\sU}(b)\,\,\, \mapsto \,\,\, \JackjLR_{21,21;321}  \mod \bh_{321}^{\sL}(b).
\end{equation}

We conjecture this property holds in general.

\begin{conjecture}[To appear in \cite{Mickler2026b}]\label{conj:adjacency}
Let $(\mu_1\nu_1\lambda_1)$ and $(\mu_2\nu_2\lambda_2)$ be \emph{adjacent} triples of partitions corresponding to a pivot at $b$. There exists a family of hook correspondences $\tilde\psi$ between window families $\windowfamily', \windowfamily$ which induce a map
\begin{equation}
\{\tilde\psi\} : \StDR_{\mu_1\nu_1}^{\lambda_1}(\windowfamily')/\{ \bh_{b}^\sU \} \to \StDR_{\mu_2\nu_2}^{\lambda_2}(\windowfamily)/\{ \bh_{b'}^\sL \}
\end{equation}
under which we have the equivalence
\begin{equation}\label{eq:gmapstogp}
\tilde\psi : \JackjLR_{\mu_1\nu_1;\lambda_1} \mod\bh_{\sigma_1}^{\sA}(b)\,\,\, \mapsto \,\,\, \JackjLR_{\mu_2\nu_2;\lambda_2} \mod \bh_{\sigma_2}^{\sA}(b)
\end{equation}
\end{conjecture}

\subsection{Consequences for LR coefficients}

Conjecture \eqref{conj:adjacency} has a concrete manifestation at the level of regular LR coefficients\footnote{P. Alexandersson has calculated many of these coefficients in \cite{Alexandersson:aa} } of the adjacent pair of triples of the previous example. These are given by
\begin{eqnarray}
g_{21,21;321}(\al)  &=&   6\al^4(1+2\al)(2+\al)(2+11\al+2\al^2), \\
g_{21,21;2211}(\al) &=&  4 \al^4(1 + 2\al)^2(3 + \al)(4 + \al).
\end{eqnarray}

Now, the relation of Jack LR polynomials of \eqref{eq:2121321er} becomes a more direct equivalence of Jack LR functions, 
\begin{equation}\label{eq:stanleymodeq}
g_{21,21;2211}(\al) \equiv \, g_{21,21;321}(\al) \mod h_{321}^{\sL}(c_4).
\end{equation}
In other words the \emph{difference} of these polynomials
\[g_{21,21;321}(\al)-g_{21,21;2211}(\al) = 2 \al^4 (1 + 2 \al) {\color{red} (3 + 2 \al)} (-4 + 6 \al + \al^2) \]
is divisible by $h_{321}^{\sL}(c_4)={\color{red} (3+2\al)}$. 
Indeed, we have
\[ g_{21,21;321}(-3/2)=g_{21,21;2211}(-3/2) = \tfrac{1215}{4}. \]

We conjecture that this divisibility holds in general. Conjecture \ref{conj:adjacency} immediately implies the following statement, which is more directly verifiable.

\begin{conjecture}\label{conj:congruence}
For two triples of partitions $(\mu_1,\nu_1,\lambda_1)$, $(\mu_2,\nu_2,\lambda_2)$ that differ by a single box move (a `pivot') in one of the pairs in the triples, say at $b_1 \in \sigma_1$, $b_2 \in \sigma_2$, so that
\[ h_{\sigma_1}^\sU({b_1}) = h_{\sigma_2}^{\sL}({b_2}),\]
then we have 
\begin{equation}\label{eq:adjmod}
 g_{\mu_1\nu_1;\lambda_1}(\al) \equiv \, g_{\mu_2\nu_2;\lambda_2}(\al) \mod h_{\sigma_1}^\sU({b_1}). 
\end{equation}
\end{conjecture}

This above conjecture is the strongest evidence we have for the approach to understanding Jack LR coefficients developed in this work.

\begin{example}
We consider two pivots for the $c=3$ Jack LR
\begin{eqnarray*}
& g_{{3 2 1}, {2 2 1}; {4 3 2 1 1}} =48 \al^5 (2 + \al)(1 + 2\al)^2 \times &\\
&(120 + 1220\al + 5574\al^2 + 12443\al^3 + 13849\al^4 + 7655\al^5 + 2073\al^6 + 254\al^7 + 12\al^8)&. 
\end{eqnarray*}

First at the box $b=(0,0)$ in $\lambda_1=43211$, for which $h_{\lambda_1}^\sL = {\color{blue}(5+3\al)}$. This pivots to the $c=1$ coefficient
\[g_{{3 2 1}, {2 2 1}; {3 3 2 1 1 1}}=96\al^5(1 + \al)(3 + \al)^2(4 + \al)(1 + 2\al)^4(5 + 2\al)(2 + 3\al). \]
The difference is
\[ 48 \al^5 (1 + 2 \al)^2 {\color{blue}(5+3\al)} (-96 - 
   556 \al - 750 \al^2 + 676 \al^3 + 2155 \al^4 + 
   1596 \al^5 + 501 \al^6 + 70 \al^7 + 4 \al^8) .\]

Secondly we pivot at the box $b=(0,0)$ in $\mu$, for which  $h_\mu^\sL(b)={\color{red} (3 + 2 \al)}$, which pivots to the coefficient
\begin{eqnarray*}
 &g_{{2211}, {221}; {43211}} = 128\al^5(1 + \al)(3 + \al)(4 + \al)(1 + 2\al)(2 + 3\al) \times &\\
&  (12 + 131\al + 321\al^2 + 294\al^3 + 97\al^4 + 9\al^5).&
\end{eqnarray*}

The difference $g_{{3 2 1}, {2 2 1}; {4 3 2 1 1}}-g_{{2211}, {221}; {43211}}$ is
\begin{eqnarray*}
&-16 \al^5 (1 + 2 \al) {\color{red} (3 + 2 \al)} \times& \\
& (-528 - 7360 \al - 26336 \al^2 - 35740 \al^3 - 11003 \al^4 + 16523 \al^5 + 15493 \al^6 + 5025 \al^7 + 690 \al^8 + 36 \al^9).&
\end{eqnarray*}
\end{example}

The equations \eqref{eq:adjmod} give a large family of compatibility relations between all Jack LR coefficients. However, these alone are not sufficient to completely determine them all, although it is the case that this conjecture strongly constrains the Strong form of the Stanley conjecture \eqref{conj:strongstanley2}. Only at the level of Jack LR \emph{polynomials} are these local compatibilities sufficient to determine the polynomials entirely.

\subsection{Extension to Shifted Jacks}
In \cite{AlexanderssonFeray:2019}, the Stanley conjectures were extended to products of \emph{shifted} Jack functions, $J^{\#}_\lambda(\al)$, which are non-homogenous extensions of the Jack functions $J_\lambda(\al)$. The products of shifted Jacks are captured by the shifted Jack LR coefficients\footnote{The notation of \cite{AlexanderssonFeray:2019} is $\tilde d_{\mu\nu}^{\lambda}$ for our $g_{\mu\nu;\lambda}$.} $g_{\mu\nu;\lambda}$, where now $|\lambda|\leq |\mu|+|\nu|$.  
\begin{conjecture}
Conjecture \eqref{conj:congruence} holds for shifted Jack Littlewood-Richardson coefficients.
\end{conjecture}

We have no conceptual explanation for this observation, other than its connection to regular Jack LR coefficients. We have not been able to determine if any windowing properties hold in the shifted case.

\begin{example}
We reference the tables \cite{Alexandersson:aa}, 
\begin{eqnarray*}
&&g_{{3 2 1}, {2 2 2};{4 3 3 1}}=\\
&&\qquad48 \al^4 (2 + \al)^2 (3 + \al)^2 (1 + 2 \al)^2 (1 + 3 \al) (2 + 3 \al)^2 (24 + 171 \al + 284 \al^2 + 116 \al^3),\\
&&g_{{3 2 1}, {2 2 2}; {4 3 2 2}}=\\
&&\qquad288 \al^5 (2 + \al)^3 (3 + \al)^2 (1 + 2 \al)^2 (2 + 3\al) (3 + 4\al)^2 (2 + 11 \al + 2 \al^2).
\end{eqnarray*}
Their difference is 
\[ -48 \al^4 {\color{red} (1 + \al)} (2 + \al)^2 (3 + \al)^2 (1 + 2 \al)^2 (2 + 3 \al) (-48 -294 \al - 157 \al^2 + 480 \al^3 + 492 \al^4 + 192 \al^5)\]
which has a factor of $\color{red} (1+a)$ which is the pivot hook $h_{4 3 3 1}^{L}(b) = h_{4 3 2 2}^{U}(b)$ at the box $b=({3,2})$.
\end{example}

\begin{example}
It was proven \cite{AlexanderssonFeray:2019} that in the shifted case the Jack LR coefficient may include a negative power of $\alpha$. Our conjecture holds even in those cases.
\begin{eqnarray}
 g_{{2 2 1 1}, {2 2 1 1};{3 2 1 1}} &=&  768(1 + \al)^6(3 + \al)^2(4 + \al)^2(1 + 2\al). \\
g_{{2 2 1 1}, {2 2 1 1};{2 2 111}} &=& 384\al^{-1}(1 + \al)^6(4 + \al)^2(5 + \al)(6 + \al)(3 + 2\al)^2.
\end{eqnarray}
With difference
\[ -1152\al^{-1} (1 + \al)^6{\color{red} (2 + \al)} (4 + \al)^2 (45 + 51 \al + 10 \al^2), \]
which contains the hook $h_{{3 2 1 1}}^\sL(b) = h_{{2 2 111}}^\sU(b) = {\color{red}(4+2\alpha)}$ for $b=(0,0)$ (once the factor of two is accounted for).
\end{example}

\subsection{Extension to Macdonald Polynomials}

Consider $P_\mu(q,t)$ the integral Macdonald functions. We define the Macdonald Littlewood-Richardson coefficients $g_{\mu\nu}^{\lambda}(q,t)$, by
\[ J_\mu(q,t) \cdot J_\nu(q,t) = \sum_{\lambda} g_{\mu\nu}^{\lambda}(q,t) J _\lambda(q,t). \]
As before, we set
\[ g_{\mu\nu;\lambda}(q,t) := g_{\mu\nu}^{\lambda}(q,t) \langle J_{\lambda} \rangle^2_{q,t}. \]

The \emph{modified} Macdonald LR coefficients are
\begin{equation}\label{modmlr}
 \tilde g_{\mu\nu;\lambda}(q,t) := t^{n(\mu)+n(\nu) -n(\lambda) }g_{\mu\nu;\lambda}(q,t^{-1})
 \end{equation}

We suspect that the structure of Macdonald Littlewood-Richardson coefficients should exhibit all of the same properties as those for Jacks.

\begin{conjecture}
Let $(\mu_1,\nu_1,\lambda_1)$, $(\mu_2,\nu_2,\lambda_2)$ be two triples of partitions  that differ by a single box move (a `pivot') in one of the pairs in the triples, say at $b_1 \in \sigma_1$, $b_2 \in \sigma_2$, so that the $(q,t)$-hook lengths\footnote{$h_{\mu}^\sL(b)(q,t) = 1-q^{arm_\mu(b)}t^{1+leg_\mu(b)}$.} agree
\[ h_{\sigma_1}^\sU({b_1}) = h_{\sigma_2}^{\sL}({b_2}) = 1-q^y t^x.\]
Let $ \tilde g_{\mu\nu;\lambda}(q,t)$ be the modified Macdonald Littlewood-Richardson coefficient \eqref{modmlr}.
If $x\neq y$, then we have a simple ratio of evaluations at the zero of the common hook $(q=s^x, t=s^{-y})$, 

\begin{equation}\label{eq:macadjmod2}
\lim_{\substack{  {q\to s^x}\\{t\to s^{-y}}}} \left( \frac{\,\,  \tilde g_{\mu_1\nu_1;\lambda_1}(q,t^{-1})}{ \tilde g_{\mu_2\nu_2;\lambda_2}(q,t^{-1})}\right) \in \mathbb{Q}^\times.
\end{equation} 

\end{conjecture}

Note: only with the normalization of $t^{n(\mu)+\ldots}$ coming from \eqref{modmlr} that the right hand side is a constant, a requirement not present in the Jack case \eqref{conj:congruence}. If $x=y$ then there can be a mismatch of the order of vanishing, but we expect that this case can be included with correct normalizations. We also expect that with normalizations the RHS can made to be $1$.

\begin{example}
We have 
\begin{eqnarray*}
&\tilde g_{21,21;321}(q,t^{-1}) =& t^{2} (t - 1)^4(q - 1)^4(qt^2 - 1)(q^2t - 1) \times \\
&&(2q^5t^5 + q^5t^4 + q^4 t^5 - q^5 t^3 + 4q^4t^4 - q^3t^5 - q^4t^3 - q^3t^4 - \\
&&   3q^4t^2 + 4q^3t^3 - 3q^2t^4 - q^4t - q^3t^2 - q^2t^3 -   qt^4 - 3q^3t + \\
&& 4q^2t^2 - 3qt^3 - q^2t - qt^2 - q^2 +  4qt - t^2 + q + t + 2),
\end{eqnarray*}
\[\tilde g_{21,21;2211}(q,t^{-1}) =t^{5} (t + 1)^2(t - 1)^4(q - 1)^4(q^2t - 1)^2(qt^3 - 1)(qt^4 - 1). \]
For the pivot hook $h^\sL_{321}(0,0) = h^\sU_{2211}(0,0) = 1- q^2 t^3$, we find
\begin{eqnarray}
&&  \tilde g_{21,21;321}(s^3,s^{2})=\tilde g_{21,21;2211}(s^3,s^{2}) = \nonumber \\
&& s^{-10}(1-s^5)(1-s^3)(1-s^2)^4(1+s^2)^2(1-s^3)^4(1-s^4)^2.\nonumber
 \end{eqnarray}
\end{example}

\section{Appendix}

\subsection{Root systems}

\newcommand{\fh}{\mathfrak{h}}

Speculatively, we rephrase the hyperplane factorization properties in terms of the root system $A_5$. Let $\fh = V_{5,1}$, the root vectors are $\alpha_{ij}=e_i-e_j \in \fh$.

The automorphism group of the Root system $A_5$ is given by the Weyl group $W(A_5)=S_6$ adjoined the $\BZ_2$ automorphism of the Dynkin diagram ($-w_0$). Thus 

The $\binom{6}{3}$ weights of $\wedge^3$ are given by $\mu_T := \tfrac{1}{2}\left(\sum_{i\in T} e_i - \sum_{i \in T^c} e_i\right)$. We have
\[ \mu_{abi} - \mu_{abj} = \alpha_{ij}, \]
reflecting that $\wedge^3$ is a minuscule representation.
\begin{corollary}
The weight graph of the representation $\wedge^3$ is $J(6,3)$.
\end{corollary}

The local factorization property for $F=\bigstanleypoly$ takes the form
\[F\big|_{\langle \mu, x\rangle = \beta} = W_\mu\big|_{\langle \mu, x\rangle = \beta} \]
where
\[W_\mu(x,\beta) := \prod_{\nu \in \bar{N}(\mu)} \bigl(\langle \nu, x \rangle + \beta\bigr)\]

\begin{lemma}

\begin{itemize}
\item[1)]If $d(T,U)=1$, then the two local factorisations agree on the intersection of the weight hyperplanes.

\item[2)]If $d(T,U)>1$, then the two local factorisations vanish on the intersection of the weight hyperplanes.

\end{itemize}
\end{lemma}
\begin{proof}
For part $(1)$ if $\mu_T \sim \mu_U$, then there exists a root $\alpha$ such that $\mu_T-\mu_U=\alpha$, ie $s_\alpha(\mu_T) = \mu_U$. For $x$ in the intersection of hyperplanes $\mu_T(x)=-\beta=\mu_U(x)$, we find $\alpha(x)=0$ and so $s_\alpha(x) = x$. Thus 
\[\prod_{\mu' \in \bar{N}(\mu)} \bigl(\langle \mu', x \rangle + \beta\bigr)\]
$\bar N(s_\alpha(\mu)) = s_\alpha \bar N(\mu)$

For part (2), we use the fact that for $J(6,3)$ we have $d(T,U)>1 \leftrightarrow U \in \bar N(\mu_T)$. 
\end{proof}

\begin{observation}
For $\rho = (-5,-3,-1,1,3,5)$ which is twice the Weyl vector of $A_5$, we have
\[ F(t\rho, \beta) = 42 \prod_{i=1}^{6} (\beta+ t\rho_i)^{q_i} \]
where $q = (1,2,2,2,2,1)$.
\end{observation}

The image of the positive lattice in Dynkin coordinates, i.e. $\langle \alpha_i,x(w) \rangle$ for $x(w)$ given by \eqref{eq:xcoords}, is

\[ \al_0 = -4(1+\al) -2(w_1+w_2+w_4+w_5) \]
\[ \al_1 = 2(1+\al) + 2(w_1+w_5)\]
\[ \al_2 = 2(1+\al) + 2(w_3+w_4)\]
\[ \al_3 = 2(1+\al) + 2(w_2+w_5) \]
\[ \al_4 = -4(1+\al) - 2(w_1+w_2+w_3+w_5) \]

It is an interesting question whether similar functions exist for other root systems with a choice of minuscule representation. 

\begin{example}
For the $A_3$ root system, the function 
\[ f_{\omega_2} := 6 \beta^5 + 8 e_2 \beta^3 + (2e_2^2 - 8 e_4)\beta, \] 
is $W(A_3)=S_4$ invariant, anti-invariant under $(x,\beta)\to(-w_0x,-\beta)$, and satisfies the hyperplane factorization formula for the weight graph of $V_{\omega_2} = \wedge^2(\BC^4)$, 
\end{example}

Calculations show that a (somewhat trivial) solution exists for $\{A_n, \omega_{\mathrm{std}}\}$ for all $n$, and non-trivially for $\{B_n, \omega_{\mathrm{spin}}\}$ for $n=3,4,5$.

\subsection{Box labels}

The $S_6 \times \BZ_2$ action on the boxes of $21,21,321$ (where $\tilde b_3 = \{b_3,\bar c_1,c_6\}$) is given by the following embedding into $J(6,3)$, and identifying $T = -T^c$.

\begin{equation}
\ytableausetup{boxsize=1.7em}
 \frac{\begin{ytableau}
 145 \\
 135 & 235  
\end{ytableau} \,\,\, \begin{ytableau}
025  \\
 012 & 124 
\end{ytableau}
}{\begin{ytableau}
*(boxB) \\
234 & 013 \\
034 & 045 & *(boxB)  
\end{ytableau}}
\end{equation}

\section{Acknowledgements}
The author would like to thank Per Alexandersson and Arun Ram for helpful conversations during the development of this work. Claude Opus 4.6 was used to assist with the generation of some of the Mathematica code and Latex typesetting for this project.

\bibliographystyle{amsplain}
\bibliography{./final}
\end{document}